\PassOptionsToPackage{dvipsnames}{xcolor}
\documentclass[a4paper,english]{jnsao} %
\KOMAoption{parskip}{half-}
\usepackage[utf8]{inputenc}
\usepackage[T1]{fontenc}
\usepackage{bbm}
\usepackage{float}
\usepackage{algpseudocode}
\usepackage{enumitem}
\usepackage{mathtools}
\usepackage[nameinlink,capitalize]{cleveref}
\usepackage{pgfplots}
\usepackage{pgfplotstable}
\usepackage{booktabs}
\pgfplotsset{compat=newest}

\usepgfplotslibrary{colorbrewer}

\manuscriptcopyright{}
\manuscriptlicense{}
\manuscripteprinttype{arXiv}
\manuscripteprint{}

\usepackage[margin=2.1cm]{geometry}

\newcommand{\N}{\mathbb{N}}
\newcommand{\norm}[1]{\|#1\|}
\newcommand{\abs}[1]{|#1|}
\newcommand{\defeq}{:=}
\DeclareMathOperator{\prox}{prox}
\DeclareMathOperator{\proj}{proj}

\theoremstyle{definition}
\newtheorem{assumption}[theorem]{Assumption}
\crefname{assumption}{Assumption}{Assumptions}
\floatstyle{ruled}
\floatname{algorithm}{Algorithm}
\newfloat{algorithm}{tbp!}{loa}
\crefname{algorithm}{Algorithm}{Algorithms}

\AddToHook{env/proposition/begin}{\crefalias{theorem}{proposition}}
\AddToHook{env/lemma/begin}{\crefalias{theorem}{lemma}}
\AddToHook{env/corollary/begin}{\crefalias{theorem}{corollary}}
\AddToHook{env/definition/begin}{\crefalias{theorem}{definition}}
\AddToHook{env/assumption/begin}{\crefalias{theorem}{assumption}}
\AddToHook{env/remark/begin}{\crefalias{theorem}{remark}}
\AddToHook{env/algorithm/begin}{\crefalias{theorem}{algorithm}}

\algrenewcommand{\algorithmiccomment}[1]{\hfill$\rightsquigarrow$\ {\itshape #1}}

\author{\texorpdfstring{%
        Christian Clason\thanks{Department of Mathematics and Scientific Computing, NAWI Graz, University of Graz, Austria, \email{c.clason@uni-graz.at}, \orcid{0000-0002-9948-8426}}
        \and
        Alberto Dom\'inguez Corella\thanks{Department of Mathematics and Scientific Computing, NAWI Graz, University of Graz, Austria, \email{alberto.dominguez-corella@uni-graz.at}, \orcid{0000-0002-2269-8070}}
        \and
        Jyrki Jauhiainen\thanks{Department of Mathematics and Scientific Computing, NAWI Graz, University of Graz, Austria, \email{jyrki.jauhiainen@uni-graz.at}, \orcid{0000-0001-6711-6997}}
    }{%
        Christian Clason; Alberto Domínguez Corella; Jyrki Jauhiainen%
}}
\shorttitle{Observability as bilevel constraint qualification}
\shortauthor{Clason, Domínguez Corella, Jauhiainen}

\title{Observability as a constraint qualification for a nonsmooth bilevel PDE-constrained problem}

\acknowledgements{%
    This research was funded in whole or in part by the Austrian Science
    Fund (FWF) \href{https://dx.doi.org/10.55776/F100800}{10.55776/F100800}.
}

\begin{document}

\maketitle

\begin{abstract}
    Motivated by experiment design for parameter identification in image reconstruction, we analyze a bilevel control problem involving sparse inversion at the lower level. Qualified optimality conditions in explicit multiplier form are shown to hold under either an observability condition for a coupled system or the Lebesgue nullity of a contact set induced by the lower-level adjoint variable. We prove that these conditions are independent of one another and that they cannot, in general, both be omitted. The approach is based on the characterization of local first-order geometric objects (tangent and normal cones) associated with a set of lifted constraints. Based on the resulting optimality conditions, we formulate a reduced gradient method and illustrate it through numerical experiments.


    \paragraph{2020 Mathematics Subject Classification:}
    49\textsc{k}20 (primary); 49\textsc{j}52, 49\textsc{j}53, 93\textsc{b}07 (secondary).
\end{abstract}

\tableofcontents

\newpage

\section{Introduction, main  contributions and state of the art}

The aim of this paper is to analyze an experiment design problem for parameter identification as a PDE-constrained bilevel optimization problem. To set the stage, consider the prototypical initial value problem:
\begin{equation}\label{eq:intro:pde}
    \left\{
        \begin{aligned}
            \partial_t y-\Delta y+q(t)\,y &= 0,\\
            y(0) &= u(x),
        \end{aligned}
    \right.
\end{equation}
together with suitable boundary conditions. Equation \eqref{eq:intro:pde} is sometimes referred to as the bio-heat equation, where the bilinear term $qy$ models the temperature transport due to blood diffusion; see, e.g., \cite{trucu2009bioheat}. We are primarily interested in equation \eqref{eq:intro:pde} not for its physical interpretation, but because it serves as a tractable model that captures the following structural features:
\begin{itemize}
    \item a smoothing effect induced by the diffusion term $-\Delta y$, which makes initial-data recovery ill-posed;
    \item the coupling of a spatially dependent quantity $u(x)$ with a time-dependent coefficient $q(t)$;
    \item the bilinear zero-order term $q(t)y$, through which the time-dependent coefficient acts on the state.
\end{itemize}
These features also appear in more structured models, such as the
Bloch--Torrey equations underlying magnetic resonance imaging (MRI), in which
diffusion is coupled with lower-order terms describing relaxation and
magnetic-field-induced precession effects \cite{PhysRev.104.563}.

The parameter identification problem considered here is the recovery of the
initial condition \(u\) from noisy measurements of the state \(y(t)\) at
positive times for a fixed choice of the time-dependent coefficient \(q\).
This is a linear inverse problem; however, due to the smoothing effect of the
parabolic equation, the forward mapping \(u\longmapsto y(u)\) is compact in
\(L^2\), with singular values that decay exponentially. This makes the
parameter identification exponentially ill-posed in the sense that any error
in the measurement necessarily leads to an exponential error in the
reconstruction, making any practical solution impossible. One therefore resorts to a (Tikhonov) regularization of the problem, which amounts to solving
\begin{equation}\label{eq:intro:lower}
    \min_{y,u} \frac12 \int_{0}^{T}\!\!\int_{\Omega} |y-y_d|^{2}\,dx\,dt
    +\frac{\lambda}{2}\int_{\Omega}|u|^{2}\,dx
    +\mu\int_{\Omega}|u|\,dx\quad\text{subject to \eqref{eq:intro:pde}.}
\end{equation}
The datum $y_d$ is a noisy measurement of the
state corresponding to the true initial condition $u^\dagger$, i.e.,
$y_d = y(u^\dag) + \eta$, where $\eta$ is normally distributed with mean zero
and variance $\delta^2$. The standard deviation $\delta>0$ represents the noise level. The
parameters $\lambda,\mu>0$ are regularization parameters that have to be chosen suitably in relation to the noise level. The theory and numerical solution of such Tikhonov-regularized inverse problems are well understood;
see, e.g., \cite{Engl,ItoJin}. In particular, the third term is known to promote ``sparsity'' in the reconstruction in the sense that the minimizing parameter $u(q)$ is zero on a set whose Lebesgue measure is controlled by the value of $\mu>0$, which serves to efficiently suppress background noise in the reconstruction.

The regularized inverse problem \eqref{eq:intro:lower} serves here as a model problem whose main feature is the nonsmoothness introduced by the
sparsity-promoting term. Our focus is not the reconstruction problem for a
fixed coefficient \(q\) but rather the choice of \(q\) itself. One  seeks to choose the control parameter $q$ such that the corresponding
lower-level solution \(u(q)\) is as close as possible to the true initial
condition \(u^\dagger\). This is an instance of experiment design; see, e.g.,
\cite{Pukelsheim:2006}. Classical approaches to experiment design are based on sensitivities, i.e.,
spectral quantities associated with the linearized mapping $u\longmapsto y(u)$, but these fail to take nonsmooth regularization into account. The approach we propose is to instead treat this as a bilevel optimization
problem, i.e., to compute $q$ as the solution of the upper-level problem
\begin{equation}\label{eq:intro:upper}
    \min_q \frac12 \int_\Omega |u(q) - u^\dag|^2 \,dx + \frac{\alpha}2  \int_0^T |q|^2\, dt\quad\text{subject to \eqref{eq:intro:lower}}.
\end{equation}
Such nested optimization problems involving differential equations are referred to as \emph{bilevel control
problems}; see, e.g., the survey papers \cite{Dempe:2026,Dempe:2022}.
Note that, in an experiment design setting, the datum $y_d$ in
\eqref{eq:intro:lower} should also depend on $q$. Here, however, we can use
the linearity of the mapping $u\longmapsto y(u)$ to consider instead the
difference $\hat y := y(u) - y(u^\dagger) = y(u-u^\dagger)$, i.e., to replace $y$
by $\hat y$ and $y_d$ by $\eta$ in \eqref{eq:intro:lower}; see
\cref{sec:numerics} below. In any case, this does not materially affect the theoretical analysis of the problem~\eqref{eq:intro:upper}. Since this is the main focus of this work, we use the simpler formulation for the sake of clarity.

Bilevel control problems are in general known to be difficult due to the fact that they require solving a full optimal control problem just to evaluate the effect of one parameter, which leads to a substantial increase in computational complexity. Problem \eqref{eq:intro:upper} poses even greater difficulties due to the nondifferentiability of the lower level, caused by the sparsity-promoting $L^1$-term. Bilevel control problems with nonsmooth lower levels have not been treated so far in the literature, since no known approaches in the field apply or are tractable in this setting. Most approaches in bilevel optimization typically rely on differentiability of the lower-level solution map or the value function, and therefore cannot be applied to problems such as \eqref{eq:intro:upper}.

The main obstruction in \eqref{Bip} is already visible in the lower-level
optimality system. The sparsity term yields the pointwise soft-thresholding
relation $u
=
\operatorname{prox}_{(\mu/\lambda)|\cdot|}
\left(-\frac1\lambda\varphi(0)\right)$,
where \(\varphi\) is the lower-level adjoint variable. Thus, the lower-level optimal
control is determined by the position of \(\varphi(0)\) relative to the
thresholds \(\pm\mu\). In particular, the regions
\(\{|\varphi(0)|<\mu\}\) and \(\{|\varphi(0)|>\mu\}\) correspond,
respectively, to the zero and nonzero regimes of the sparse reconstruction,
while the critical set
\[
    \{|\varphi(0)|=\mu\}
\]
is the switching set at which the soft-thresholding map changes regime. This
set plays a role analogous to a free boundary, or contact set, in
obstacle-type variational inequalities; it marks the interface where the
pointwise optimality system ceases to be locally smooth. A perturbation of the
upper-level coefficient \(q\) changes the state and adjoint equations, and may
therefore also move the active, inactive, and contact regions that determine
the lower-level solution. This is the main reason why standard arguments based on taking the first variation fail to yield
necessary optimality conditions.

This paper is devoted to the analysis of the bilevel problem \eqref{eq:intro:upper}
by means of tools from variational analysis, and in particular from
variational geometry. The main objective is to derive qualified necessary optimality conditions
for local minimizers of \eqref{eq:intro:upper} in multiplier form. Such explicit characterizations do not hold in
general; they require a constraint qualification at the solution, which allows
one to separate the influence of the various constraints on the optimal
solution. 

The main contributions of the manuscript can be described as follows.

First, we identify the graph of the lower-level solution mapping as the
appropriate lifted constraint set for the variational-geometric analysis of
the bilevel problem. This allows us to reformulate the bilevel problem as a
single-level optimization one; see \cref{liftedcons}. We
then establish several algebraic and topological properties of the lifted
constraint set, including nonconvexity, weak sequential closedness, path
connectedness, and nowhere density; see \cref{topc,nonconvexity}.

Second, we analyze the local first-order geometry of the lifted constraint
set by characterizing its tangent and normal cones under suitable
nondegeneracy hypotheses; see
\cref{p1,Boucon,Boucon2,Ctc,prop:ClarkeNormal}. These cone characterizations
lead to qualified first-order necessary optimality conditions for local
minimizers in multiplier form; see \cref{thm:nec-F,thm:nec-FN}. The analysis
relies on two different assumptions.
\begin{itemize}
    \item \emph{(Observability condition).}
        The first assumption is an observability condition for a coupled linearized
        state--adjoint system; see \cref{obsass}. To the best of our knowledge,
        this is the first use of such a condition as a constraint qualification in
        bilevel optimal control. Under this assumption, the polar of the Bouligand
        tangent cone admits an explicit normal-cone representation in terms of
        multiplier variables and pointwise complementarity conditions; see
        \cref{Boucon}. This yields a qualified first-order necessary condition for
        local minimizers; see \cref{thm:nec-F}. We give a
        construction showing that the observability condition cannot be
        dispensed with: If it fails, the normal-cone formula can break down, and
        local minimizers need not admit multipliers of the form described in the
        theorem; see \cref{thmfailobs} and the discussion following
        \cref{thm:nec-FN}.

    \item \emph{(Nullity of the contact set).}
        The second assumption is that the contact set generated by the lower-level
        adjoint variable has zero Lebesgue measure. In this case, the nonsmoothness
        of the soft-thresholding map is inactive in the relevant first-order
        geometry, and the Bouligand and Clarke tangent cones coincide; see \cref{Ctc}. This yields a normal-cone
        representation, and hence a multiplier system, without imposing the
        observability condition; see \cref{Boucon2,prop:ClarkeNormal} and
        \cref{thmindass}. It also gives an alternative qualified first-order
        necessary condition for local minimizers; see \cref{thm:nec-FN}.
\end{itemize}
We prove that these assumptions are independent of one another, in the sense
that neither implies the other; see \cref{thmindass}.

Third, we use the resulting optimality system to formulate a reduced gradient
method for \eqref{eq:intro:upper}. The method computes the lower-level state and
adjoint variables, solves the upper-level adjoint system, and updates the
coefficient by means of the synthetic gradient relation. We illustrate its
behavior numerically and show the effect of the optimized coefficient on the
quality of the reconstruction; see \cref{sec:numerics}.

Let us now discuss the surrounding literature and the state of the art in the context of our results. Classical bilevel optimization is well established in finite dimensions; see, for instance, \cite{Bard1998,Dempe:2002,DempeZemkoho2020}. In optimal control and PDE-constrained optimization, the lower level is often an optimization problem in function spaces, so standard finite-dimensional reformulations have to be used with care. Karush--Kuhn--Tucker (KKT) reformulations lead to single-level formulations with complementarity constraints, known as mathematical programs with complementarity constraints (MPCCs). However, such reformulations are not automatically equivalent to the original bilevel problem \cite{DempeDutta2012,DempeZemkoho2013}. Similarly, calmness assumptions used in value-function or penalized reformulations are restrictive and have to be verified separately \cite{HenrionSurowiec2011,MehlitzMinchenkoZemkoho2021}. Function-space bilevel optimal control has been studied through KKT/MPCC, value-function, and solution-map based approaches \cite{Mehlitz2016,Mehlitz2017,MehlitzWachsmuth2016,Mehlitz:2020}.

A closely related topic is inverse optimal control with PDE constraints. Harder and Wachsmuth derive C-stationarity for an inverse optimal control problem in which parameters of a lower-level elliptic optimal control problem are identified from upper-level data, while later work establishes M-stationarity for a class of MPCCs in Lebesgue spaces and applies these results to inverse optimal control \cite{Harder:2019,HarderWachsmuth2022}. Related value-function and computational approaches are developed in \cite{DempeHarderMehlitzWachsmuth2019,Harder:2023}, and a broader overview is given in \cite{Dempe:2022}.

Among the closest model-level comparisons is the bilevel approach of Castro and De los Reyes for optimal observation placement in variational data assimilation \cite{CastroDeLosReyes2020}. There -- as here -- the lower level reconstructs the initial condition of a parabolic system. The upper level, however, optimizes space-time observation placement rather than a time-dependent coefficient in the PDE. Christof's work on nonsmooth elliptic lower-level problems provides another useful comparison: There the lower-level solution map can still be Fr\'echet differentiable under suitable assumptions \cite{Christof:2020}.
Bilevel optimization also occurs in computational bilevel learning for inverse problems and imaging. Parameter learning for nonsmooth variational reconstruction models was developed in \cite{Kunisch:2013,Schonlieb2013,Schonlieb16,Schonlieb17,Schonlieb17a,Holler:2018}, with later work on learned discretizations of TV and bilevel training of TV-type regularizers \cite{ChambollePock2021,PagliariPapafitsorosRaitaVikelis2022}.

Although observability conditions have, to the best of our knowledge, not been used in the context of constraint qualifications for bilevel control problems, the link between observability and controllability is well-known in classical control theory: For example, for finite-dimensional time-invariant linear systems, observability is equivalent to controllability of the dual system through the Kalman rank condition; see, e.g., \cite{Sontag1998}. The situation is naturally more delicate for infinite-dimensional control systems, but similar duality results between observability and control hold in this setting as well. This field has received significant attention over the last decades, and we only refer to the now classical works \cite{Fattorini1966,Sakawa1975,Dolecki1977,Russell1978,Lions1988,Glowinski1994,Glowinski1995,Tucsnak2009}. Observability results are also crucial for unique continuation \cite{Tataru1996} and inverse problems \cite{Isakov2017,Yamamoto2025} for partial differential equations, where spectral inequalities and Carleman estimates are standard tools.
Specifically, our observability condition is related to classical parabolic observability, unique continuation, coefficient identification, and null controllability\cite{LebeauRobbiano1995,FursikovImanuvilov1996,LeRousseauLebeau2012,ErvedozaZuazua2011,EscauriazaSereginSverak2003,Zhang2012}. However, it is not a standard observability estimate for a single heat equation observed on an interior or boundary subset but an injectivity assumption for a scalar-in-time observation of a coupled state--adjoint sensitivity system. Relevant points of comparison are therefore coupled and cascade parabolic systems with reduced observation or control channels \cite{FernandezCaraGuerrero2006,AmmarKhodjaBenabdallahGonzalezBurgosDeTeresa2011,GonzalezBurgosDeTeresa2010,KavianDeTeresa2010,BenabdallahCristofolGaitanDeTeresa2010}, as well as insensitizing-control and inverse-problem settings involving reduced measurements and identifiability \cite{deTeresaZuazua2009,BenabdallahCristofolGaitanYamamoto2009,Yamamoto2009}.

The variational-geometric part of the analysis is connected to normal cones and stationarity theory for pointwise constraints and complementarity systems in function spaces. Tangent and normal cone formulas for pointwise-defined sets in Lebesgue spaces are studied in \cite{MehlitzWachsmuth2018}. Related MPCC and complementarity-constrained problems in Banach, Lebesgue, and Sobolev spaces are treated in \cite{ClasonDengMehlitzPruefert2020,HarderWachsmuth2018,HarderWachsmuth2022,Wachsmuth2015}. In bilevel imaging, De los Reyes and Villac\'is derive Mordukhovich and Bouligand stationarity conditions for total variation (TV) regularized lower-level problems in a finite-dimensional discrete imaging setting \cite{DelosReyes:2022}. They also prove Lipschitz continuity and directional differentiability of the lower-level solution operator and characterize the Bouligand subdifferential of the solution mapping. They also use a single-level reformulation based on the lower-level optimality system, but their variational analysis is carried out for a primal-dual constraint system involving the graph of a subdifferential mapping. In contrast, our lifted constraint set is the graph of the lower-level solution mapping itself.

Finally, gradient-based approximation techniques for bilevel problems with nonsmooth lower-level structure were developed in \cite{Ochs:2016}. More recent single-step and single-loop methods for bilevel parameter learning are studied in \cite{SV_2024,SuonperaValkonen2025}. These works are mainly computational, whereas the focus here is on function-space first-order geometry and necessary optimality conditions.

The paper is organized as follows. In \cref{sec:statements}, we give a precise formulation of the bilevel control problem and state the main results of this work. \Cref{sec:prems} collects the notation, definitions, and preliminary results used in the analysis. The proofs of the main results are given in \cref{sec:proofs}. Finally, in \cref{sec:numerics}, we present the reduced gradient method and illustrate it with numerical experiments.

\section{The bilevel problem and overview of results}\label{sec:statements}

In this section, we introduce the bilevel problem we are interested in and state the main results.

\subsection{The bilevel optimization problem}
Let $\Omega \subseteq \mathbb{R}^d$ be a bounded convex domain and $T>0$ a time horizon. As usual, we denote $Q:=(0,T)\times\Omega$ and $\Sigma:=(0,T)\times\partial\Omega$. We fix parameters $\lambda,\mu>0$.
For a time-dependent coefficient ${q} \in L^2(0,T)$, the \emph{lower-level problem} is given by
\begin{equation}\label{cost}
    \min_{u\in L^{2}(\Omega)}
    \left\{
        \frac12 \int_{0}^{T}\!\!\int_{\Omega} |y-y_d|^{2}\,dx\,dt
        +\frac{\lambda}{2}\int_{\Omega}|u|^{2}\,dx
        +\mu\int_{\Omega}|u|\,dx
    \right\}
\end{equation}
\begin{equation}\label{system}
    \text{subject to}\quad
    \left\{
        \begin{aligned}
            \partial_t y-\Delta y+q(t)\,y &= 0 & \text{in }Q,\\
            y &= 0 & \text{on }\Sigma,\\
            y(0) &= u & \text{in }\Omega .
        \end{aligned}
    \right.
\end{equation}
For each time-dependent coefficient, this problem consists in reconstructing an initial condition from a given tracking datum
\(y_d \in L^2(Q)\). It can be seen by standard arguments, that
for each coefficient the lower-level problem \eqref{cost}--\eqref{system}
has a unique solution. Given a coefficient \(q\in L^2(0,T)\), we denote by
\(u_q\in L^2(\Omega)\) the unique minimizer of problem
\eqref{cost}--\eqref{system}.

Given $u_d \in L^2(\Omega)$ and $\alpha > 0$, the \emph{bilevel} problem we consider is defined as
\begin{equation}\label{Bip}
    \min_{{q}\in L^2(0,T)} \left\{ \frac12
        \int_\Omega |u_{q}(x) - u_d(x)|^2 \, dx
        + \frac{\alpha}{2}
    \int_0^T |{q}(t)|^2 \, dt \right\}.
\end{equation}
This problem consists in finding a parameter \(q\in L^2(0,T)\)
such that the initial condition \(u_q\in L^2(\Omega)\) reconstructed at the
lower level from the observation \(y_d\) using \(q\) is as close as possible to
a reference datum \(u_d\), while simultaneously controlling the size
of \(q\) through a regularization term.

\subsection{The set of lifted constraints}\label{liftedcons}
Since the lower-level problem is convex, its solution can be fully characterized by its
first-order optimality system. Let us  introduce now a set encoding this system. Define \emph{the set of lifted constraints} by
\[
    \mathcal{C}:=
    \left\{(q,u)\in L^2(0,T)\times L^2(\Omega)\ \middle|\ \exists\, y,\varphi\in W(0,T):
        \begin{cases}
            \begin{cases}
                \partial_t y-\Delta y+q(t)\,y=0 & \text{in }Q,\\
                y=0 & \text{on }\Sigma,\\
                y(0)=u & \text{in }\Omega,
            \end{cases}\\[0.6em]
            \begin{cases}
                -\partial_t\varphi-\Delta\varphi+q(t)\,\varphi=y-y_d & \text{in }Q,\\
                \varphi=0 & \text{on }\Sigma,\\
                \varphi(T)=0 & \text{in }\Omega,
            \end{cases}\\[0.6em]
            u=\prox_{(\mu/\lambda)|\cdot|}\!\Bigl(-\frac1\lambda\,\varphi(0)\Bigr)
            \quad \text{a.e.\ in }\Omega.
        \end{cases}
    \right\}.
\]
It can be seen that $(q,u)\in L^2(0,T)\times L^2(\Omega)$ belongs to $\mathcal C$
if and only if $u$ is the unique minimizer of the lower level problem
\eqref{cost}--\eqref{system} corresponding to $q$; in particular, it follows that $\mathcal C$ is nonempty.

Consider now the functional $\Phi: L^2(0,T)\times L^2(\Omega)\to\mathbb{R}\cup\{+\infty\}$ given by
\[
    \Phi({q},u)
    := \frac12\int_\Omega |u(x)-u_d(x)|^2\,dx
    +\frac{\alpha}{2}\int_0^T |{q}(t)|^2\,dt
    + \iota_{\mathcal C}({q},u),
\]
where $\iota_{\mathcal C}:L^2(0,T)\times L^2(\Omega)\to\mathbb{R}\cup\{+\infty\}$
denotes the indicator function of $\mathcal C$, defined by
\[
    \iota_{\mathcal C}(q,u):=
    \begin{cases}
        0 & \text{if}\,\,(q,u)\in\mathcal C,\\[2mm]
        +\infty & \text{if}\,\,(q,u)\notin\mathcal C.
    \end{cases}
\]
By construction, the set $\mathcal C$ is the graph of the solution mapping $q\longmapsto u_q$. It also follows that if $\bar q\in L^2(0,T)$ is a local minimizer of  the bilevel problem \eqref{cost}--\eqref{Bip}, then $(\bar q,u_{\bar q})$ is a local minimizer of $\Phi$.

\subsection{Geometric properties of the constraints}

We begin recording some of the basic topological features of the set of lifted constraints.
\begin{proposition}\label{topc}
    The set $\mathcal C$ is weakly sequentially closed, path connected and nowhere dense.
\end{proposition}

A natural follow-up question is whether the lifted constraint set is convex. The following result shows that this can happen only when the lower-level problem is trivial, namely when the tracking datum in the lower-level problem is zero almost everywhere.
\begin{proposition}\label{nonconvexity}
    The following statements are equivalent.
    \begin{enumerate}[label=(\roman*),nosep]
        \item The set $\{y_d\neq 0\}$ has zero Lebesgue measure.
        \item The set $\mathcal C$ is convex.
    \end{enumerate}
\end{proposition}
Save for the trivial case $y_d\equiv 0$, the set of lifted constraints is nonconvex; consequently, in the subsequent analysis, we will rely on
its local first-order geometry.
We begin characterizing the Bouligand tangent cone.
\begin{proposition}\label{p1}
    Let $(q,u)\in\mathcal C$ be given, and denote by $(y,\varphi)\in W(0,T)\times W(0,T)$
    its associated state--adjoint pair. The Bouligand tangent cone
    $T_{\mathcal C}(q,u)$ consists of all pairs $(h,v)\in L^2(0,T)\times L^2(\Omega)$
    for which there exists a pair $(z,\psi)\in W(0,T)\times W(0,T)$ satisfying
    \begin{equation}\label{Btclin}
        \left\{
            \begin{gathered}
                \left\{
                    \begin{aligned}
                        \partial_t z - \Delta z + q(t)\,z + h(t)\,y &= 0
                        && \text{in }Q,\\
                        z &= 0
                        && \text{on }\Sigma,\\
                        z(0) &= v
                        && \text{in }\Omega,
                    \end{aligned}
                \right.
                \qquad
                \left\{
                    \begin{aligned}
                        -\partial_t \psi - \Delta \psi + q(t)\,\psi + h(t)\,\varphi &= z
                        && \text{in }Q,\\
                        \psi &= 0
                        && \text{on }\Sigma,\\
                        \psi(T) &= 0
                        && \text{in }\Omega,
                    \end{aligned}
                \right.
                \\[1em]
                v=
                \begin{cases}
                    0
                    & \text{if }\,\, |\varphi(0)|<\mu,\\[4pt]
                    -\dfrac{1}{\lambda}\min\{\psi(0),0\}
                    & \text{if }\,\, \varphi(0)=-\mu,\\[4pt]
                    -\dfrac{1}{\lambda}\max\{\psi(0),0\}
                    & \text{if }\,\, \varphi(0)=\mu,\\[4pt]
                    -\dfrac{1}{\lambda}\,\psi(0)
                    & \text{if }\,\, |\varphi(0)|>\mu,
                \end{cases}
                \qquad \text{a.e. in }\Omega.
            \end{gathered}
        \right.
    \end{equation}
\end{proposition}
This proposition shows that irregularities can stem from the contact set generated by the adjoint at time zero. The following corollary states that if this set has zero Lebesgue measure, then no such degeneracies occur, and the Bouligand and Clarke tangent cones coincide.
\begin{corollary}\label{Ctc}
    Let $(q,u)\in\mathcal C$ be given and denote by $\varphi\in W(0,T)$ its associated adjoint. If the set $\{|\varphi(0)|=\mu\}$ has zero Lebesgue measure, then the Clarke tangent cone coincides with the Bouligand one at $(q,u)$, in symbols, $\hat T_{\mathcal C}(q,u) = T_{\mathcal C}(q,u)$.
\end{corollary}

The characterization of normal cones consists of transforming polar
inequalities into explicit representations of the dual variable acting on the
relevant components. In order to do this for the set of lifted constraints, we
make an assumption that excludes the possibility that nontrivial perturbations
of the initial datum have a vanishing induced time-dependent coefficient
signal.

\begin{assumption}\label{obsass}
    Let $(q,u)\in\mathcal C$ be given, and denote by $(y,\varphi)\in W(0,T)\times W(0,T)$ its associated state--adjoint pair. Consider the operator $\mathcal O:L^2(\Omega)\to L^2(0,T)$ given by
    \[
        \mathcal O(\pi):=\int_{\Omega}\big[y(\cdot,x)\eta(\cdot,x) + \varphi(\cdot,x)\xi(\cdot,x)\big]\, dx,
    \]
    where $\xi,\eta\in W(0,T)$ satisfy
    \[
        \left\{
            \begin{aligned}
                \partial_t\xi-\Delta\xi+q(t)\,\xi &= 0 && \text{in }Q,\\
                \xi &= 0 && \text{on }\Sigma,\\
                \xi(0) &= \pi && \text{in }\Omega,
            \end{aligned}
        \right.
        \qquad
        \left\{
            \begin{aligned}
                -\partial_t\eta-\Delta\eta+q(t)\,\eta &= \xi && \text{in }Q,\\
                \eta &= 0 && \text{on }\Sigma,\\
                \eta(T) &= 0 && \text{in }\Omega.
            \end{aligned}
        \right.
    \]
    We assume that $\mathcal O:L^2(\Omega)\to L^2(0,T)$ is injective, i.e.,
    \[
        \mathcal O(\pi)=0\quad \implies\quad \pi=0\qquad \forall\,\pi\in L^2(\Omega).
    \]
\end{assumption}
This assumption can be regarded as an observability-type condition. It will play the role of a constraint qualification in the next proposition. Essentially, the condition prevents
loss of information in the coefficient component and allows the polar
inequality to be rewritten as an explicit multiplier system.
\begin{proposition}\label{Boucon}
    Let $(q,u)\in\mathcal C$ be such that \cref{obsass} holds. Denote by $(y,\varphi)\in W(0,T)\times W(0,T)$
    the associated state--adjoint pair.  Then a pair $(h^*,v^*)\in L^2(0,T)\times L^2(\Omega)$ belongs to $N_{\mathcal C}(q,u)$ if and only if there exist
    $\xi,\eta\in W(0,T)$ and $\pi,\rho\in L^2(\Omega)$ satisfying
    \begin{equation}\label{sls}
        \left\{
            \begin{gathered}
                \left\{
                    \begin{aligned}
                        \partial_t\xi-\Delta\xi+q(t)\,\xi &= 0
                        && \text{in }Q,\\
                        \xi &= 0
                        && \text{on }\Sigma,\\
                        \xi(0) &= \pi
                        && \text{in }\Omega,
                    \end{aligned}
                \right.
                \qquad
                \left\{
                    \begin{aligned}
                        -\partial_t\eta-\Delta\eta+q(t)\,\eta &= \xi
                        && \text{in }Q,\\
                        \eta &= 0
                        && \text{on }\Sigma,\\
                        \eta(T) &= 0
                        && \text{in }\Omega,
                    \end{aligned}
                \right.
                \\[1em]
                v^*=\rho-\lambda\,\pi-\eta(0)
                \qquad\text{a.e. in }\Omega,
                \\[0.5em]
                h^*=\int_\Omega \bigl(y\,\eta+\varphi\,\xi\bigr)\,dx
                \qquad\text{a.e. in }[0,T].
            \end{gathered}
        \right.
    \end{equation}
    and the complementarity conditions
    \begin{equation}
        \left\{\,\,\,\,
            \begin{gathered}
                \pi=0 \ \text{a.e.\ on } \{|\varphi(0)|<\mu\},\qquad
                \pi\ge 0 \ \text{a.e.\ on } \{\varphi(0)=-\mu\},\qquad
                \pi\le 0 \ \text{a.e.\ on } \{\varphi(0)=\mu\},
                \\[0.5em]
                \rho=0 \ \text{a.e.\ on } \{|\varphi(0)|>\mu\},\qquad
                \rho\le 0 \ \text{a.e.\ on } \{\varphi(0)=-\mu\},\qquad
                \rho\ge 0 \ \text{a.e.\ on } \{\varphi(0)=\mu\}.
            \end{gathered}
        \right.
    \end{equation}
\end{proposition}
The observability assumption in the previous proposition might seem artificial, but it cannot be directly dispensed with without further assumptions.
\begin{theorem}\label{thmfailobs}
    If \cref{obsass} is not fulfilled, then the claim of \cref{Boucon} need not hold.
\end{theorem}
If the contact set induced by the adjoint variable has zero Lebesgue measure,
then any admissible perturbation supported there vanishes almost everywhere.
Consequently, no nontrivial \emph{invisible} directions can occur, and the polar
condition does not require the observability assumption introduced above.

\begin{proposition}\label{Boucon2}
    Let $(q,u)\in\mathcal C$, and denote by $(y,\varphi)\in W(0,T)\times W(0,T)$
    its associated state--adjoint pair. Suppose that the set $\{|\varphi(0)|=\mu\}$ has zero Lebesgue measure. Then a pair $(h^*,v^*)\in L^2(0,T)\times L^2(\Omega)$ belongs to $N_{\mathcal C}(q,u)$ if and only if there exist
    $\xi,\eta\in W(0,T)$ and $\pi,\rho\in L^2(\Omega)$ satisfying \eqref{sls} such that
    \[
        \pi=0\,\,\,\, \text{a.e.}\,\,\text{on}\,\,\,\,\{|\varphi(0)|<\mu\}\quad \text{and}\quad \rho=0
        \,\,\,\, \text{a.e.}\,\, \text{on} \,\,\,\,\{|\varphi(0)|>\mu\}.
    \]
\end{proposition}
Combining the previous proposition with \cref{Ctc}, we obtain the following consequence.
\begin{corollary}\label{prop:ClarkeNormal}
    Let $(q,u)\in\mathcal C$ be given and denote by $\varphi\in W(0,T)$ its associated adjoint.   Suppose that the set $\{|\varphi(0)|=\mu\}$ has zero Lebesgue measure. Then the Clarke normal cone of $\mathcal C$ at $(q,u)$ coincides with the polar of the Bouligand tangent cone of $\mathcal C$ at $(q,u)$, in symbols, $\hat N_{\mathcal C}(q,u) = N_{\mathcal C}(q,u)$.
\end{corollary}
A natural question proceeding from the previous propositions is whether their hypotheses are related; namely, whether the nullity of the contact set induced by the adjoint variable implies the observability condition above, or vice versa. It turns out they are independent, in the sense that neither implies the other.
\begin{theorem}\label{thmindass}
    Let $(q,u)\in\mathcal C$ be given, and let $\varphi\in W(0,T)$
    denote its associated adjoint. The following statements hold.
    \begin{enumerate}[label=(\roman*),nosep]
        \item \Cref{obsass} does not imply that the set
            $\{|\varphi(0)|=\mu\}$ has zero Lebesgue measure.
        \item The condition that the set
            $\{|\varphi(0)|=\mu\}$ has zero Lebesgue measure does not imply
            \Cref{obsass}.
    \end{enumerate}
\end{theorem}

\subsection{A first-order necessary condition}
We now employ the geometric characterization obtained in the previous subsection to derive a first-order necessary optimality condition for the bilevel problem \eqref{cost}--\eqref{Bip}. The argument relies on transferring local optimality from the reduced bilevel formulation to the lifted formulation and then applying \cref{Boucon}.
\begin{theorem}\label{thm:nec-F}
    Let $\bar q\in L^2(0,T)$ be a local minimizer of the bilevel problem \eqref{cost}--\eqref{Bip}, and set
    $\bar u:=u_{\bar q}$. Let the pair $(\bar y,\bar\varphi)\in W(0,T)\times W(0,T)$ satisfy
    \begin{equation}\label{opt:lower}
        \left\{
            \begin{gathered}
                \left\{
                    \begin{aligned}
                        \partial_t \bar y - \Delta \bar y + \bar q(t)\,\bar y &= 0
                        && \text{in }Q,\\
                        \bar y &= 0
                        && \text{on }\Sigma,\\
                        \bar y(0) &= \bar u
                        && \text{in }\Omega,
                    \end{aligned}
                \right.
                \qquad
                \left\{
                    \begin{aligned}
                        -\partial_t \bar\varphi - \Delta \bar\varphi + \bar q(t)\,\bar\varphi
                        &= \bar y - y_d
                        && \text{in }Q,\\
                        \bar\varphi &= 0
                        && \text{on }\Sigma,\\
                        \bar\varphi(T) &= 0
                        && \text{in }\Omega,
                    \end{aligned}
                \right.
                \\[1em]
                \bar u
                =
                \operatorname{prox}_{(\mu/\lambda)|\cdot|}
                \!\left(-\frac1\lambda \bar\varphi(0)\right)
                \qquad \text{a.e.\ in }\Omega.
            \end{gathered}
        \right.
    \end{equation}
    Suppose that $(\bar q,\bar u)$ satisfies \cref{obsass}. Then there exist $\xi,\eta\in W(0,T)$ and $\pi,\rho\in L^2(\Omega)$ satisfying
    \begin{equation}\label{opt:upper}
        \left\{
            \begin{gathered}
                \left\{
                    \begin{aligned}
                        \partial_t\xi - \Delta\xi + \bar q(t)\,\xi &= 0
                        && \text{in }Q,\\
                        \xi &= 0
                        && \text{on }\Sigma,\\
                        \xi(0) &= \pi
                        && \text{in }\Omega,
                    \end{aligned}
                \right.
                \qquad
                \left\{
                    \begin{aligned}
                        -\partial_t\eta - \Delta\eta + \bar q(t)\,\eta &= \xi
                        && \text{in }Q,\\
                        \eta &= 0
                        && \text{on }\Sigma,\\
                        \eta(T) &= 0
                        && \text{in }\Omega,
                    \end{aligned}
                \right.
                \\[1em]
                \bar u - u_d+\rho-\lambda\pi-\eta(0)=0
                \qquad \text{a.e. in }\Omega,
                \\[1em]
                \alpha\,\bar q(t)
                +
                \int_\Omega\bigl(\bar y(t)\,\eta(t)+\bar\varphi(t)\,\xi(t)\bigr)\,dx
                = 0
                \qquad \text{for a.e.\ }t\in(0,T),
            \end{gathered}
        \right.
    \end{equation}
    and the complementarity conditions
    \begin{equation}\label{comcon}
        \left\{\,\,\,\,
            \begin{gathered}
                \pi=0 \ \text{a.e.\ on } \{|\bar\varphi(0)|<\mu\},\qquad
                \pi\ge 0 \ \text{a.e.\ on } \{\bar\varphi(0)=-\mu\},\qquad
                \pi\le 0 \ \text{a.e.\ on } \{\bar\varphi(0)=\mu\},
                \\[0.5em]
                \rho=0 \ \text{a.e.\ on } \{|\bar\varphi(0)|>\mu\},\qquad
                \rho\le 0 \ \text{a.e.\ on } \{\bar\varphi(0)=-\mu\},\qquad
                \rho\ge 0 \ \text{a.e.\ on } \{\bar\varphi(0)=\mu\}.
            \end{gathered}
        \right.
    \end{equation}
\end{theorem}
Analogously to the previous subsection, using \cref{Boucon2}, we can replace the observability assumption in the previous theorem by requiring the Lebesgue nullity of the contact set generated by the adjoint variable to obtain the same stationarity condition.

\begin{theorem}\label{thm:nec-FN}
    Let $\bar q\in L^2(0,T)$ be a local minimizer of the bilevel problem \eqref{cost}--\eqref{Bip}, and set
    $\bar u:=u_{\bar q}$. Let the pair  $(\bar y,\bar\varphi)\in W(0,T)\times W(0,T)$ satisfy \eqref{opt:lower}. Suppose that the set $\{|\bar\varphi(0)|=\mu\}$ has zero Lebesgue measure. Then there exist $\xi,\eta\in W(0,T)$ and $\pi,\rho\in L^2(\Omega)$ satisfying \eqref{opt:upper} such that
    \[
        \pi=0\,\,\,\, \text{a.e.}\,\,\text{on}\,\,\,\,\{|\bar\varphi(0)|<\mu\}\quad \text{and}\quad \rho=0
        \,\,\,\, \text{a.e.}\,\, \text{on} \,\,\,\,\{|\bar\varphi(0)|>\mu\}.
    \]
\end{theorem}
We mention that the construction in the proof of \cref{thmfailobs} can be used straightforwardly to construct a bilevel problem for which both the nullity of the contact set and the observability condition (\cref{obsass}) fail at a local minimizer, and for which there are no multipliers satisfying \eqref{comcon}.

\section{Notation and preliminaries}\label{sec:prems}

In this section, we briefly collect the essential definitions and standard results needed in the following.

\subsection{Function spaces}
We use the Gelfand triple
\[
    H_0^1(\Omega)\hookrightarrow L^2(\Omega)\equiv L^2(\Omega)^*\hookrightarrow H^{-1}(\Omega),
\]
where both embeddings are dense and continuous.  The duality
pairing between $H^{-1}(\Omega)$ and $H_0^1(\Omega)$ is denoted by
$\langle\cdot,\cdot\rangle$, and when there is no risk of confusion, also the $L^2(\Omega)$
inner product.

We consider the space $W(0,T):=\bigl\{y\in L^2(0,T;H_0^1(\Omega)):\ \partial_t y\in L^2(0,T;H^{-1}(\Omega))\bigr\}$
endowed with its usual norm so that we have the continuous embedding $W(0,T)\hookrightarrow C([0,T];L^2(\Omega))$.
In particular, every $y\in W(0,T)$ admits traces
$y(0),y(T)\in L^2(\Omega)$, and the conditions of the type $y(0)=u$ and $\varphi(T)=0$
appearing below are understood in this trace sense.

If $y,\psi\in W(0,T)$, then the map
$t\longmapsto \langle y(t),\psi(t)\rangle$ is absolutely continuous on $[0,T]$ and
\[
    \frac{d}{dt}\langle y(t),\psi(t)\rangle
    =
    \langle \partial_t y(t),\psi(t)\rangle
    +
    \langle \partial_t \psi(t),y(t)\rangle
    \quad\text{for a.e.\ }t\in(0,T).
\]
Integrating in time yields the formula
\begin{equation}\label{eq:ibp-W}
    \int_0^T \langle \partial_t y(t),\psi(t)\rangle\,dt
    +\int_0^T \langle \partial_t \psi(t),y(t)\rangle\,dt
    =
    \langle y(T),\psi(T)\rangle-\langle y(0),\psi(0)\rangle .
\end{equation}
Finally, since $W(0,T)\hookrightarrow L^\infty(0,T;L^2(\Omega))$, for every
$q\in L^2(0,T)$ and $y\in W(0,T)$ we have
\[
    q\,y\in L^2(Q)\quad \text{and}
    \quad
    \|q\,y\|_{L^2(Q)}
    \le
    \|q\|_{L^2(0,T)}\,\|y\|_{L^\infty(0,T;L^2(\Omega))}.
\]
In particular, the lower-order term $q(t)y(t)$ defines an element of
$L^2(0,T;H^{-1}(\Omega))$ through the canonical embedding
$L^2(\Omega)\hookrightarrow H^{-1}(\Omega)$.

\subsection{Mild and weak solutions}
Recall that we write $Q:=(0,T)\times\Omega$ and $\Sigma:=(0,T)\times\partial\Omega$.
Let $-\Delta:H_{0}^1(\Omega)\cap H^2(\Omega)\to L^2(\Omega)$ denote the Laplacian with homogeneous Dirichlet boundary condition, i.e., the nonnegative self-adjoint operator associated with the bilinear form
$a:H_0^1(\Omega)\times H_0^1(\Omega)\to\mathbb R$ given by
$a(y,v):=\int_\Omega \nabla y\cdot \nabla v\,dx$.  We denote by $(e^{t\Delta})_{t\ge0}$ the strongly continuous
contraction semigroup on $L^2(\Omega)$ generated by $\Delta$.
Thus, if $u=\sum_{k=1}^\infty u_k\varphi_k$ in $L^2(\Omega)$, where \((\lambda_k,\varphi_k)_{k\ge1}\) are the Dirichlet eigenpairs of \(-\Delta\), then
\[
    e^{t\Delta}u:=\sum_{k=1}^\infty e^{-\lambda_k t}u_k\varphi_k.
\]
We refer to $(e^{t\Delta})_{t\ge0}$ as the Dirichlet heat semigroup. This semigroup extends uniquely to a strongly
continuous semigroup on \(H^{-1}(\Omega)\), and we denote this extension by the
same symbol.

Let $q\in L^2(0,T)$. Since $q\in L^1(0,T)$ as well, the scalar integrating factor
$m_q:[0,T]\times[0,T]\to\mathbb R$ defined by
$m_q(t,s):=\exp\!\bigl(-\int_s^t q(r)\,dr\bigr)$ for $0\le s\le t\le T$ is well-defined.
If $u\in L^2(\Omega)$ and $f\in L^2(0,T;H^{-1}(\Omega))$, we call
$y\in C([0,T];L^2(\Omega))$ a \emph{mild solution} of
\begin{align}\label{samplepde}
    \left\{
        \begin{aligned}
            \partial_t y - \Delta y + q(t)y &= f && \text{in } Q,\\
            y &= 0 && \text{on } \Sigma,\\
            y(0) &= u && \text{in } \Omega.
        \end{aligned}
    \right.
\end{align}
if, for every $t\in[0,T]$, one has $y(t)=m_q(t,0)e^{t\Delta}u+\int_0^t m_q(t,s)\,e^{(t-s)\Delta}f(s)\,ds$, where the integral is understood as a Bochner integral in $H^{-1}(\Omega)$.
This is simply the Duhamel formula after removal of the zero-order term by
an integrating factor.

For the weak formulation, observe first that if $y\in W(0,T)$, then
$y\in L^\infty(0,T;L^2(\Omega))$ by the continuous embedding
$W(0,T)\hookrightarrow C([0,T];L^2(\Omega))$. Hence, for $q\in L^2(0,T)$,
the product $q\,y$ belongs to $L^2(Q)$, and therefore also to
$L^2(0,T;H^{-1}(\Omega))$. Given $q\in L^2(0,T)$, $u\in L^2(\Omega)$, and
$f\in L^2(0,T;H^{-1}(\Omega))$, by a \emph{weak solution} of \eqref{samplepde}
we mean a function $y\in W(0,T)$ with $y(0)=u$ such that for a.e.\ $t\in(0,T)$
and every $v\in H_0^1(\Omega)$,
\begin{equation}\label{eq:weak-parabolic}
    \langle \partial_t y(t),v\rangle
    +\int_\Omega \nabla y(t)\cdot \nabla v\,dx
    +\int_\Omega q(t)\,y(t)\,v\,dx
    =\langle f(t),v\rangle .
\end{equation}
In our setting weak and mild solutions coincide.
By the standard theory of linear nonautonomous parabolic equations associated
with time-dependent bilinear forms, it follows that for every
$q\in L^2(0,T)$, $u\in L^2(\Omega)$, and
$f\in L^2(0,T;H^{-1}(\Omega))$, there exists a unique weak solution
$y\in W(0,T)$ of \eqref{eq:weak-parabolic}. Moreover, if
$\|q\|_{L^2(0,T)}\le M$, then there exists a constant
$C=C(M,T,\Omega)>0$ such that
\[
    \|y\|_{L^\infty(0,T;L^2(\Omega))}
    +\|y\|_{L^2(0,T;H_0^1(\Omega))}
    +\|\partial_t y\|_{L^2(0,T;H^{-1}(\Omega))}
    \le
    C\Bigl(\|u\|_{L^2(\Omega)}+\|f\|_{L^2(0,T;H^{-1}(\Omega))}\Bigr).
\]
In particular, for each $q\in L^2(0,T)$ and $u\in L^2(\Omega)$, the state equation
\eqref{system} has a unique weak solution $y\in W(0,T)$. We also consider the backward problem
\[
    \left\{
        \begin{aligned}
            -\partial_t\varphi-\Delta\varphi+q(t)\varphi &= g && \text{in }Q,\\
            \varphi &= 0 && \text{on }\Sigma,\\
            \varphi(T) &= \varphi_T && \text{in }\Omega,
        \end{aligned}
    \right.
\]
where $\varphi_T\in L^2(\Omega)$ and $g\in L^2(0,T;H^{-1}(\Omega))$.
By a \emph{weak solution} of this problem we mean a function $\varphi\in W(0,T)$
with terminal trace $\varphi(T)=\varphi_T$ such that for a.e.\ $t\in(0,T)$ and
every $v\in H_0^1(\Omega)$,
\[
    -\langle \partial_t \varphi(t),v\rangle
    +\int_\Omega \nabla \varphi(t)\cdot \nabla v\,dx
    +\int_\Omega q(t)\,\varphi(t)\,v\,dx
    =\langle g(t),v\rangle
\]
The backward equation also has a unique weak solution in $W(0,T)$,
together with the analogous a priori estimate. In particular, the adjoint equation in the definition of the set of constraints $\mathcal C$ has a
unique weak solution $\varphi\in W(0,T)$ for every right-hand side
$y-y_d\in L^2(Q)\subseteq L^2(0,T;H^{-1}(\Omega))$. %
\subsection{The soft-thresholding map and its Nemytskii operator}
Fix $\tau>0$. We denote by $\prox_{\tau|\cdot|}:\mathbb R\to\mathbb R$ the
scalar proximal mapping given by
\[
    \prox_{\tau|\cdot|}(w)
    :=\operatorname*{argmin}_{r\in\mathbb R}
    \left\{
        \frac12|r-w|^2+\tau|r|
    \right\}.
\]
Since $\Omega$ is bounded, the embedding $L^2(\Omega)\hookrightarrow L^1(\Omega)$
is continuous, and therefore the functional
\[
    L^2(\Omega)\ni u\longmapsto \tau\|u\|_{L^1(\Omega)}
    =\tau\int_\Omega |u(x)|\,dx\in\mathbb R
\]
is proper, convex, and continuous on $L^2(\Omega)$. Its proximal map $\prox_{\tau\|\cdot\|_{L^1(\Omega)}}:L^2(\Omega)\to L^2(\Omega)$ is thus well-defined and coincides with the Nemytskii operator induced by the scalar proximal map, i.e.,
\[
    \bigl(\prox_{\tau\|\cdot\|_{L^1(\Omega)}}(w)\bigr)(x)
    =
    \prox_{\tau|\cdot|}(w(x))
    \qquad\text{for a.e.\ }x\in\Omega .
\]
The scalar map $\prox_{\tau|\cdot|}$ is nondecreasing and piecewise affine with
slopes in $\{0,1\}$, hence it is Lipschitz, and consequently the corresponding Nemytskii $\prox_{\tau\|\cdot\|_{L^1(\Omega)}}$ is $1$-Lipschitz on $L^2(\Omega)$.
\subsection{Some variational geometry}
Throughout this subsection $X$ denotes a Banach space.  Let $C\subseteq X$ and $\theta\in X$. The \emph{Bouligand tangent cone} to $C$ at $\theta$ is defined by
\[
    T_C(\theta)
    :=
    \Bigl\{d\in X:\ \exists\,t_n\searrow 0,\ \exists\,d_n\to d\ \text{in }X
    \ \text{such that }\ \theta+t_n d_n\in C\ \ \forall n\in\mathbb N\Bigr\}.
\]
It follows from its definition that, $T_C(\theta)$ is closed. The \emph{Clarke tangent cone} to $C$
at $\theta$ is defined by
\[
    \hat T_C(\theta)
    :=
    \Bigl\{d\in X:\ \forall\,\theta_n\to \theta\ \text{in }C,\ \forall\,t_n\searrow 0,\
    \exists\,d_n\to d\ \text{in }X\ \text{with }\theta_n+t_n d_n\in C\Bigr\}.
\]
It follows from its definition that, $\hat T_C(\theta)$ is closed and convex. Notice that $\hat T_C(\theta)\subseteq T_C(\theta)$. The \emph{polar cone} of a nonempty set  $K\subseteq X$ is defined by
\[
    K^\circ:=\{x^*\in X^*:\ \langle x^*,d\rangle\le 0\ \ \forall d\in K\}.
\]
The \emph{normal cone} of $C$ at $\theta$ is given by $N_C(\theta):=T_C(\theta)^\circ$, likewise the \emph{Clarke normal cone} is the polar of the Clarke tangent
cone, i.e., $\hat N_C(\theta):=\hat T_C(\theta)^\circ$.

\section{Proofs}\label{sec:proofs}
This section contains the proofs of all statements in \cref{sec:statements}.

\subsection{Proof of Proposition \ref{topc}}
The proof is carried out in essentially two main steps. We first establish that the lower-level
solution mapping $L^2(0,T)\ni q \longmapsto u_q\in L^2(\Omega)$ is weak-to-strong sequentially continuous. This implies, in particular, that the graph of
this mapping, i.e., the lifted constraint set $\mathcal C$, is weakly sequentially closed. In the second step, we prove the remaining topological properties of
$\mathcal C$, namely path connectedness and nowhere density.

\subsubsection{On the continuity properties of the solution mapping}
As discussed above, we first prove a preliminary lemma showing that the
lower-level solution mapping is sequentially continuous from the weak topology of
$L^2(0,T)$ to the strong topology of $L^2(\Omega)$.
\begin{lemma}\label{conom}
    Let  $\mathcal S:L^2(0,T)\to L^2(\Omega)$ be given by $\mathcal S(q):=u_q$. Then $\mathcal S$ is weak-to-strong sequentially continuous.
\end{lemma}
\begin{proof}
    For $q\in L^2(0,T)$ and $v\in L^2(\Omega)$, let $y_{q,v}\in W(0,T)$ denote the
    unique weak solution of \eqref{system}. Given $q\in L^2(0,T)$, consider the
    lower-level functional $\mathcal J_q:L^2(\Omega)\to\mathbb R$ given by
    \[
        \mathcal J_q(v)
        :=\frac12\int_0^T\|y_{q,v}(t)-y_d(t)\|_{L^2(\Omega)}^2\,dt
        +\frac\lambda2\|v\|_{L^2(\Omega)}^2
        +\mu\|v\|_{L^1(\Omega)}.
    \]
    Observe that each $\mathcal J_q$ is $\lambda$--strongly convex, and has $u_q$ as
    its unique minimizer.

    Let $(q_n)_{n\in\mathbb N}\subseteq L^2(0,T)$ and $q\in L^2(0,T)$ be such that
    $q_n\rightharpoonup q$ in $L^2(0,T)$. Set $u_n:=u_{q_n}$ and $u:=u_q$. We prove below
    that $u_n\to u$ in $L^2(\Omega)$.

    By optimality of each $u_n$, $\frac\lambda2\|u_n\|_{L^2(\Omega)}^2
    \le \mathcal J_{q_n}(u_n)
    \le \mathcal J_{q_n}(0)
    =\frac12\int_0^T\|y_d(t)\|_{L^2(\Omega)}^2\,dt$,
    so $(u_n)_{n\in\mathbb N}$ is bounded in $L^2(\Omega)$. We can extract a subsequence (not relabeled) such that
    $u_n\rightharpoonup \bar u$ in $L^2(\Omega)$ for some $\bar u\in L^2(\Omega)$.
    Since $q_n\rightharpoonup q$ in $L^2(0,T)$ and $u_n\rightharpoonup \bar u$ in
    $L^2(\Omega)$, it must be that  $y_{q_n,u_n}\to y_{q,\bar u}$ strongly in  $L^2(Q)$. In particular, $\frac12\int_0^T\|y_{q_n,u_n}(t)-y_d(t)\|_{L^2(\Omega)}^2\,dt
    \to
    \frac12\int_0^T\|y_{q,\bar u}(t)-y_d(t)\|_{L^2(\Omega)}^2\,dt$.
    Additionally, $v\mapsto \|v\|_{L^2(\Omega)}^2$ and
    $v\mapsto \|v\|_{L^1(\Omega)}$ are convex and continuous on $L^2(\Omega)$, hence also weakly lower
    semicontinuous. Combining these facts, we get  $\mathcal J_q(\bar u)\le \liminf_{n\to\infty}\mathcal J_{q_n}(u_n)$. Since $y_{q_n,u}\to y_{q,u}$ strongly in $L^2(Q)$,
    $\mathcal J_{q_n}(u)\to \mathcal J_q(u)$.
    By optimality of $u_n$ we have $\mathcal J_{q_n}(u_n)\le \mathcal J_{q_n}(u)$ for
    all $n\in\mathbb N$, hence taking limes superior gives $\limsup_{n\to\infty}\mathcal J_{q_n}(u_n)\le \mathcal J_q(u)$.
    Combining this with the previous limes inferior estimate yields
    $\mathcal J_q(\bar u)\le \mathcal J_q(u)$. Since $\mathcal J_q$ has $u=u_q$ as unique
    minimizer, it follows that $\bar u=u$. Thus every weakly convergent
    subsequence of $(u_n)_{n\in\mathbb N}$ converges to $u$, and therefore $u_n\rightharpoonup u$ weakly in $L^2(\Omega)$.

    Let us now upgrade the convergence of $(u_n)_{n\in\mathbb N}$ to $u$ from weak to strong.
    For each $n\in\mathbb N$, $\mathcal J_{q_n}$ is $\lambda$--strongly convex on $L^2(\Omega)$,
    and hence satisfies the global error bound
    \[
        \mathcal J_{q_n}(v)\ge \mathcal J_{q_n}(u_n)+\frac\lambda2\|v-u_n\|_{L^2(\Omega)}^2\quad \forall v\in L^2(\Omega).
    \]
    Moreover; from the fact that $\mathcal J_{q_n}(u)\to \mathcal J_q(u)$ and
    $\mathcal J_{q_n}(u_n)\to \mathcal J_q(u)$, we infer that
    \[
        \|u-u_n\|_{L^2(\Omega)}^2\le \frac2\lambda\Big(\mathcal J_{q_n}(u) - \mathcal J_{q_n}(u_n)\Big)\longrightarrow0\quad\text{as}\quad n\longrightarrow+\infty.
    \]
    Since every sequence $q_n\rightharpoonup q$ in $L^2(0,T)$ satisfies
    $\mathcal S(q_n)=u_{q_n}\to u_q=\mathcal S(q)$ in $L^2(\Omega)$, the mapping
    $\mathcal S:L^2(0,T)\to L^2(\Omega)$ is weak-to-strong sequentially continuous.
\end{proof}

\subsubsection{Weak sequential closedness}
Note that $\mathcal C=\operatorname{gph}\mathcal S$, where $\mathcal S:L^2(0,T)\to L^2(\Omega)$ is the solution mapping. Let $(q_n,u_n)\subseteq\mathcal C$ satisfy $(q_n,u_n)\rightharpoonup(q,u)$ in $L^2(0,T)\times L^2(\Omega)$.  Since $u_n=\mathcal S(q_n)$ and $\mathcal S$ is weak-to-strong
sequentially continuous, we have
\[
    u_n=\mathcal S(q_n)\to \mathcal S(q)
    \,\,\,\text{strongly in }L^2(\Omega).
\]
In particular, $u_n\rightharpoonup \mathcal S(q)$ weakly in $L^2(\Omega)$.
By uniqueness of weak limits, $u=\mathcal S(q)$. Hence $\mathcal C=\operatorname{gph}\mathcal S$ is weakly sequentially closed.

\subsubsection{Path connectedness}
Note that $\mathcal C$ is the graph of the solution mapping
$\mathcal S:L^2(0,T)\to L^2(\Omega)$, that is, $\mathcal C=\{(q,\mathcal S(q)):\ q\in L^2(0,T)\}$.
Let $(q_0,u_0),(q_1,u_1)\in\mathcal C$ be arbitrary. Then $u_0=\mathcal S(q_0)$ and
$u_1=\mathcal S(q_1)$. Define a path $\gamma:[0,1]\to L^2(0,T)\times L^2(\Omega)$
by
\[
    \gamma(t):=\bigl(q(t),\mathcal S(q(t))\bigr),
    \qquad
    q(t):=(1-t)q_0+tq_1,
    \quad t\in[0,1].
\]
Since norm convergence implies weak convergence, \cref{conom} shows that $\mathcal S$ is norm-to-norm sequentially continuous.
As both spaces are metric, $\mathcal S$ is norm-to-norm continuous; hence
$\gamma$ is continuous. Moreover, $\gamma(t)\in\mathcal C$ for every
$t\in[0,1]$, and $\gamma(0)=(q_0,u_0)$, $\gamma(1)=(q_1,u_1)$. Thus $\mathcal C$
is path connected.

\subsubsection{Nowhere denseness}
We show first that $\mathcal C$ has empty interior in $L^2(0,T)\times L^2(\Omega)$.
Fix $(q,u)\in\mathcal C$. Then $u=u_q=\mathcal S(q)$, and by uniqueness of the
lower-level solution we have
\[
    \{\,v\in L^2(\Omega):\ (q,v)\in\mathcal C\,\}=\{u_q\}.
\]
Let $r>0$ be arbitrary. Choose $w\in L^2(\Omega)\setminus \{0\}$ and set
$\varepsilon:=\frac{r}{2\|w\|_{L^2(\Omega)}}$. Then
\[
    \|(q,u_q+\varepsilon w)-(q,u_q)\|_{L^2(0,T)\times L^2(\Omega)}
    =\varepsilon\|w\|_{L^2(\Omega)}=\frac r2<r,
\]
so $(q,u_q+\varepsilon w)\in B_r(q,u_q)$. However, $(q,u_q+\varepsilon w)\notin\mathcal C$; otherwise $u_q+\varepsilon w=u_q$, yielding $w=0$, a contradiction. Therefore no
ball around $(q,u_q)$ is contained in $\mathcal C$, so $\operatorname{int}\mathcal C=\varnothing$.
Since $\mathcal C$ is weakly sequentially closed, it is in particular norm-closed; it follows then that
$\overline{\mathcal C}=\mathcal C$ and hence $\operatorname{int}\overline{\mathcal C}=\varnothing$.
Thus $\mathcal C$ is nowhere dense.

\subsection{Proof of Proposition \ref{nonconvexity}}
We prove the two implications separately. The first one is an easy consequence of the structure of the
lower-level optimality system. For the converse, we assume that the lifted
constraint set is convex and show that this forces the tracking datum to vanish
almost everywhere. The argument uses carefully chosen coefficients together
with uniqueness of the associated lower-level solutions.
\subsubsection{\texorpdfstring{Implication $(i)\Rightarrow (ii)$}{Implication (i) => (ii)}}
Assume that $\{y_d\neq 0\}$ has zero Lebesgue measure. Then $y_d=0$ in
$L^2(Q)$. Fix $q\in L^2(0,T)$. The objective functional of the lower level problem then becomes
\[
    \mathcal J_q(u):
    =\frac12\int_0^T\|y_{q,u}(t)\|_{L^2(\Omega)}^2\,dt
    +\frac\lambda2\|u\|_{L^2(\Omega)}^2
    +\mu\|u\|_{L^1(\Omega)}\;\ge\;0
    \quad\text{for all }u\in L^2(\Omega).
\]
Since $y_{q,0}\equiv 0$, we have $\mathcal J_q(0)=0$, hence $u=0$ is a global
minimizer. By $\lambda$--strong convexity of $\mathcal J_q$, this minimizer is
unique, so $u_q=0$.
Consequently, for every $q\in L^2(0,T)$ we have $(q,0)\in\mathcal C$, and for
every $(q,u)\in\mathcal C$ we have $u=u_q=0$. Therefore $\mathcal C = L^2(0,T)\times\{0\}$, which is an affine (hence convex) subset of $L^2(0,T)\times L^2(\Omega)$.

\subsubsection{\texorpdfstring{Implication $(ii)\Rightarrow (i)$}{Implication (ii) => (i)}}
Suppose, towards a contradiction, that \(\mathcal C\) is convex while
\(y_d\not\equiv0\) in \(L^2(Q)\). The proof is divided into three steps.

\paragraph{Step 1 (convexity of the set implies constancy of some solutions).}
For each constant $c\in\mathbb R$, write ${q}_c\equiv c$, and let $u_c$ be the unique $u$ satisfying $({q}_c,u)\in\mathcal C$.
Fix $c_0,c_1\in\mathbb R$ and $t\in[0,1]$. Since $({q}_{c_0},u_{c_0}),({q}_{c_1},u_{c_1})\in\mathcal C$ and the set $\mathcal C$ is assumed to be  convex, we have
\[
    \bigl((1-t){q}_{c_0}+t{q}_{c_1},\ (1-t)u_{c_0}+tu_{c_1}\bigr)\in\mathcal C.
\]
By definition of $\mathcal C$, $u_{(1-t)c_0+tc_1}=(1-t)u_{c_0}+tu_{c_1}$.
Therefore the graph $\{(c,u_c):c\in\mathbb R\}\subseteq\mathbb R\times L^2(\Omega)$ is convex, and the mapping $c\longmapsto u_c$ is affine.
\smallbreak\noindent
The objective functional of the lower level problem
\[
    \mathcal J_c(u)=\frac12\int_0^T \|y^c_u(t)-y_d(t)\|_{L^2(\Omega)}^2\,dt+\frac\lambda2\|u\|_{L^2(\Omega)}^2+\mu\|u\|_{L^1(\Omega)}
\]
satisfies $\mathcal J_c(u_c)\le \mathcal J_c(0)=\tfrac12\int_0^T\|y_d(t)\|_{L^2(\Omega)}^2\,dt$, so
\[
    \|u_c\|_{L^2(\Omega)}^2\le\lambda^{-1}\int_0^T\|y_d(t)\|_{L^2(\Omega)}^2\,dt
    \quad\forall\,c\in\mathbb R.
\]
We see that $\{u_c\}_{c\in\mathbb R}\subseteq L^2(\Omega)$ is bounded.
The affine mapping $c\longmapsto u_c$ has the form $u_c=u_0+c\,w$ for some $u_0,w\in L^2(\Omega)$. If $w\neq 0$, then $\|u_c\|_{L^2(\Omega)}\to+\infty$ as $|c|\to+\infty$, contradicting boundedness;  whence $w=0$. We conclude there exists $\bar u\in L^2(\Omega)$ such that $u_c\equiv \bar u$ for all $c\in\mathbb R$.

\paragraph{Step 2 (constant solutions must be zero).}
The solution of the state equation
\[
    \begin{cases}
        \partial_t y-\Delta y+c\,y=0 & \text{in }Q\\
        y=0 & \text{on }\Sigma\\
        y(0)=\bar u & \text{in }\Omega
    \end{cases}
\]
associated with $u_c$ can be written as $y_{c,\bar u}(t)=e^{-ct}e^{t\Delta}\bar u$, where $(e^{t\Delta})_{t\in[0,+\infty)}$ denotes the heat semigroup.
The solution of the corresponding adjoint equation
\[
    \begin{cases}
        -\partial_t\varphi-\Delta\varphi+c\,\varphi=y_{c,\bar u}-y_d & \text{in }Q\\
        \varphi=0 & \text{on }\Sigma\\
        \varphi(T)=0 & \text{in }\Omega
    \end{cases}
\]
can be written, using the variation of parameters formula, as
\[
    \varphi_c(0)=\int_0^T e^{-cs}e^{s\Delta}\bigl(y_{c,\bar u}(s)-y_d(s)\bigr)\,ds
    =\int_0^T e^{-2cs}e^{2s\Delta}\bar u\,ds-\int_0^T e^{-cs}e^{s\Delta}y_d(s)\,ds.
\]
The heat semigroup is $L^2$–contractive, hence
\[
    \Big\|\int_0^T e^{-2cs}e^{2s\Delta}\bar u\,ds\Big\|_{L^2}
    \le \int_0^T e^{-2cs}\,\|e^{2s\Delta}\bar u\|_{L^2}\,ds
    \le \int_0^T e^{-2cs}\,ds\,\|\bar u\|_{L^2}
    =\frac{1-e^{-2cT}}{2c}\,\|\bar u\|_{L^2}\xrightarrow[c\to\infty]{}0,
\]
and, by Cauchy–Schwarz in time,
\[
    \Big\|\int_0^T e^{-cs}e^{s\Delta}y_d(s)\,ds\Big\|_{L^2}
    \le \int_0^T e^{-cs}\,\|y_d(s)\|_{L^2}\,ds
    \le \|y_d\|_{L^2(0,T;L^2)}\Big(\int_0^T e^{-2cs}\,ds\Big)^{1/2}
    \xrightarrow[c\to\infty]{}0.
\]
Thus $\varphi_c(0)\to0$ in $L^2(\Omega)$ as $c\to+\infty$.
Since $\Omega$ is bounded, the functional $\|\cdot\|_{L^1}:L^2(\Omega)\to \mathbb R$ is convex and continuous,
so the proximal map $\operatorname{prox}_{(\mu/\lambda)\|\cdot\|_{L^1}}:L^2(\Omega)\to L^2(\Omega)$ is firmly nonexpansive, and in particular, $1$–Lipschitz. Then the feasibility conditions give
\[
    \bigl\|\,\bar u-\operatorname{prox}_{(\mu/\lambda)\|\cdot\|_{L^1}}(0)\bigr\|_{L^2} = \bigl\|\,\operatorname{prox}_{(\mu/\lambda)\|\cdot\|_{L^1}}\big(-\lambda^{-1}\varphi_c(0)\big)-\operatorname{prox}_{(\mu/\lambda)\|\cdot\|_{L^1}}(0)\bigr\|_{L^2}
    \le \tfrac{1}{\lambda}\,\|\varphi_c(0)\|_{L^2}\xrightarrow[c\to+\infty]{}0,
\]
hence $\bar u=\operatorname{prox}_{(\mu/\lambda)\|\cdot\|_{L^1}}(0)=0$.
Consequently, by Step~1, $u_c\equiv0$.

\paragraph{Step 3 (a contradiction to feasibility).}
With $\bar u=0$, the feasibility condition becomes
\[
    \operatorname{prox}_{(\mu/\lambda)\|\cdot\|_{L^1}}\!\bigl(-\tfrac1\lambda \varphi_c(0)\bigr)=0.
\]
so $|\varphi_c(0)|\le \mu$ a.e. in $\Omega$ for all $c\in\mathbb R$. We conclude that  $\|\varphi_c(0)\|_{L^2}\le \mu\,|\Omega|^{1/2}$ for all $c\in\mathbb R$.

Expand $y_d(s)=\sum_{k\in\mathbb N} y_k(s)e_k$ in the Dirichlet eigenbasis
$\{-\Delta e_k=\lambda_k e_k\}$.
Using $e^{s\Delta}e_k=e^{-\lambda_k s}e_k$, we obtain
\[
    \varphi_c(0)
    =-\int_0^T e^{-cs}e^{s\Delta}y_d(s)\,ds
    =-\sum_{k\ge1}\Big(\int_0^T e^{(-c-\lambda_k)s}y_k(s)\,ds\Big)e_k
    \quad\forall c\in\mathbb R.
\]
Since $y_d\not\equiv0$ in $L^2(0,T;L^2)$ by assumption, we can choose $k_0\in\mathbb N$ such that
$y_{k_0}\not\equiv0$ in $L^2(0,T)$.
Define $g(s)=e^{-\lambda_{k_0}s}y_{k_0}(s)\in L^1(0,T)\setminus\{0\}$.
An elementary application of Varadhan’s integral lemma \cite[Theorem 4.3.1]{DZ_1998}, together with the Laplace lower and upper bounds \cite[Lemmas 4.3.4 and 4.3.6]{DZ_1998}, yields
\[
    \limsup_{M\to\infty}\Big|\int_0^T e^{Ms}g(s)\,ds\Big|=+\infty.
\]
This gives a sequence $c_n\to -\infty$ such that
\[
    \Big|\int_0^T e^{(-c_n-\lambda_{k_0})s}y_{k_0}(s)\,ds\Big|
    \xrightarrow[n\to+\infty]{}+\infty.
\]
Because $(e_k)_{k\in\mathbb N}$ is an orthonormal basis of $L^2(\Omega)$,
\[
    \|\varphi_{c_n}(0)\|_{L^2}^2
    =\sum_{k\ge1}\Big|\int_0^T e^{(-c_n-\lambda_k)s}y_k(s)\,ds\Big|^2
    \ \ge\
    \Big|\int_0^T e^{(-c_n-\lambda_{k_0})s}y_{k_0}(s)\,ds\Big|^2
    \to +\infty,
\]
which contradicts the uniform bound $\sup_{c\in\mathbb R}\|\varphi_c(0)\|_{L^2}\le \mu\,|\Omega|^{1/2}$.
Therefore $\mathcal C$ cannot be convex.

\subsection{Proof of Proposition \ref{p1}}
We prove the result by establishing the two inclusions. We first record some
preliminary lemmas on the linearized state--adjoint system and the directional
behavior of the soft-thresholding map. We then pass from Bouligand tangent
directions to the stated linearized system and, conversely, show that every
direction satisfying this system is realized by admissible perturbations in the
lifted constraint set.

\subsubsection{Some preliminary lemmas}
We begin with a simple calculation, whose proof is straightforward and which will be used throughout.
\begin{lemma}\label{al0}
    Set $\tau:=\mu/\lambda$. Let $f:\mathbb R\to\mathbb R$ be given by
    \[
        f(s):=\operatorname{prox}_{\tau|\cdot|}(s)
        =\operatorname{sign}(s)\,\max\{|s|-\tau,0\}.
    \]
    Its directional derivative at $r\in\mathbb R$ in direction $d\in\mathbb R$ is given by
    \begin{align}\label{ddb}
        f'(r;d)=
        \begin{cases}
            0 & \text{if }|r|<\tau\\
            \max\{d,0\} &\text{if } r=\tau\\
            \min\{d,0\} &\text{if } r=-\tau\\
            d &\text{if } |r|>\tau.
        \end{cases}
    \end{align}
\end{lemma}

We now record a couple properties associated to equations \eqref{Btclin}.
\begin{lemma}\label{al1}
    Given $q\in L^2(0,T)$, consider the linear operator  $A_q: L^2(\Omega)\to L^2(\Omega)$ given by $A_qv:=\psi_{q}(0)$, where
    \begin{equation}\label{Btclinn}
        \left\{
            \begin{aligned}
                \partial_t z_q-\Delta z_q+q(t)\,z_q&=0 &&\text{in }Q,\\
                z_q&=0 &&\text{on }\Sigma,\\
                z_q(0)&=v &&\text{in }\Omega,
            \end{aligned}
        \right.
        \qquad
        \left\{
            \begin{aligned}
                -\partial_t \psi_q-\Delta \psi_q+q(t)\,\psi_q&=z_q &&\text{in }Q,\\
                \psi_q&=0 &&\text{on }\Sigma,\\
                \psi_q(T)&=0 &&\text{in }\Omega .
            \end{aligned}
        \right.
    \end{equation}
    Then, the following statements hold.
    \begin{enumerate}[label=(\roman*),nosep]
        \item\label{al1:i} For each $q\in L^2(0,T)$, the linear operator $A_q:L^2(\Omega)\to L^2(\Omega)$ is bounded, self-adjoint and satisfies
            \[
                \langle A_qv,v\rangle_{L^2(\Omega)} \ge 0\quad \forall v\in L^2(\Omega).
            \]

        \item\label{al1:ii} For any $M>0$ there exists $C>0$ such that, for  $q_1,q_2\in \mathbb B_{L^2(0,T)}(0,M)$, there holds
            \[
                \|A_{q_1}v-A_{q_2}v\|_{L^2(\Omega)}\le C\|q_1-q_2\|_{L^2(0,T)} \|v\|_{L^2(\Omega)}\quad\forall v\in L^2(\Omega).
            \]
    \end{enumerate}
\end{lemma}
\begin{proof}
    This follows from standard arguments using the usual parabolic estimates and the integration by parts formula.
\end{proof}

Finally, we record a lemma on the uniqueness of solutions of system \eqref{Btclin} in a special case.

\begin{lemma}\label{lem:unique-v} Let $(q,u)\in\mathcal C$, and denote by $(y,\varphi)\in W(0,T)\times W(0,T)$
    its associated state--adjoint pair.
    Let $h\in L^2(0,T)$ and
    $(v_1,z_1,\psi_1),(v_2,z_2,\psi_2)\in L^2(\Omega)\times W(0,T)\times W(0,T)$.  Suppose that each triple $(v_i,z_i,\psi_i)$ satisfies
    \eqref{Btclin} with the same $h$.
    Then $v_1=v_2$ in $L^2(\Omega)$.
\end{lemma}
\begin{proof}
    Set $w:=v_1-v_2$, $\sigma:=z_1-z_2$, and $\chi:=\psi_1-\psi_2$.
    Subtracting the two instances of \eqref{Btclin}  yields
    \[
        \left\{
            \begin{aligned}
                \partial_t \sigma-\Delta \sigma+q(t)\,\sigma&=0 &&\text{in }Q,\\
                \sigma&=0 &&\text{on }\Sigma,\\
                \sigma(0)&=w &&\text{in }\Omega,
            \end{aligned}
        \right.
        \qquad
        \left\{
            \begin{aligned}
                -\partial_t \chi-\Delta \chi+q(t)\,\chi&=\sigma &&\text{in }Q,\\
                \chi&=0 &&\text{on }\Sigma,\\
                \chi(T)&=0 &&\text{in }\Omega .
            \end{aligned}
        \right.
    \]
    Testing the $\sigma$--equation by $\chi$ and the $\chi$--equation
    by $\sigma$, subtracting the resulting identities, and using
    \eqref{eq:ibp-W} (with $\chi(T)=0$ and $\sigma(0)=w$) gives
    \begin{equation}\label{eq:unique-v:id}
        \int_{\Omega}\chi(0)w\, dx
        =\int_0^T \|\sigma(t)\|_{L^2(\Omega)}^2\,dt
        \;\ge\;0.
    \end{equation}
    Denote by
    $f:\mathbb R\to\mathbb R$ the function in \cref{al0}.
    Its directional derivative at $r\in\mathbb R$ in direction $d\in\mathbb R$ is given in \eqref{ddb}.
    We write
    \[
        r:=-\frac{1}{\lambda}\,\varphi(0)\quad\text{and}\quad d_i:=-\frac{1}{\lambda}\,\psi_i(0).
    \]
    Then $v_i=f'(r;d_i)$. For each
    fixed $r\in\mathbb R$, the map $d\longmapsto f'(r;d)$ is monotone nondecreasing,
    hence
    \[
        w\hspace{0.02cm}(d_1-d_2)=(f'(r;d_1)-f'(r;d_2))(d_1-d_2)\ge 0\quad\text{in $L^2(\Omega)$}.
    \]
    Since $d_1-d_2=-(\chi(0))/\lambda$ and $\lambda>0$, this implies $\int_{\Omega} w\,\chi(0)\, dx\le 0$.
    Combining this with \eqref{eq:unique-v:id} gives
    $\int_{\Omega}\chi(0)w\, dx=0$ and therefore
    $\int_0^T\|\sigma(t)\|_{L^2(\Omega)}^2\,dt=0$. Hence $\sigma\equiv 0$ in $Q$,
    so $w=\sigma(0)=0$ in $L^2(\Omega)$.
\end{proof}

\subsubsection{The Bouligand tangent cone}
Fix $(q,u)\in\mathcal C$ and let $(y,\varphi)\in W(0,T)\times W(0,T)$ be the
associated state--adjoint pair. Denote by
$f:\mathbb R\to\mathbb R$ the function in \cref{al0}.
Let $\Theta$ denote the set of elements $(h,v)\in L^2(0,T)\times L^2(\Omega)$ for which there exists a pair $(z,\psi)\in W(0,T)\times W(0,T)$ satisfying \eqref{Btclin}. We prove below that $\Theta = T_{\mathcal C}(q,u)$.

\paragraph{(Inclusion $T_{\mathcal C}(q,u)\subseteq\Theta$).} Fix $(h,v)\in T_{\mathcal C}(q,u)$. By definition of the
Bouligand tangent cone, there exist sequences $\varepsilon_n\searrow0$ and
$(h_n,v_n)\to(h,v)$ in $L^2(0,T)\times L^2(\Omega)$ such that
\begin{equation}\label{fcb}
    (q,u)+\varepsilon_n(h_n,v_n)\in\mathcal C\qquad\forall\,n\in\mathbb N.
\end{equation}
For each $n\in\mathbb N$, let $(y_n,\varphi_n)\in W(0,T)\times W(0,T)$ be the
state--adjoint pair associated with the tuple $(q+\varepsilon_n h_n,\ u+\varepsilon_n v_n)$. Due to \eqref{fcb}, for each $n\in\mathbb N$,
\begin{equation}\label{proxCondition}
    u+\varepsilon_n v_n
    =\operatorname{prox}_{(\mu/\lambda)|\cdot|}\!\left(-\frac{1}{\lambda}\,\varphi_n(0)\right)
    =f\!\left(-\frac{1}{\lambda}\,\varphi_n(0)\right)
    \quad\text{a.e.\ in }\Omega.
\end{equation}
Define, for each $n\in\mathbb N$,
$z_n:=\varepsilon_n^{-1}(y_n-y)$ and
$\psi_n:=\varepsilon_n^{-1}(\varphi_n-\varphi)$.
Subtracting the state and adjoint equations and using parabolic stability estimates,
we obtain $z_n\to z$ and $\psi_n\to\psi$ in $W(0,T)$, where
$(z,\psi)$ is the unique solution of the linearized state--adjoint system
\[
    \left\{
        \begin{aligned}
            \partial_t z - \Delta z + q(t)\,z + h(t)\,y &= 0
            && \text{in } Q,\\
            z &= 0
            && \text{on } \Sigma,\\
            z(0) &= v
            && \text{in } \Omega,
        \end{aligned}
    \right.
    \qquad
    \left\{
        \begin{aligned}
            -\partial_t \psi - \Delta \psi + q(t)\,\psi + h(t)\,\varphi &= z
            && \text{in } Q,\\
            \psi &= 0
            && \text{on } \Sigma,\\
            \psi(T) &= 0
            && \text{in } \Omega.
        \end{aligned}
    \right.
\]
It remains to identify the pointwise trace relation between $v$ and $\psi(0)$. Set
\[
    r:=-\frac{1}{\lambda}\,\varphi(0)\quad\text{and}\quad  r_n:=-\frac{1}{\lambda}\,\varphi_n(0)\quad \forall n\in\mathbb N.
\]
Using the continuous embedding $W(0,T)\hookrightarrow C\big([0,T];L^2(\Omega)\big)$, we conclude that  $(r_n-r)/\varepsilon_n\to -\psi(0)/\lambda$ in $L^2(\Omega)$. Moreover,
\eqref{proxCondition} and $u=f(r)$ imply $v_n=\varepsilon_n^{-1}\big(f(r_n)-f(r)\big)$ for all $n\in\mathbb N$. Up to a subsequence, we thus obtain, for a.e.\ $x\in\Omega$,
\[
    v_n(x)
    =\frac{f\bigl(r(x)+\varepsilon_n\,\tfrac{r_n(x)-r(x)}{\varepsilon_n}\bigr)-f(r(x))}{\varepsilon_n}
    \longrightarrow
    f'\!\left(r(x);\,-\frac{1}{\lambda}\psi(0,x)\right)\quad \text{as} \quad n\to+\infty.
\]
Hence $v=f'(-\varphi(0)/\lambda;-\psi(0)/\lambda)$ a.e.\ in $\Omega$, and thus $v$ and $\psi$ satisfy the relation \eqref{Btclin} (due to the representation \eqref{ddb}).

\paragraph{(Inclusion $\Theta\subseteq T_{\mathcal C}(q,u)$).} Fix $(h,v)\in L^2(0,T)\times L^2(\Omega)$ such that there exists a pair
$(z,\psi)\in W(0,T)\times W(0,T)$ satisfying \eqref{Btclin}. Let $\{\varepsilon_n\}_{n\in\mathbb N}\subseteq(0,+\infty)$ be a sequence such that $\varepsilon_n\searrow 0$. For each $n\in\mathbb N$, let \(q_n:=q+\varepsilon_n h\), and let \(u_n:=u_{q_n}\) be the unique lower-level
solution corresponding to \(q_n\). Let \((y_n,\varphi_n)\in W(0,T)\times W(0,T)\)
be the associated state--adjoint pair. Then \((q_n,u_n)\in\mathcal C\), and the
lower-level optimality system gives
\[
    u_n
    =
    \operatorname{prox}_{(\mu/\lambda)|\cdot|}
    \left(-\frac1\lambda\varphi_n(0)\right)
    =
    f\left(-\frac1\lambda\varphi_n(0)\right)
    \quad\text{a.e. in }\Omega .
\]
It follows that $(q_n,u_n)\in\mathcal C$. Define, for each $n\in\mathbb N$,  $v_n:=\varepsilon_n^{-1}(u_n-u)$. By construction, we see that
\[
    (q,u)+\varepsilon_n(h,v_n)=(q_n,u_n)\in\mathcal C\quad\forall\,n\in\mathbb N.
\]
It suffices to prove that $v_n\to v$ in $L^2(\Omega)$ to conclude the desired inclusion.
Set
\[
    z_n:=\frac{y_n-y}{\varepsilon_n}\quad\text{and}\quad
    \psi_n:=\frac{\varphi_n-\varphi}{\varepsilon_n}
    \quad \forall n\in\mathbb N.
\]
Subtracting the state equations for $y_n$ and $y$, using $q_n=q+\varepsilon_n h$,
and dividing by $\varepsilon_n$ yields
\begin{equation}\label{eq:lin-sys-n}
    \left\{
        \begin{aligned}
            \partial_t z_n-\Delta z_n+q_n(t)\,z_n+h(t)\,y&=0
            &&\text{in }Q,\\
            z_n&=0
            &&\text{on }\Sigma,\\
            z_n(0)&=v_n
            &&\text{in }\Omega,
        \end{aligned}
    \right.
    \quad
    \left\{
        \begin{aligned}
            -\partial_t \psi_n-\Delta \psi_n+q_n(t)\,\psi_n+h(t)\,\varphi&=z_n
            &&\text{in }Q,\\
            \psi_n&=0
            &&\text{on }\Sigma,\\
            \psi_n(T)&=0
            &&\text{in }\Omega .
        \end{aligned}
    \right.
\end{equation}
We now split the rest of the proof of this inclusion into two substeps.

\paragraph{Substep 1 (the sequence is bounded).}
For each \(n\in\mathbb N\), let \(\rho_n\in W(0,T)\) be the unique solution of
\begin{equation}\label{eq:data-adj}
    \left\{
        \begin{aligned}
            -\partial_t \rho_n-\Delta \rho_n+q_n(t)\,\rho_n &= -y_d
            &&\text{in }Q,\\
            \rho_n&=0 &&\text{on }\Sigma,\\
            \rho_n(T)&=0 &&\text{in }\Omega,
        \end{aligned}
    \right.
\end{equation}
and define \(b_{q_n}:=\rho_n(0)\in L^2(\Omega)\). Likewise, let \(\rho\in W(0,T)\) be
the unique solution of
\begin{equation}\label{eq:data-adj2}
    \left\{
        \begin{aligned}
            -\partial_t \rho-\Delta \rho+q(t)\,\rho &= -y_d
            &&\text{in }Q,\\
            \rho&=0 &&\text{on }\Sigma,\\
            \rho(T)&=0 &&\text{in }\Omega,
        \end{aligned}
    \right.
\end{equation}
and set \(b_q:=\rho(0)\in L^2(\Omega)\).
For each $n\in\mathbb N$, let $A_{q_n}:L^2(\Omega)\to L^2(\Omega)$ be the operator in \cref{al1}.
By linearity,
\[
    \varphi_n(0)=A_{q_n}u_n+b_{q_n}
    \quad\text{and}\quad
    \varphi(0)=A_q u+b_q.
\]
The proximal relations are equivalent to the inclusions
\begin{equation}\label{eq:incl-qn}
    0\in (\lambda I+A_{q_n})u_n+b_{q_n}+\mu\,\partial|u_n|\quad\text{and}\quad
    0\in (\lambda I+A_q)u+b_q+\mu\,\partial|u|.
\end{equation}
Choose $\zeta_n\in\partial|u_n|$ and $\zeta\in\partial|u|$ such that
\[
    (\lambda I+A_{q_n})u_n+b_{q_n}+\mu \zeta_n=0\quad\text{and}\quad
    (\lambda I+A_q)u+b_q+\mu \zeta=0.
\]
Subtracting and taking the $L^2(\Omega)$-inner product with $(u_n-u)$ gives
\begin{align*}
    &\lambda\|u_n-u\|_{L^2(\Omega)}^2
    +\langle A_{q_n}(u_n-u),u_n-u\rangle_{L^2(\Omega)}
    +\mu\langle \zeta_n-\zeta,u_n-u\rangle_{L^2(\Omega)}
    \\
    &\hspace{2cm}
    =
    -\langle (A_{q_n}-A_q)u,u_n-u\rangle_{L^2(\Omega)}
    -\langle b_{q_n}-b_q,u_n-u\rangle_{L^2(\Omega)}.
\end{align*}
By \cref{al1} \ref{al1:i} and the monotonicity of \(\partial|\cdot|\) we have
\[
    \langle A_{q_n}(u_n-u),u_n-u\rangle_{L^2(\Omega)}\ge 0\quad \text{and}\quad  \langle \zeta_n-\zeta,u_n-u\rangle_{L^2(\Omega)}\ge 0\qquad\forall n\in\mathbb N.
\]
Hence, $\lambda\|u_n-u\|_{L^2}^2
\le
\|(A_{q_n}-A_q)u\|_{L^2}\,\|u_n-u\|_{L^2}
+\|b_{q_n}-b_q\|_{L^2}\,\|u_n-u\|_{L^2}$, and therefore
\begin{equation}\label{eq:Lip-u}
    \|u_n-u\|_{L^2(\Omega)}
    \le
    \frac{1}{\lambda}\Big(\|(A_{q_n}-A_q)u\|_{L^2(\Omega)}+\|b_{q_n}-b_q\|_{L^2(\Omega)}\Big)
\end{equation}
for all $n\in\mathbb N$. By \cref{al1} \ref{al1:ii}, there exists $C_1>0$ such that
\[
    \|(A_{q_n}-A_q)u\|_{L^2(\Omega)}
    \le C_1\,\|q_n-q\|_{L^2(0,T)}\,\|u\|_{L^2(\Omega)}\quad \forall n\in\mathbb N.
\]
For the second term on the right-hand side of  \eqref{eq:Lip-u}, set \(w_n:=\rho_n-\rho\). Then each \(w_n\) satisfies
\begin{equation}\label{eq:wn}
    \left\{
        \begin{aligned}
            -\partial_t w_n-\Delta w_n+q_n(t)\,w_n &= \big(q(t)-q_n(t)\big)\,\rho
            &&\text{in }Q,\\
            w_n&=0 &&\text{on }\Sigma,\\
            w_n(T)&=0 &&\text{in }\Omega.
        \end{aligned}
    \right.
\end{equation}
Standard energy estimates yield a constant \(C_2>0\) such that
\[
    \|w_n(0)\|_{L^2(\Omega)} \le C_2\,\|(q_n-q)\rho\|_{L^2(Q)}
    \le C_2\,\|q_n-q\|_{L^2(0,T)}\,\|\rho\|_{L^\infty(0,T;L^2(\Omega))}\quad \forall n\in\mathbb N.
\]
Similarly, there exists $C_3>0$ such that
$\|\rho\|_{L^\infty(0,T;L^2(\Omega))}\le C_3\,\|y_d\|_{L^2(Q)}$, hence
\[
    \|b_{q_n}-b_q\|_{L^2(\Omega)}
    =\|w_n(0)\|_{L^2(\Omega)}
    \le C_2 C_3\,\|q_n-q\|_{L^2(0,T)}\,\|y_d\|_{L^2(Q)}\quad \forall n\in\mathbb N.
\]
Combining with \eqref{eq:Lip-u} we infer there exists $C_4>0$ such that
\[
    \|u_n-u\|_{L^2(\Omega)} \le C_4\|q_n-q\|_{L^2(0,T)}=C_4\varepsilon_n\|h\|\quad \forall n\in\mathbb N.
\]
Thereby concluding that $(v_n)_{n\in\mathbb N}$ is bounded in $L^2(\Omega)$.

\paragraph{Substep 2 (convergence and identification of the limit).}
Since $(v_n)_{n\in\mathbb N}$ is bounded in $L^2(\Omega)$ by Step~1, there exist a
subsequence (not relabeled) and $\bar v\in L^2(\Omega)$ such that $v_n \rightharpoonup \bar v $ weakly in $L^2(\Omega)$. We show that $\bar v=v$ and then conclude that the full sequence converges.

From \eqref{eq:lin-sys-n} and standard
energy estimates we obtain boundedness of $(z_n)_{n\in\mathbb N}$ in $W(0,T)$.
Using the Aubin--Lions compactness theorem, and yet more standard parabolic estimates, one can find that there exists a further
subsequence (not relabeled) and functions $\bar z,\bar\psi\in W(0,T)$ such that
\begin{equation}\label{eq:zn-psin-compact}
    z_n \rightharpoonup \bar z \ \text{ weakly in }W(0,T),\qquad z_n\to \bar z \ \text{ strongly in }L^2(Q), \qquad
    \psi_n \to \bar\psi \ \text{ strongly in }W(0,T).
\end{equation}
Passing to the limit in \eqref{eq:lin-sys-n} shows the desired PDE relations for $\bar z$ and $\bar\psi$. Let us now identify the trace relation in \eqref{Btclin} for $\bar v$.
Set $r:=-\frac1\lambda\varphi(0)$ and $r_n:=-\frac1\lambda\varphi_n(0)$. Then
$u=f(r)$ and $u_n=f(r_n)$, hence
\[
    v_n=\frac{u_n-u}{\varepsilon_n}
    =\frac{f(r_n)-f(r)}{\varepsilon_n}\quad \forall n\in\mathbb N.
\]
Moreover, since $\psi_n=(\varphi_n-\varphi)/\varepsilon_n$ and
$W(0,T)\hookrightarrow C([0,T];L^2(\Omega))$ continuously, we have
\begin{equation}\label{eq:rn-expansion}
    \frac{r_n-r}{\varepsilon_n}
    =-\frac{1}{\lambda}\,\psi_n(0)\to -\frac{1}{\lambda}\,\bar\psi(0)
    \quad\text{strongly in }L^2(\Omega).
\end{equation}
Define $d_n:=(r_n-r)/\varepsilon_n$ and $d:=-\bar\psi(0)/\lambda$. For each $n\in\mathbb N$, define
$F_n:\mathbb R\times\mathbb R\to\mathbb R$ by $F_n(\rho,\eta):=\varepsilon^{-1}_n\big(f(\rho+\varepsilon_n\eta)-f(\rho)\big)$. Since $u_n=f(r_n)$ and $u=f(r)$, we have $v_n=F_n(r;d_n)$. Using that $f$
is $1$--Lipschitz, we obtain that,  for a.e.\ in $\Omega$, $|F_n(r;d_n)-F_n(r;d)|\le |d_n-d|$,
hence $\|F_n(r;d_n)-F_n(r;d)\|_{L^2(\Omega)}\to0$. Moreover, we have, for every $n\in\mathbb N$,
\[
    |F_n(r;d)|
    =\left|\frac{f(r+\varepsilon_n d)-f(r)}{\varepsilon_n}\right|
    \le \frac{|\,\varepsilon_n d\,|}{\varepsilon_n}
    =|d|\quad\text{a.e in $\Omega$.}
\]
In addition, for a.e.\ $x\in\Omega$, the scalar limit
\[
    F_n(r;d)(x)=\frac{f(r(x)+\varepsilon_n d(x))-f(r(x))}{\varepsilon_n}
    \longrightarrow f'(r(x);d(x))
    \quad\text{as }n\to\infty
\]
holds by the definition of the directional derivative. The dominated convergence theorem then implies $\|F_n(r;d)-f'(r;d)\|_{L^2(\Omega)}\to 0$. By the triangle inequality
\[
    \|v_n-f'(r;d)\|_{L^2(\Omega)}
    \le
    \|v_n-F_n(r;d)\|_{L^2(\Omega)}+\|F_n(r;d)-f'(r;d)\|_{L^2(\Omega)}
    \longrightarrow 0.
\]
This can be rewritten as $v_n\to f'(r;d)$ strongly in $L^2(\Omega)$. In particular, $\bar v$ satisfies
\begin{equation}\label{eq:vbar-dd}
    \bar v=f'(r;d)=f'\!\left(-\frac1\lambda\varphi(0);\,-\frac1\lambda\bar\psi(0)\right)
    \quad\text{a.e.\ in }\Omega,
\end{equation}
and therefore \eqref{Btclin} holds for $(\bar v,\bar z,\bar\psi)$ (using
\eqref{ddb}).

Thus, for the fixed $h$, the triple $(\bar v,\bar z,\bar\psi)$ satisfies
\eqref{Btclin}. By assumption, $(v,z,\psi)$ also satisfies
\eqref{Btclin} with the same $h$. \cref{lem:unique-v} implies
$\bar v=v$ in $L^2(\Omega)$. Consequently, $v_n\rightharpoonup v$ weakly in
$L^2(\Omega)$ along the extracted subsequence. Since we already proved that
$v_n\to f'(r;d)$ strongly in $L^2(\Omega)$, the whole sequence converges and $v_n\to v$ strongly in $L^2(\Omega)$. This concludes the proof of the inclusion.

\subsection{Proof of Corollary \ref{Ctc}}
Fix $(q,u)\in\mathcal C$ and let $(y,\varphi)\in W(0,T)\times W(0,T)$ be the
associated state--adjoint pair. Denote by
$f:\mathbb R\to\mathbb R$ the function in \cref{al0}.

Let $\Theta$ denote the set of elements $(h,v)\in L^2(0,T)\times L^2(\Omega)$ for which there exists a pair $(z,\psi)\in W(0,T)\times W(0,T)$ satisfying \eqref{Btclin}. We already know from \cref{p1} that $T_{\mathcal C}(q,u)=\Theta$, and it follows from the basic definitions of both cones that $\hat T_{\mathcal C}(q,u)\subseteq T_{\mathcal C}(q,u)$. It is enough then to prove that $\Theta \subseteq \hat T_{\mathcal C}(q,u)$ to conclude the proof; we do this below.

Fix $(h,v)\in\Theta$, and let $(z,\psi)\in W(0,T)\times W(0,T)$ be a pair
satisfying \eqref{Btclin}. We prove that $(h,v)\in \hat
T_{\mathcal C}(q,u)$.
Let $(q_n,u_n)_{n\in\mathbb N}\subseteq \mathcal C$ be any sequence such that
$(q_n,u_n)\to(q,u)$ in $L^2(0,T)\times L^2(\Omega)$, and let $t_n\searrow 0$ be
arbitrary. Define, for each $n\in\mathbb N$,
\[
    \tilde q_n:=q_n+t_n h,
    \qquad
    \tilde u_n:=u_{\tilde q_n},
    \qquad
    h_n:=h,
    \qquad
    v_n:=\frac{\tilde u_n-u_n}{t_n}.
\]
Therefore $(q_n,u_n)+t_n(h_n,v_n)=(\tilde q_n,\tilde u_n)\in\mathcal C$ for all $n\in\mathbb N$. Thus it remains to show that $v_n\to v$ in $L^2(\Omega)$. Let $(y_n,\varphi_n)$ (resp.\ $(\tilde y_n,\tilde\varphi_n)$) be the
state--adjoint pair associated with $(q_n,u_n)$ (resp.\ $(\tilde q_n,\tilde
u_n)$). By standard parabolic stability estimates with respect to the
coefficient and the initial datum (using that  $(q_n,u_n)\to(q,u)$ strongly in $L^2(0,T)\times L^2(\Omega)$), we can conclude that $ y_n\to y$ and $\varphi_n\to \varphi$ strongly in $W(0,T)$;
hence in particular $\varphi_n(0)\to\varphi(0)$ strongly in $L^2(\Omega)$. Define
\[
    z_n:=\frac{\tilde y_n-y_n}{t_n}\quad \text{and}\quad \psi_n:=\frac{\tilde\varphi_n-\varphi_n}{t_n}\qquad \forall n\in\mathbb N.
\]
Subtraction of the two state equations and the two adjoint equations
gives the linearized system
\begin{equation}\label{eq:ctc-linsys-n}
    \left\{
        \begin{aligned}
            \partial_t z_n-\Delta z_n+\tilde q_n(t)\,z_n+h(t)\,y_n&=0
            &&\text{in }Q,\\
            z_n&=0
            &&\text{on }\Sigma,\\
            z_n(0)&=v_n
            &&\text{in }\Omega,
        \end{aligned}
    \right.
    \quad
    \left\{
        \begin{aligned}
            -\partial_t\psi_n-\Delta \psi_n+\tilde q_n(t)\,\psi_n+h(t)\,\varphi_n&=z_n
            &&\text{in }Q,\\
            \psi_n&=0
            &&\text{on }\Sigma,\\
            \psi_n(T)&=0
            &&\text{in }\Omega .
        \end{aligned}
    \right.
\end{equation}
Arguing exactly as in Step~1 of the proof of the inclusion
$\Theta\subseteq T_{\mathcal C}(q,u)$ in \cref{p1}, but now applied
to the pair of coefficients $q_n$ and $\tilde q_n=q_n+t_n h$, one obtains a
Lipschitz estimate
\[
    \|\tilde u_n-u_n\|_{L^2(\Omega)}\le C\,\|\tilde q_n-q_n\|_{L^2(0,T)}
    =C\,t_n\|h\|_{L^2(0,T)}
\]
with a constant $C>0$ independent of $n$; this because $(q_n)$ is bounded and
\cref{al1} \ref{al1:ii} applies on bounded balls. Hence $(v_n)$ is bounded in
$L^2(\Omega)$, and we can extract a weakly convergent subsequence (not relabeled) so that $v_n \rightharpoonup\bar v$ for some $\bar v\in L^2(\Omega)$. From \eqref{eq:ctc-linsys-n} and standard energy estimates,
$(z_n)$ is bounded in $W(0,T)$. Using the Aubin--Lions compactness theorem, and yet more standard parabolic estimates, one can find that there exists a further
subsequence (not relabeled) and functions $\bar z,\bar\psi\in W(0,T)$ such that
\begin{equation}\label{eq:ctc-subseq}
    z_n\to \bar z \ \text{ strongly in }L^2(Q)\quad\text{and} \quad
    \psi_n \to \bar\psi \ \text{ strongly in }W(0,T).
\end{equation}
Since $\tilde q_n\to q$ in $L^2(0,T)$ and $y_n\to y$, $\varphi_n\to\varphi$ in
$W(0,T)$, passing to the limit in \eqref{eq:ctc-linsys-n} yields that the couple
$(\bar z,\bar\psi)$ satisfies the system
\[
    \left\{
        \begin{aligned}
            \partial_t \bar z - \Delta \bar z + q(t)\,\bar z + h(t)\,y &= 0
            && \text{in } Q,\\
            \bar z &= 0
            && \text{on } \Sigma,\\
            \bar z(0) &= \bar v
            && \text{in } \Omega,
        \end{aligned}
    \right.
    \qquad
    \left\{
        \begin{aligned}
            -\partial_t \bar \psi - \Delta \bar \psi + q(t)\,\bar \psi + h(t)\,\varphi &= \bar z
            && \text{in } Q,\\
            \bar \psi &= 0
            && \text{on } \Sigma,\\
            \bar \psi(T) &= 0
            && \text{in } \Omega.
        \end{aligned}
    \right.
\]
Set, for each $n\in\mathbb N$,
\[
    r_n:=-\frac1\lambda\,\varphi_n(0),\qquad
    \tilde r_n:=-\frac1\lambda\,\tilde\varphi_n(0),\qquad
    d_n:=\frac{\tilde r_n-r_n}{t_n}=-\frac1\lambda\,\psi_n(0).
\]
Then $u_n=f(r_n)$ and $\tilde u_n=f(\tilde r_n)=f(r_n+t_n d_n)$, hence
\begin{equation}\label{eq:ctc-vn-prox}
    v_n=\frac{f(r_n+t_n d_n)-f(r_n)}{t_n}\qquad\text{a.e.\ in }\Omega.
\end{equation}
From $\varphi_n(0)\to\varphi(0)$ in $L^2(\Omega)$ we have $r_n\to r:=-\varphi(0)/\lambda$
in $L^2(\Omega)$, and from \eqref{eq:ctc-subseq} we have
$d_n\to d:=-\bar\psi(0)/\lambda$ in $L^2(\Omega)$.

Recall the assumption $|\{|\varphi(0)|=\mu\}|=0$.
Fix $\delta>0$ and set $ B_\delta:=\{\,||r|-\tau|<\delta\,\}$. Then $|B_\delta|\to0$ as $\delta \searrow0$. Define $G_{n,\delta}:=\{|r_n-r|\ge \delta/4\}$. Since $r_n\to r$ in $L^2(\Omega)$, by Chebyshev's inequality,
\[
    |G_{n,\delta}|\le \frac{16}{\delta^2}\,\|r_n-r\|_{L^2(\Omega)}^2\to0.
\]
On $\Omega\setminus(B_\delta\cup G_{n,\delta})$ we have $||r|-\tau|\ge\delta$ and
$|r_n-r|<\delta/4$, hence
\[
    \bigl||r_n|-\tau\bigr|
    \ge \bigl||r|-\tau\bigr|-||r_n|-|r||
    \ge \delta-\delta/4
    =3\delta/4
    \ge \delta/2.
\]
Moreover, set $ F_{n,\delta}:=\{|t_n d_n|\ge \delta/4\}
=\Bigl\{|d_n|\ge \frac{\delta}{4t_n}\Bigr\}$. We claim that, for fixed $\delta>0$,
\begin{equation}\label{eq:ctc-tail}
    \|d_n\,\mathbbm{1}_{F_{n,\delta}}\|_{L^2(\Omega)}\to0\quad\text{as } n\to+\infty.
\end{equation}
Fix $\varepsilon>0$ and choose
$M>0$ such that $\int_{\{|d|>M\}}|d|^2<\varepsilon$. For $n\in\mathbb N$ large enough,
$\|d_n-d\|_{L^2(\Omega)}^2<\varepsilon$ and $\delta/(4t_n)\ge 2M$. Then
\[
    \int_{F_{n,\delta}} |d_n|^2
    =\int_{\{|d_n|\ge \delta/(4t_n)\}}|d_n|^2
    \le \int_{\{|d_n|\ge 2M\}}|d_n|^2
    \le \int_{\{|d|>M\}}|d_n|^2 + \int_{\{|d_n-d|>M\}}|d_n|^2.
\]
Using $|d_n|^2\le 2|d|^2+2|d_n-d|^2$, we estimate
\[
    \int_{\{|d|>M\}}|d_n|^2 \le 2\int_{\{|d|>M\}}|d|^2 + 2\|d_n-d\|_{L^2}^2 \le 4\varepsilon,
\]
and since $|\{|d_n-d|>M\}|\le \|d_n-d\|_{L^2(\Omega)}^2/M^2\to0$ and $d\in L^2(\Omega)$,
\[
    \int_{\{|d_n-d|>M\}}|d_n|^2
    \le 2\int_{\{|d_n-d|>M\}}|d|^2 + 2\int_{\{|d_n-d|>M\}}|d_n-d|^2 \to 0\quad \text{as }n\to+\infty.
\]
This proves \eqref{eq:ctc-tail}. Now define
\[
    S_{n,\delta}:=\Omega\setminus(B_\delta\cup G_{n,\delta}\cup F_{n,\delta}).
\]
For $x\in S_{n,\delta}$ we have $\bigl||r_n(x)|-\tau\bigr|\ge\delta/2$ and
$|t_nd_n(x)|<\delta/4$, hence the segment
$\{ r_n(x) + s t_n d_n(x) : s \in [0,1] \}$ does not cross the thresholds $\pm\tau$.
Since $f$ is affine on each region $\{s\in\mathbb R:\,|s|<\tau\}$ and $\{s\in\mathbb R:\,|s|>\tau\}$, the quotient
\eqref{eq:ctc-vn-prox} coincides pointwise on $S_{n,\delta}$ with the directional
derivative $f'(r_n;d_n)$. Using also that $|v_n|\le |d_n|$ and $|f'(r_n;d_n)|\le |d_n|$ (by the $1$-Lipschitzness of
$f$ and \eqref{ddb}), we obtain
\[
    |v_n-f'(r_n;d_n)|
    \le 2|d_n|\,\mathbbm{1}_{B_\delta\cup G_{n,\delta}\cup F_{n,\delta}}
    \quad\text{a.e.\ in }\Omega,
\]
and therefore
\[
    \|v_n-f'(r_n;d_n)\|_{L^2(\Omega)}
    \le 2\|d_n\mathbbm{1}_{B_\delta}\|_{L^2}
    +2\|d_n\mathbbm{1}_{G_{n,\delta}}\|_{L^2}
    +2\|d_n\mathbbm{1}_{F_{n,\delta}}\|_{L^2}.
\]
For fixed $\delta$, we have $\|d_n\mathbbm{1}_{G_{n,\delta}}\|_{L^2}\to0$ since
\[
    \|d_n\mathbbm{1}_{G_{n,\delta}}\|_{L^2}
    \le \|d_n-d\|_{L^2}+\|d\,\mathbbm{1}_{G_{n,\delta}}\|_{L^2}\to0,
\]
and $\|d_n\mathbbm{1}_{F_{n,\delta}}\|_{L^2}\to0$ by \eqref{eq:ctc-tail}. Moreover, $\limsup_{n\to\infty}\|d_n\mathbbm{1}_{B_\delta}\|_{L^2}
\le \|d\,\mathbbm{1}_{B_\delta}\|_{L^2}$. Hence
\[
    \limsup_{n\to\infty}\|v_n-f'(r_n;d_n)\|_{L^2(\Omega)}
    \le 2\|d\,\mathbbm{1}_{B_\delta}\|_{L^2(\Omega)}.
\]
Letting $\delta\downarrow0$ and using $|B_\delta|\to0$ yields $\|v_n-f'(r_n;d_n)\|_{L^2(\Omega)}\to0$. Next, since $\{|\varphi(0)|=\mu\}$ has null measure, the map $(\rho,\eta)\mapsto f'(\rho;\eta)$ is continuous at
$(r(x),d(x))$ for a.e.\ $x\in\Omega$ (indeed $f'(\rho;\eta)=0$ for $|\rho|<\tau$
and $f'(\rho;\eta)=\eta$ for $|\rho|>\tau$). Together with $r_n\to r$ and
$d_n\to d$ in $L^2(\Omega)$, a localization argument as above yields $ f'(r_n;d_n)\to f'(r;d)$ in $L^2(\Omega)$.
Combining the last two convergences gives
\begin{equation}
    v_n\to f'(r;d)
    =f'\!\left(-\frac1\lambda\varphi(0);\,-\frac1\lambda\bar\psi(0)\right)
    \quad\text{in }L^2(\Omega).
\end{equation}
In particular, $\bar v=f'(r;d)$ a.e.\ in $\Omega$. By \cref{al0}, this is
exactly the trace relation in \eqref{Btclin}. Therefore the triple $(\bar v,\bar z,\bar\psi)$ satisfies
\eqref{Btclin} with the given $h$. By assumption, $(v,z,\psi)$ also satisfies \eqref{Btclin} with the
same $h$. \cref{lem:unique-v} therefore implies $\bar v=v$ in $L^2(\Omega)$. Since the preceding argument applies to every subsequence of $(v_n)$,
every subsequence has a further subsequence converging strongly to $v$. Hence the full sequence satisfies $v_n\to v$ in $L^2(\Omega)$. This completes the proof.

\subsection{Proof of Proposition \ref{Boucon}}

We prove the result by establishing the two inclusions. We first collect a few
preliminary lemmas that will be used to pass between the polar inequality and
the adjoint representation. Using  \cref{p1}, we then verify that every pair admitting the stated
multiplier representation belongs to the normal cone. For the converse, we
start from the polar inequality and use the observability condition to recover
the adjoint variables and the corresponding complementarity conditions.

\subsubsection{Some preparatory lemmas}
We begin with a few auxiliary results that will be used to derive the
multiplier representation from the polar inequality.
\begin{lemma}\label{lem:Theta-surj}
    Let $(q,u)\in\mathcal C$ with associated state--adjoint pair
    $(y,\varphi)\in W(0,T)\times W(0,T)$. Let $\Theta\subseteq L^2(0,T)\times L^2(\Omega)$
    be the set of all pairs $(h,v)$ for which there exists $(z,\psi)\in W(0,T)\times W(0,T)$
    satisfying \eqref{Btclin}. Then, for every $h\in L^2(0,T)$ there exists
    a unique $v\in L^2(\Omega)$ such that $(h,v)\in\Theta$. In particular, the projection
    of $\Theta$ onto its first component equals $L^2(0,T)$.
\end{lemma}

\begin{proof}
    Fix $h\in L^2(0,T)$. Let $(z^h,\psi^h)\in W(0,T)\times W(0,T)$ be the unique weak
    solution of
    \[
        \left\{
            \begin{aligned}
                \partial_t z^h-\Delta z^h+q(t)\,z^h &= -h(t)\,y
                && \text{in }Q,\\
                z^h&=0 && \text{on }\Sigma,\\
                z^h(0)&=0 && \text{in }\Omega,
            \end{aligned}
        \right.
        \qquad
        \left\{
            \begin{aligned}
                -\partial_t \psi^h-\Delta \psi^h+q(t)\,\psi^h &= z^h-h(t)\,\varphi
                && \text{in }Q,\\
                \psi^h&=0 && \text{on }\Sigma,\\
                \psi^h(T)&=0 && \text{in }\Omega.
            \end{aligned}
        \right.
    \]
    Set $g:=\psi^h(0)\in L^2(\Omega)$. For $v\in L^2(\Omega)$, let $(z_v,\psi_v)\in W(0,T)\times W(0,T)$
    be the unique solution of \eqref{Btclinn} (that is, \eqref{Btclin} with $h=0$ and initial value $v$),
    and recall from \cref{al1} that $A_q:L^2(\Omega)\to L^2(\Omega)$ defined by
    $A_q v:=\psi_v(0)$ is bounded, self-adjoint, and satisfies
    $\langle A_q v,v\rangle_{L^2(\Omega)}\ge 0$ for all $v\in L^2(\Omega)$.

    Define the closed convex cone
    \[
        K:=\Bigl\{\,v\in L^2(\Omega):\ v=0 \text{ a.e.\ on }\{|\varphi(0)|<\mu\},\ \
            v\ge 0 \text{ a.e.\ on }\{\varphi(0)=-\mu\},\ \
        v\le 0 \text{ a.e.\ on }\{\varphi(0)=\mu\}\Bigr\}.
    \]
    Consider the bilinear form $a:L^2(\Omega)\times L^2(\Omega)\to\mathbb R$ given by $a(v,w):=\langle (A_q+\lambda I)v,w\rangle_{L^2(\Omega)}$.
    It is bounded and coercive since, by \cref{al1} \ref{al1:i},
    \[
        a(v,v)=\langle A_q v,v\rangle_{L^2(\Omega)}+\lambda\|v\|_{L^2(\Omega)}^2
        \ge \lambda\|v\|_{L^2(\Omega)}^2\quad\forall v\in L^2(\Omega).
    \]
    Define $\ell\in (L^2(\Omega))^*$ by $\ell(w):=-\langle g,w\rangle_{L^2(\Omega)}$.
    By the Lions--Stampacchia theorem \cite[Theorem 2.1]{LS_1967}, there exists a unique $v\in K$
    such that
    \begin{equation}\label{eq:VI-Theta-surj}
        a(v,r-v)\ge \ell(r-v)\quad \forall r\in K,
    \end{equation}
    i.e., $\big\langle g+(A_q+\lambda I)v,\ r-v\big\rangle_{L^2(\Omega)}\ge 0$ for all $r\in K$. Now set $(z,\psi):=(z^h+z_v,\ \psi^h+\psi_v)$. By construction, $(z,\psi)$ solves the two linear equations in \eqref{Btclin}
    with the given $h$ and initial trace $z(0)=v$. Moreover,
    \[
        \psi(0)=\psi^h(0)+\psi_v(0)=g+A_q v.
    \]
    Define $s:=\psi(0)+\lambda v=g+(A_q+\lambda I)v$. Then \eqref{eq:VI-Theta-surj}
    is equivalent to $-s\in N_K(v)$, where $N_K(v)$ denotes the normal cone of $K$ at $v$.
    Since $K$ is defined by pointwise sign constraints, localization of the normal
    cone condition yields
    \[
        s=0 \ \text{a.e.\ on }\{|\varphi(0)|>\mu\},\qquad
        s\ge 0 \ \text{a.e.\ on }\{\varphi(0)=-\mu\},\qquad
        s\le 0 \ \text{a.e.\ on }\{\varphi(0)=\mu\},
    \]
    and $v\,s=0$ a.e.\ on $\{\varphi(0)=\pm\mu\}$, while $v=0$ a.e.\ on $\{|\varphi(0)|<\mu\}$
    since $v\in K$. These relations are exactly equivalent to trace relation in \eqref{Btclin}; namely,
    on $\{|\varphi(0)|>\mu\}$ we have $s=0$ and hence $v=-\psi(0)/\lambda$; on
    $\{\varphi(0)=-\mu\}$ we have $v\ge0$, $s\ge0$ and $vs=0$, which gives
    $v=-(1/\lambda)\min\{\psi(0),0\}$; and on $\{\varphi(0)=\mu\}$ we have $v\le0$,
    $s\le0$ and $vs=0$, which gives $v=-(1/\lambda)\max\{\psi(0),0\}$.

    Thus $(h,v)\in\Theta$, proving existence. Uniqueness of $v$ for the fixed $h$
    follows from \cref{lem:unique-v}.
\end{proof}

\begin{lemma}\label{aliv}
    Let $(q,u)\in\mathcal C$ be given. Let $y,\varphi\in W(0,T)$ denote its associated state-adjoint pair. Let $B: L^2(0,T)\to L^2(\Omega)$ be given by $B(h):=\psi_{h}(0)$, where
    \begin{equation}\label{Btclinnali}
        \left\{
            \begin{aligned}
                \partial_t z_h-\Delta z_h+q(t)\,z_h&=-hy
                &&\text{in }Q,\\
                z_h&=0
                &&\text{on }\Sigma,\\
                z_h(0)&=0
                &&\text{in }\Omega,
            \end{aligned}
        \right.
        \qquad
        \left\{
            \begin{aligned}
                -\partial_t \psi_h-\Delta \psi_h+q(t)\,\psi_h&=z_h-h\varphi
                &&\text{in }Q,\\
                \psi_h&=0
                &&\text{on }\Sigma,\\
                \psi_h(T)&=0
                &&\text{in }\Omega .
            \end{aligned}
        \right.
    \end{equation}
    Then, the operator $B:L^2(0,T)\to L^2(\Omega)$ is linear and bounded.
\end{lemma}
\begin{proof}
    The result follows from  standard energy estimates.
\end{proof}

\begin{lemma}\label{lemadjoint}
    Let $(q,u)\in\mathcal C$ be given. Let $y,\varphi\in W(0,T)$ denote its associated state-adjoint pair. Let $B:L^2(0,T)\to L^2(\Omega)$ be the mapping in \cref{aliv}. Then its adjoint $B^*:L^2(\Omega)\to L^2(0,T)$ is given by
    \[
        B^*(\pi)(t)=- \int_{\Omega}\big[y(t,x)\hspace{0.02cm}\eta_\pi(t,x) + \varphi(t,x)\hspace{0.02cm}\xi_\pi(t,x)\big]\, dx,
    \]
    where $\xi_\pi,\eta_\pi\in W(0,T)$ are given by
    \begin{equation}\label{Btclinnalinor}
        \left\{
            \begin{aligned}
                \partial_t \xi_\pi-\Delta \xi_\pi+q(t)\,\xi_\pi&=0
                &&\text{in }Q,\\
                \xi_\pi&=0
                &&\text{on }\Sigma,\\
                \xi_\pi(0)&=\pi
                &&\text{in }\Omega,
            \end{aligned}
        \right.
        \qquad
        \left\{
            \begin{aligned}
                -\partial_t \eta_\pi-\Delta \eta_\pi+q(t)\,\eta_\pi&=\xi_\pi
                &&\text{in }Q,\\
                \eta_\pi&=0
                &&\text{on }\Sigma,\\
                \eta_\pi(T)&=0
                &&\text{in }\Omega .
            \end{aligned}
        \right.
    \end{equation}
\end{lemma}
\begin{proof}
    This follows from standard computations involving the integration by parts formula.
\end{proof}

\begin{lemma}\label{hstar}
    Let $(q,u)\in\mathcal C$ be given. Let $y,\varphi\in W(0,T)$ denote its associated state-adjoint pair. If $(h^*,v^*)\in N_{\mathcal C}(q,u)$, then there exists $\pi\in L^2(\Omega)$ such that
    \[
        h^*(t) = \int_\Omega \big[y(t,x)\eta_\pi(t,x)+\varphi(t,x)\xi_\pi(t,x)\big]\,dx
        \quad\text{for a.e.\ }t\in(0,T),
    \]
    where $\xi_\pi,\eta_\pi\in W(0,T)$ are given by \eqref{Btclinnalinor}.
\end{lemma}
\begin{proof}
    Let $B:L^2(0,T)\to L^2(\Omega)$ be the bounded linear map from \cref{aliv},
    defined by $Bh:=\psi_h(0)$, where $(z_h,\psi_h)$ solves \eqref{Btclinnali}.
    Let $A_q:L^2(\Omega)\to L^2(\Omega)$ be the bounded self-adjoint positive
    operator from \cref{al1}.

    Fix $(h,v)\in T_{\mathcal C}(q,u)$ and let $(z,\psi)\in W(0,T)\times W(0,T)$ be
    the associated pair in \cref{p1}, i.e.\ satisfying \eqref{Btclin}.
    By linearity of the PDEs in \eqref{Btclin}, writing $(z,\psi)=(z^{(v)},\psi^{(v)})+(z^{(h)},\psi^{(h)})$
    with $(z^{(v)},\psi^{(v)})$ solving \eqref{Btclinn} (i.e.\ $h=0$, $z(0)=v$) and
    $(z^{(h)},\psi^{(h)})$ solving \eqref{Btclinnali} (i.e.\ $z(0)=0$), we obtain
    \begin{equation}\label{eq:hstar:decomp}
        \psi(0)=\psi^{(v)}(0)+\psi^{(h)}(0)=A_q v + Bh.
    \end{equation}
    Moreover, the trace relation in \eqref{Btclin} implies pointwise a.e.\ in $\Omega$ that
    $\psi(0)\,v=-\lambda |v|^2$ (this is a direct case-by-case verification),
    and hence
    \begin{equation}\label{eq:hstar:key}
        \langle \psi(0),v\rangle_{L^2(\Omega)}=-\lambda\|v\|_{L^2(\Omega)}^2.
    \end{equation}
    Combining \eqref{eq:hstar:decomp} and \eqref{eq:hstar:key} yields $\langle A_q v,v\rangle_{L^2(\Omega)} + \langle Bh,v\rangle_{L^2(\Omega)}= -\lambda\|v\|_{L^2(\Omega)}^2$. Since $\langle A_q v,v\rangle_{L^2(\Omega)}\ge 0$ by \cref{al1} \ref{al1:i}, we obtain the estimate
    \begin{equation}\label{eq:hstar:vn-bound}
        \|v\|_{L^2(\Omega)} \le \frac{1}{\lambda}\,\|Bh\|_{L^2(\Omega)}.
    \end{equation}
    Now assume $(h^*,v^*)\in N_{\mathcal C}(q,u)$. Then
    \begin{equation}\label{eq:hstar:polar}
        \langle h^*,h\rangle_{L^2(0,T)} + \langle v^*,v\rangle_{L^2(\Omega)} \le 0
        \qquad \forall (h,v)\in T_{\mathcal C}(q,u).
    \end{equation}
    Fix $h\in L^2(0,T)$. Choose any $v$ such that $(h,v)\in T_{\mathcal C}(q,u)$ and any
    $\tilde v$ such that $(-h,\tilde v)\in T_{\mathcal C}(q,u)$ (existence follows from \cref{lem:Theta-surj}).
    Applying \eqref{eq:hstar:polar} to $(h,v)$ and to $(-h,\tilde v)$ and using
    Cauchy--Schwarz together with \eqref{eq:hstar:vn-bound} gives
    \[
        \langle h^*,h\rangle_{L^2(0,T)} \le \frac{\|v^*\|_{L^2(\Omega)}}{\lambda}\,\|Bh\|_{L^2(\Omega)}
        \quad\text{and}\quad
        -\langle h^*,h\rangle_{L^2(0,T)} \le \frac{\|v^*\|_{L^2(\Omega)}}{\lambda}\,\|Bh\|_{L^2(\Omega)}.
    \]
    Therefore,
    \begin{equation}\label{eq:hstar:factor}
        |\langle h^*,h\rangle_{L^2(0,T)}| \le \frac{\|v^*\|_{L^2(\Omega)}}{\lambda}\,\|Bh\|_{L^2(\Omega)}
        \qquad \forall h\in L^2(0,T).
    \end{equation}
    In particular, if $Bh=0$ then $\langle h^*,h\rangle_{L^2(0,T)}=0$.

    Define the linear functional $\ell$ on $\operatorname{ran}B\subseteq L^2(\Omega)$ by $\ell(Bh):=\langle h^*,h\rangle_{L^2(0,T)}$. This is well-defined, if $Bh_1=Bh_2$, then $B(h_1-h_2)=0$, hence $\langle h^*,h_1-h_2\rangle_{L^2(0,T)}=0$
    by the preceding observation. Moreover, \eqref{eq:hstar:factor} implies
    $|\ell(w)|\le \frac{\|v^*\|}{\lambda}\|w\|$ for all $w\in\operatorname{ran}B$, so $\ell$
    is continuous on $\operatorname{ran}B$. By Hahn--Banach, $\ell$ extends to a bounded
    linear functional on $L^2(\Omega)$, hence by the Riesz representation theorem there exists $\pi\in L^2(\Omega)$ such that $\ell(w)=-\langle w,\pi\rangle_{L^2(\Omega)}$ for all  $w\in L^2(\Omega)$. In particular, for every $h\in L^2(0,T)$,
    \[
        \langle h^*,h\rangle_{L^2(0,T)}=\ell(Bh)=\langle Bh,-\pi\rangle_{L^2(\Omega)}=-\langle h,B^*\pi\rangle_{L^2(0,T)},
    \]
    so $h^*=-B^*\pi$ in $L^2(0,T)$. The explicit representation of $B^*\pi$ is given by
    \cref{lemadjoint}, i.e.
    \[
        h^*(t)=-(B^*\pi)(t)=\int_\Omega\bigl(y(t,x)\,\eta_\pi(t,x)+\varphi(t,x)\,\xi_\pi(t,x)\bigr)\,dx
        \quad\text{for a.e.\ }t\in(0,T),
    \]
    where $(\xi_\pi,\eta_\pi)$ solves \eqref{Btclinnalinor}. This proves the claim.
\end{proof}

\begin{lemma}\label{lem:polar-characterization}
    Let $(q,u)\in\mathcal C$ be given, and let $(y,\varphi)\in W(0,T)\times W(0,T)$ be
    its associated state--adjoint pair.
    Then $(h^*,v^*)\in N_{\mathcal C}(q,u)$ if and only if there exists $\pi\in L^2(\Omega)$ such that
    \[
        h^*= \int_\Omega \big[y(\cdot,x)\eta(\cdot,x)+\varphi(\cdot,x)\xi(\cdot,x)\big]\,dx,
    \]
    and for every $(h,v)\in T_{\mathcal C}(q,u)$, writing $\psi\in W(0,T)$ for the adjoint component
    associated with $(h,v)$ through \eqref{Btclin}, one has
    \[
        \int_{\Omega}\Bigl[v^*(x)+\eta(0,x)+\lambda\,\pi(x)\Bigr]\,v(x)\,dx
        -\int_{\Omega}\Bigl[\psi(0,x)+\lambda\,v(x)\Bigr]\,\pi(x)\,dx \le 0,
    \]
    where $(\xi,\eta)\in W(0,T)\times W(0,T)$ is the unique solution of \eqref{Btclinnalinor}.
\end{lemma}

\begin{proof}
    Let $A_q:L^2(\Omega)\to L^2(\Omega)$ be the operator from \cref{al1}, i.e.,
    $A_q v:=\psi^{(v)}(0)$ where $(z^{(v)},\psi^{(v)})$ solves \eqref{Btclinn}, and let
    $B:L^2(0,T)\to L^2(\Omega)$ be the operator from \cref{aliv}, i.e.,
    $Bh:=\psi^{(h)}(0)$ where $(z^{(h)},\psi^{(h)})$ solves \eqref{Btclinnali}.

    Fix $(h,v)\in T_{\mathcal C}(q,u)$ and let $(z,\psi)\in W(0,T)\times W(0,T)$ be the associated
    pair from \cref{p1}. By linearity of \eqref{Btclin} we may write
    \[
        (z,\psi)=(z^{(v)},\psi^{(v)})+(z^{(h)},\psi^{(h)}),
    \]
    where $(z^{(v)},\psi^{(v)})$ solves \eqref{Btclinn} (i.e.\ $h=0$, $z^{(v)}(0)=v$) and
    $(z^{(h)},\psi^{(h)})$ solves \eqref{Btclinnali} (i.e.\ $z^{(h)}(0)=0$). Taking traces at $t=0$
    yields
    \begin{equation}\label{eq:psi0-decomp}
        \psi(0)=A_q v + Bh.
    \end{equation}
    Given $\pi\in L^2(\Omega)$, let $(\xi,\eta)\in W(0,T)\times W(0,T)$ be the unique solution of
    \eqref{Btclinnalinor}. Since \eqref{Btclinnalinor} coincides with \eqref{Btclinn} with initial datum
    $v=\pi$, we have $ \eta(0)=A_q\pi$. By \cref{al1} \ref{al1:i}, $A_q$ is self-adjoint, hence
    \begin{equation}\label{eq:selfadj}
        \langle A_q v,\pi\rangle _{L^2(\Omega)}=\langle v,A_q\pi\rangle_{L^2(\Omega)}=\langle v,\eta(0)\rangle_{L^2(\Omega)}\quad v,\pi\in L^2(\Omega).
    \end{equation}
    We now divide the rest of the proof into two steps.
    \paragraph{Step 1 (necessity).}
    Assume $(h^*,v^*)\in N_{\mathcal C}(q,u)$. By \cref{hstar} there exists $\pi\in L^2(\Omega)$
    such that $h^*=-B^*\pi$, i.e.
    \[
        \langle h^*,h\rangle_{L^2(0,T)}=-\langle Bh,\pi\rangle_{L^2(\Omega)} \quad \forall h\in L^2(0,T).
    \]
    Fix $(h,v)\in T_{\mathcal C}(q,u)$ and let $\psi$ be its associated adjoint component. Polarity gives
    \[
        0\ge \langle h^*,h\rangle_{L^2(0,T)}+\langle v^*,v\rangle_{L^2(\Omega)}
        =-\langle Bh,\pi\rangle_{L^2(\Omega)}+\langle v^*,v\rangle_{L^2(\Omega)}.
    \]
    Using \eqref{eq:psi0-decomp} we write $Bh=\psi(0)-A_q v$, so
    \[
        0\ge -\langle \psi(0),\pi\rangle_{L^2(\Omega)}+\langle A_q v,\pi\rangle_{L^2(\Omega)}+\langle v^*,v\rangle_{L^2(\Omega)}.
    \]
    Invoking \eqref{eq:selfadj} yields $0\ge \langle v^*+\eta(0),v\rangle_{L^2(\Omega)}-\langle \psi(0),\pi\rangle_{L^2(\Omega)}$.
    Finally, since $\langle \pi,v\rangle_{L^2(\Omega)}=\langle v,\pi\rangle_{L^2(\Omega)}$, we may add the identically vanishing term
    $\lambda\langle \pi,v\rangle_{L^2(\Omega)}-\lambda\langle v,\pi\rangle_{L^2(\Omega)}=0$ to obtain
    \[
        \int_{\Omega}\bigl[v^*(x)+\eta(0,x)+\lambda\,\pi(x)\bigr]\,v(x)\,dx
        -\int_{\Omega}\bigl[\psi(0,x)+\lambda\,v(x)\bigr]\,\pi(x)\,dx \le 0.
    \]

    \paragraph{Step 2 (sufficiency).}
    Conversely, assume there exists \(\pi\in L^2(\Omega)\) satisfying the stated
    representation of \(h^*\) and the stated inequality. By \cref{lemadjoint}, the
    representation is equivalent to \(h^*=-B^*\pi\). Fix $(h,v)\in T_{\mathcal C}(q,u)$,
    and let $\psi$ be the associated adjoint component in \eqref{Btclin}. Using
    \(h^*=-B^*\pi\), \eqref{eq:psi0-decomp}, and \eqref{eq:selfadj}, we see that
    \[
        \begin{aligned}
            \langle h^*,h\rangle_{L^2(0,T)} + \langle v^*,v\rangle_{L^2(\Omega)}
            &= -\langle Bh,\pi\rangle_{L^2(\Omega)} + \langle v^*,v\rangle_{L^2(\Omega)} \\
            &= -\langle \psi(0)-A_q v,\pi\rangle_{L^2(\Omega)} + \langle v^*,v\rangle_{L^2(\Omega)} \\
            &= \langle v^*+\eta(0),v\rangle_{L^2(\Omega)} - \langle \psi(0),\pi\rangle_{L^2(\Omega)} .
        \end{aligned}
    \]
    On the other hand, the assumed inequality differs from the right-hand side only by
    $\lambda\langle \pi,v\rangle_{L^2(\Omega)}-\lambda\langle v,\pi\rangle_{L^2(\Omega)}=0$, hence it implies
    \[
        \langle h^*,h\rangle_{L^2(0,T)}+\langle v^*,v\rangle_{L^2(\Omega)}\le 0 \quad \forall (h,v)\in T_{\mathcal C}(q,u).
    \]
    Therefore $(h^*,v^*)\in N_{\mathcal C}(q,u)$. This completes the proof.
\end{proof}

\subsubsection{First inclusion}
Let $(h^*,v^*)\in N_{\mathcal C}(q,u)$ be given. By \cref{lem:polar-characterization}, there exists \(\pi\in L^2(\Omega)\) such that, if
\((\xi,\eta)\in W(0,T)\times W(0,T)\) is the corresponding solution of
\eqref{Btclinnalinor}, then
\[
    h^*(t)=\int_\Omega\bigl(y(t,x)\eta(t,x)+\varphi(t,x)\xi(t,x)\bigr)\,dx
    \quad\text{for a.e. }t\in(0,T),
\]
and, for every \((h,v)\in T_{\mathcal C}(q,u)\), with associated adjoint
component \(\psi\),
\[
    \langle v^*+\eta(0)+\lambda\pi,v\rangle_{L^2(\Omega)}
    -
    \langle \psi(0)+\lambda v,\pi\rangle_{L^2(\Omega)}
    \le 0.
\]
We set $\rho:=v^*+\eta(0)+\lambda\pi $. Then \(v^*=\rho-\lambda\pi-\eta(0)\), and it remains to prove the stated
pointwise conditions on \(\pi\) and \(\rho\).

Let $A_q:L^2(\Omega)\to L^2(\Omega)$ be the operator from \cref{al1}, i.e.,
$A_q v:=\psi^{(v)}(0)$ where $(z^{(v)},\psi^{(v)})$ solves \eqref{Btclinn}, and let
$B:L^2(0,T)\to L^2(\Omega)$ be the operator from \cref{aliv}, i.e.,
$Bh:=\psi^{(h)}(0)$ where $(z^{(h)},\psi^{(h)})$ solves \eqref{Btclinnali}. Consider also the closed convex cone
\[
    K:=\Bigl\{\,v\in L^2(\Omega):\ v=0 \ \text{a.e.\ on }\{|\varphi(0)|<\mu\},\ \
        v\ge 0 \ \text{a.e.\ on }\{\varphi(0)=-\mu\},\ \
    v\le 0 \ \text{a.e.\ on }\{\varphi(0)=\mu\}\Bigr\}.
\]
We will divide the proof in four steps.
\paragraph{Step 1 (some useful operators).}
Consider the operator $\mathbf{v}:L^2(\Omega)\to L^2(\Omega)$ given by $\mathbf{v}(g):=v_g$, where $v_g\in K$ is the unique solution of the variational inequality
\begin{equation}\label{eq:VI-vg}
    \big\langle g+(A_q+\lambda I)v_g,\, r-v_g\big\rangle_{L^2(\Omega)}\ge 0
    \quad \forall r\in K.
\end{equation}
The Lions--Stampacchia theorem and its proof \cite[Theorem 2.1]{LS_1967} ensures that the operator $\mathbf{v}:L^2(\Omega)\to L^2(\Omega)$ is well-defined and Lipschitz continuous.
Define now the operator $\mathbf{s}:L^2(\Omega)\to L^2(\Omega)$ given by $\mathbf s(g):=g + (A_q +\lambda I)\mathbf v(g)$. Simple estimates show that  $\mathbf{s}:L^2(\Omega)\to L^2(\Omega)$ is also Lipschitz continuous.

\paragraph{Step 2 (controllability and observability issues).}
By \cref{obsass} and \cref{lemadjoint}, we obtain that the mapping $B^*:L^2(\Omega)\to L^2(0,T)$ is injective, in particular $\ker B^*=\{0\}$. This implies
\[
    \overline{\operatorname{ran} B} = \big(\operatorname{ker} B^*\big)^\perp = \{0\}^\perp = L^2(\Omega).
\]
We now use this density to extend the polar inequality to arbitrary data
\(g\in L^2(\Omega)\). Let first \(g\in \operatorname{ran} B\), and choose
\(h\in L^2(0,T)\) such that \(g=Bh\). By definition of \(\mathbf v(g)\),
the element \(v:=\mathbf v(g)\) solves
\[
    \big\langle g+(A_q+\lambda I)v,\,r-v\big\rangle_{L^2(\Omega)}\ge 0
    \quad \forall r\in K .
\]
Since \(g=Bh\), the linearized state--adjoint system corresponding to the pair
\((h,v)\) has initial adjoint trace
\[
    \psi(0)=A_qv+Bh=A_qv+g .
\]
Consequently, $\psi(0)+\lambda v
=
g+(A_q+\lambda I)v
=
\mathbf s(g)$. Thus \((h,\mathbf v(g))\) is an admissible tangent direction, and
\(\cref{lem:polar-characterization}\) yields the variational inequality
\[
    \langle \rho,\mathbf v(g)\rangle_{L^2(\Omega)}
    -
    \langle \mathbf s(g),\pi\rangle_{L^2(\Omega)}
    \le 0
    \quad \forall g\in \operatorname{ran}B .
\]
Finally, \(\overline{\operatorname{ran}B}=L^2(\Omega)\), while the mappings
\(\mathbf v\) and \(\mathbf s\) are Lipschitz continuous. Hence the map
\[
    g\longmapsto
    \langle \rho,\mathbf v(g)\rangle_{L^2(\Omega)}
    -
    \langle \mathbf s(g),\pi\rangle_{L^2(\Omega)}
\]
is continuous on \(L^2(\Omega)\). Passing to the limit along a sequence
\(g_n\in\operatorname{ran}B\) with \(g_n\to g\) gives
\begin{equation}\label{eq:pi-rho-ineq-allg}
    \langle \rho,\mathbf v(g)\rangle_{L^2(\Omega)}
    -
    \langle \mathbf s(g),\pi\rangle_{L^2(\Omega)}
    \le 0
    \quad \forall g\in L^2(\Omega).
\end{equation}
\paragraph{Step 3 (relations for $\pi$).} Let $w\in L^2(\Omega)$ satisfy
$w=0$ a.e.\ on $\{|\varphi(0)|\ge\mu\}$. Further, for every $r\in K$, we have
$r=0$ a.e.\ on $\{|\varphi(0)|<\mu\}$, and therefore
\[
    \langle w+(A_q+\lambda I)\cdot 0,\ r-0\rangle_{L^2(\Omega)}=\langle w,r\rangle_{L^2(\Omega)}=0.
\]
Hence $v=0$ satisfies \eqref{eq:VI-vg} with datum $g=w$; by uniqueness of the
variational inequality solution, $\mathbf v(w)=0$, and therefore $\mathbf s(w)=w$. Inserting
$g=w$ into \eqref{eq:pi-rho-ineq-allg} gives $-\langle w,\pi\rangle_{L^2(\Omega)}\le 0$.
Repeating the same argument with $-w$ yields $\langle w,\pi\rangle_{L^2(\Omega)}\le 0$, hence
$\langle w,\pi\rangle_{L^2(\Omega)}=0$ for all such $w$, which implies $\pi=0$ a.e.\ on
$\{|\varphi(0)|<\mu\}$. Next, let $w\in L^2(\Omega)$ satisfy $w\ge 0$ a.e.\ and $w=0$ a.e.\ on
$\{\varphi(0)\neq -\mu\}$. Further, for every $r\in K$ we have $r\ge 0$ a.e.\ on
$\{\varphi(0)=-\mu\}$, and therefore $\langle w,r\rangle_{L^2(\Omega)}\ge 0$. As above, this
implies $\mathbf{v}(w)=0$ and $\mathbf{s}(w)=w$, so \eqref{eq:pi-rho-ineq-allg} yields
$\langle w,\pi\rangle_{L^2(\Omega)}\ge 0$ for all such $w$, which is equivalent to
$\pi\ge 0$ a.e.\ on $\{\varphi(0)=-\mu\}$. Similarly, if $w\le 0$ a.e.\ and
$w=0$ a.e.\ on $\{\varphi(0)\neq \mu\}$, then $r\le 0$ a.e.\ on $\{\varphi(0)=\mu\}$
for all $r\in K$, hence $\langle w,r\rangle_{L^2(\Omega)}\ge 0$, and the same argument shows
$\langle w,\pi\rangle_{L^2(\Omega)}\ge 0$ for all such $w$, which is equivalent to
$\pi\le 0$ a.e.\ on $\{\varphi(0)=\mu\}$.

\paragraph{Step 4 (relations for $\rho$).}
Let $\bar v\in L^2(\Omega)$ satisfy $\bar v=0$ a.e.\ on
$\{|\varphi(0)|\le \mu\}$. Then $\bar v\in K$. Define $g:=-(A_q+\lambda I)\bar v$.
With this choice, $g+(A_q+\lambda I)\bar v=0$, and hence $\bar v$ satisfies
\eqref{eq:VI-vg} with datum $g$ (with equality). By uniqueness, $\mathbf{v}(g)=\bar v$ and
$\mathbf{s}(g)=0$. Inserting this $g$ into \eqref{eq:pi-rho-ineq-allg} yields
$\langle \rho,\bar v\rangle_{L^2(\Omega)}\le 0$. Repeating the argument with $-\bar v$ gives
$\langle \rho,\bar v\rangle_{L^2(\Omega)}\ge 0$, hence $\langle \rho,\bar v\rangle_{L^2(\Omega)}=0$ for all
$\bar v$ supported in $\{|\varphi(0)|>\mu\}$, which implies $\rho=0$ a.e.\ on
$\{|\varphi(0)|>\mu\}$.
Let $\bar v\in L^2(\Omega)$ satisfy $\bar v\ge 0$ a.e.\ and
$\bar v=0$ a.e.\ on $\{\varphi(0)\neq -\mu\}$. Then $\bar v\in K$, and the same
choice $g:=-(A_q+\lambda I)\bar v$ yields $\mathbf{v}(g)=\bar v$ and $\mathbf{s}(g)=0$, so
\eqref{eq:pi-rho-ineq-allg} gives $\langle \rho,\bar v\rangle_{L^2(\Omega)}\le 0$ for all such
$\bar v$, which is equivalent to $\rho\le 0$ a.e.\ on $\{\varphi(0)=-\mu\}$.
Likewise, if $\bar v\le 0$ a.e.\ and $\bar v=0$ a.e.\ on $\{\varphi(0)\neq \mu\}$,
the same argument yields $\langle \rho,\bar v\rangle_{L^2(\Omega)}\le 0$ for all such $\bar v$,
which is equivalent to $\rho\ge 0$ a.e.\ on $\{\varphi(0)=\mu\}$.

\subsubsection{Second inclusion}
Suppose there exist $\xi,\eta\in W(0,T)$ and $\pi,\rho\in L^2(\Omega)$ satisfying
\[
    \left\{
        \begin{aligned}
            \partial_t\xi-\Delta\xi+q(t)\,\xi &= 0 && \text{in }Q,\\
            \xi &= 0 && \text{on }\Sigma,\\
            \xi(0) &= \pi && \text{in }\Omega,
        \end{aligned}
    \right.
    \qquad
    \left\{
        \begin{aligned}
            -\partial_t\eta-\Delta\eta+q(t)\,\eta &= \xi && \text{in }Q,\\
            \eta &= 0 && \text{on }\Sigma,\\
            \eta(T) &= 0 && \text{in }\Omega,
        \end{aligned}
    \right.
\]
such that
\[
    v^*=\rho-\lambda\,\pi-\eta(0)
    \quad\text{and}\quad
    h^*(t)=\int_\Omega \bigl(y(t,x)\,\eta(t,x)+\varphi(t,x)\,\xi(t,x)\bigr)\,dx,
\]
where $\pi$ and $\rho$ satisfy
\[
    \pi=0 \ \text{a.e.\ on } \{|\varphi(0)|<\mu\},\qquad
    \pi\ge 0 \ \text{a.e.\ on } \{\varphi(0)=-\mu\},\qquad
    \pi\le 0 \ \text{a.e.\ on } \{\varphi(0)=\mu\},
\]
\[
    \rho=0 \ \text{a.e.\ on } \{|\varphi(0)|>\mu\},\qquad
    \rho\le 0 \ \text{a.e.\ on } \{\varphi(0)=-\mu\},\qquad
    \rho\ge 0 \ \text{a.e.\ on } \{\varphi(0)=\mu\}.
\]
We prove below that $(h^*,v^*)$ belongs to ${N}_{\mathcal C}(q,u)$.

Let $(h,v)\in T_{\mathcal C}(q,u)$ be arbitrary and let $(z,\psi)\in W(0,T)\times W(0,T)$
be associated with $(h,v)$ through \eqref{Btclin}.
Recall the operators $A_q:L^2(\Omega)\to L^2(\Omega)$ and $B:L^2(0,T)\to L^2(\Omega)$
from \cref{al1,aliv}; in particular, $\psi(0)=A_q v + Bh$. Since $(\xi,\eta)$ solves \eqref{Btclinnalinor} with initial datum $\pi$, \cref{lemadjoint} yields
$h^*=-B^*\pi$, i.e., $\langle h^*,h\rangle_{L^2(0,T)}=-\langle Bh,\pi\rangle_{L^2(\Omega)}$.
Using that $\psi(0)=A_q v + Bh$, we obtain
\[
    \langle h^*,h\rangle_{L^2(0,T)}=-\langle \psi(0)-A_q v,\pi\rangle_{L^2(\Omega)}=-\langle \psi(0),\pi\rangle_{L^2(\Omega)}+\langle A_q v,\pi\rangle_{L^2(\Omega)}.
\]
Since $\eta(0)=A_q\pi$ and $A_q$ is self-adjoint (\cref{al1}),
\[
    \langle A_q v,\pi\rangle_{L^2(\Omega)}=\langle v,A_q\pi\rangle_{L^2(\Omega)}=\langle v,\eta(0)\rangle_{L^2(\Omega)},
\]
and thus
\[
    \langle h^*,h\rangle_{L^2(0,T)}+\langle v^*,v\rangle_{L^2(\Omega)}=\langle v^*+\eta(0),v\rangle_{L^2(\Omega)}-\langle \psi(0),\pi\rangle_{L^2(\Omega)}.
\]
Therefore, we have the identity
\begin{equation}\label{eq:second-inc-key}
    \langle h^*,h\rangle_{L^2(0,T)}+\langle v^*,v\rangle_{L^2(\Omega)}=\langle \rho,v\rangle_{L^2(\Omega)}-\langle \psi(0)+\lambda v,\pi\rangle_{L^2(\Omega)}.
\end{equation}
It remains to show that the right-hand side of \eqref{eq:second-inc-key} is $\le 0$.
This follows from \eqref{Btclin} and the sign conditions on $\pi,\rho$, pointwise a.e.\ in $\Omega$.
On $\{|\varphi(0)|<\mu\}$ we have $v=0$ by \eqref{Btclin} and $\pi=0$ by assumption, hence
$\rho v-(\psi(0)+\lambda v)\pi=0$ there. On $\{|\varphi(0)|>\mu\}$ we have
$v=-\psi(0)/\lambda$ by \eqref{Btclin}, so $\psi(0)+\lambda v=0$, and also $\rho=0$ by assumption; hence the integrand vanishes there as well. On $\{\varphi(0)=-\mu\}$, \eqref{Btclin} gives
\[
    v=-\frac1\lambda \min\{\psi(0),0\}\ge 0\quad\text{and}\quad
    \psi(0)+\lambda v=\max\{\psi(0),0\}\ge 0.
\]
Since $\rho\le 0$ and $\pi\ge 0$ a.e.\ on $\{\varphi(0)=-\mu\}$, it follows that
$\rho v\le 0$ and $-(\psi(0)+\lambda v)\pi\le 0$, hence
$\rho v-(\psi(0)+\lambda v)\pi\le 0$ there. On $\{\varphi(0)=\mu\}$, \eqref{Btclin} gives
\[
    v=-\frac1\lambda \max\{\psi(0),0\}\le 0,
    \quad\text{and}\quad
    \psi(0)+\lambda v=\min\{\psi(0),0\}\le 0,
\]
and since $\rho\ge 0$ and $\pi\le 0$ a.e.\ on $\{\varphi(0)=\mu\}$, we again have
$\rho v\le 0$ and $-(\psi(0)+\lambda v)\pi\le 0$, hence the integrand is $\le 0$ there. Consequently, $\rho v-(\psi(0)+\lambda v)\pi\le 0$ a.e.\ in $\Omega$, and integrating yields $\langle \rho,v\rangle_{L^2(\Omega)}-\langle \psi(0)+\lambda v,\pi\rangle_{L^2(\Omega)}\le 0$. By \eqref{eq:second-inc-key}, this shows that $(h^*,v^*)\in N_{\mathcal C}(q,u)$.

\subsection{Proof of Theorem \ref{thmfailobs}}
We prove the claim by constructing a situation in which the observability
condition fails and the normal-cone representation from \cref{Boucon} cannot
hold. The construction produces nontrivial elements in the kernel of the
observability operator and exploits the resulting symmetry of the tangent
cone to identify an element of the polar cone. We then show that this polar
element cannot be represented by any adjoint multiplier system satisfying the
conditions in \cref{Boucon}.

\subsubsection{A technical lemma}
We first prove a technical lemma that isolates the main obstruction caused by
the failure of the observability condition.

\begin{figure}[t]
    \centering
    \includegraphics[width=\linewidth]{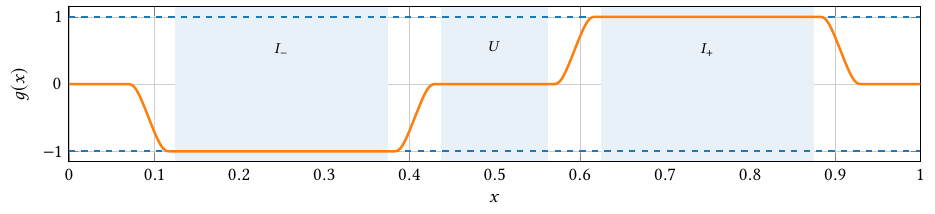}
    \caption{Numerical illustration of the construction in \cref{conlem}: A function $g$ with plateaus on $I_-$ and $I_+$ at $\pm\mu = \pm 1$ and on $U=[7/16,9/16]$ at $0$.}\label{fig:conlem:g}
\end{figure}
\begin{lemma}\label{conlem}
    Let $\mu>0$ and set
    \[
        I_-:=\Bigl[\frac18,\frac38\Bigr],\qquad
        I_+:=\Bigl[\frac58,\frac78\Bigr],\qquad
        J:=(0,1)\setminus (I_-\cup I_+).
    \]
    Consider the sequence $(e_k)_{k\in\mathbb N}$ of functions $e_k:(0,1)\to\mathbb R$ defined by $e_k(x):=\sqrt{2}\sin(k\pi x)$.
    There exists a function $g\in H^2(0,1)\cap H_0^1(0,1)$ such that
    \begin{enumerate}[label=(\roman*)]
        \item $g(1-x)=-g(x)$ for all $x\in(0,1)$;
        \item $g(x)=-\mu$ for all $x\in I_-$ and $g(x)=\mu$ for all $x\in I_+$;
        \item $|g(x)|<\mu$ for all $x\in J$;
        \item $\langle g,e_{2k-1}\rangle_{L^2(0,1)}=0$ and $\langle g,e_{2k}\rangle_{L^2(0,1)}\neq0$ for all $k\in\mathbb N$.
    \end{enumerate}
\end{lemma}
\begin{proof}
    Choose a function $\theta\in C^\infty([0,\frac12])$  such that
    \[
        0\le \theta\le 1,\qquad
        \theta=1 \text{ on } \Bigl[\frac18,\frac38\Bigr],\qquad
        \theta<1 \text{ on } \Bigl[0,\frac12\Bigr]\setminus \Bigl[\frac18,\frac38\Bigr],
        \qquad
        \theta=0 \text{ in a neighborhood of } \Bigl\{0,\frac12\Bigr\}.
    \]
    Define $g_0\in C^\infty([0,1])$ by
    \[
        g_0(x):=
        \begin{cases}
            -\mu\,\theta(x)& \text{if } x\in[0,\frac12],\\
            \mu \theta(1-x)& \text{if } x\in[\frac12,1].
        \end{cases}
    \]
    Observe that $g_0(0)=g_0(1)=0$, and hence $g_0\in H_0^1(0,1)\cap H^2(0,1)$.
    By construction,
    \[
        g_0(1-x)=-g_0(x)\qquad \forall x\in(0,1),
    \]
    therefore $g_0=-\mu$ on $I_-$ and $g_0=\mu$ on $I_+$.
    Moreover, $|g_0|<\mu$ on $J$. Choose now a nonempty open interval $U:=(1/2-\eta,1/2+\eta)$ so that $U\subseteq J$.
    Since $g_0$ is continuous and $|g_0|<\mu$ on the compact set $\overline U$,
    there exists $\delta>0$ such that $|g_0(x)|\le \mu-2\delta$ for all $x\in U$.
    Set
    \[
        X:=\Bigl\{r\in C^\infty(0,1): \overline{\operatorname{supp}}\,r\subseteq U\,\,\,\, \text{and}\,\,\,\,r(1-x)=-r(x)\ \forall x\in(0,1)\Bigr\}.
    \]
    For each $k\in\mathbb N$, define the linear functional $L_k:X\to\mathbb R$ by $L_k(r):=\langle r,e_{2k}\rangle_{L^2(0,1)}$.
    We claim that the family $(L_k)_{k\in\mathbb N}$ is linearly independent on $X$; let $m\in\mathbb N$ and suppose that
    \[
        \sum_{k=1}^m a_k L_k(\varphi)=0 \quad\forall \varphi\in X
    \]
    for some scalars $a_1,\dots,a_m$. Consider $F:(0,1)\to\mathbb R$ given by $F(x):=\sum_{k=1}^m a_k e_{2k}(x)$.  Then $ \int_0^1 r(x)F(x)\,dx=0$ for all $r\in X$.
    Since each $e_{2k}$ is midpoint-odd, so is $F$. Let $\varphi\in C^\infty(0,1)$ with $\overline{\operatorname{supp}}\hspace{0.04cm}\varphi\subseteq U\cap(0,\tfrac12)$ be
    arbitrary, and define $r_{\varphi}:(0,1)\to\mathbb R$ by
    \[
        r_\varphi(x):=\varphi(x)-\varphi(1-x).
    \]
    Then $r_\varphi\in X$. Using a change of variables and the midpoint-oddness of $F$, we get
    \[
        \int_0^1 r_\varphi(x)F(x)\,dx
        =
        \int_0^1 \varphi(x)F(x)\,dx-\int_0^1 \varphi(1-x)F(x)\,dx
        =
        2\int_0^1 \varphi(x)F(x)\,dx.
    \]
    Since $r_\varphi\in X$, the left-hand side is zero. Hence
    \[
        \int_0^1 \varphi(x)F(x)\,dx=0
        \quad  \text{for all $\varphi\in C^\infty(0,1)$ with $\overline{\operatorname{supp}}\hspace{0.04cm}\varphi\subseteq U\cap(0,\tfrac12)$}.
    \]
    Therefore $F=0$ almost everywhere on the nonempty open set
    $U\cap(0,\tfrac12)$; but $F$ is real analytic on $(0,1)$, being a finite linear
    combination of sines, and so $F\equiv0$ on $(0,1)$. Since
    $\{e_n\}_{n\ge1}$ is an orthonormal basis of  $L^2(0,1)$, it follows that $a_k=0$ for
    all $k=1,\dots,m$. This proves the claim.

    Due to the linear independence of $(L_k)_{k\in\mathbb N}$, for each $k\in\mathbb N$ we can choose $r_k\in X$ such that
    \[
        \langle r_k,e_{2k}\rangle_{L^2(0,1)}=1\quad\text{and}
        \quad
        \langle r_k,e_{2\ell}\rangle_{L^2(0,1)}=0 \quad \forall \ell\in\{1,\dots,k-1\}.
    \]
    We now choose coefficients $\varepsilon_k\in\mathbb R$ recursively. Suppose
    that $\varepsilon_1,\dots,\varepsilon_{k-1}$ have already been chosen. Set $c_k:=\Big\langle g_0+\sum_{j=1}^{k-1}\varepsilon_j r_j,\ e_{2k}\Big\rangle_{L^2(0,1)}$.  Since $\langle r_k,e_{2k}\rangle_{L^2(0,1)}=1$, we have
    \[
        \Bigl\langle g_0+\sum_{j=1}^{k-1}\varepsilon_j r_j+\varepsilon_k r_k,\ e_{2k}\Bigr\rangle_{L^2(0,1)}
        =c_k+\varepsilon_k.
    \]
    We choose $\varepsilon_k$ such that $\varepsilon_k\neq -c_k$
    and, simultaneously,
    \[
        |\varepsilon_k|\,\|r_k\|_{H^2(0,1)}\le 2^{-k}\quad\text{and}\quad
        |\varepsilon_k|\,\|r_k\|_{L^\infty(0,1)}\le \delta\,2^{-k}.
    \]
    Define $g:=g_0+\sum_{k=1}^\infty \varepsilon_k r_k$; see \cref{fig:conlem:g} for an illustration of this function.
    By construction, the series converges absolutely in $H^2(0,1)$. Since each partial sum
    belongs to $H_0^1(0,1)$, and $H_0^1(0,1)$ is closed in $H^1(0,1)$, we obtain $g\in H^2(0,1)\cap H_0^1(0,1)$.
    Moreover, as every $r_k$ is midpoint-odd, so is $g$. Since each $r_k$ is supported in $U$, we have $r_k=0$ on
    $I_-\cup I_+$. Therefore
    \[
        g=g_0=-\mu \,\text{ on }I_-\quad\text{and}\quad
        g=g_0=\mu \,\text{ on }I_+.
    \]
    Also, on $J\setminus U$ one has $g=g_0$, and so $|g|<\mu$ there. On the other
    hand, for $x\in U$,
    \[
        |g(x)-g_0(x)|
        \le \sum_{k=1}^\infty |\varepsilon_k|\,\|r_k\|_{L^\infty(0,1)}
        \le \delta,
    \]
    hence $|g(x)|
    \le |g_0(x)|+\delta
    \le (\mu-2\delta)+\delta
    =\mu-\delta<\mu$.
    Thus $|g|<\mu$ on $J$.  Let now $m\in\mathbb N$ be fixed. By
    construction, $\langle r_k,e_{2m}\rangle_{L^2(0,1)}=0$ for $k>m$. The defining property of $r_k$ gives $\langle g,e_{2m}\rangle_{L^2(0,1)}
    =
    \Bigl\langle g_0+\sum_{k=1}^m \varepsilon_k r_k,\ e_{2m}\Bigr\rangle_{L^2(0,1)}$. But when $\varepsilon_m$ was chosen, it was ensured that $\Big\langle g_0+\sum_{k=1}^{m-1} \varepsilon_k r_k+\varepsilon_m r_m,\ e_{2m}\Big\rangle_{L^2(0,1)}
    \neq 0$. Hence $\langle g,e_{2m}\rangle_{L^2(0,1)}\neq0$. Since each $e_{2m-1}$ is midpoint-even, using the change of variables formula yields
    \[
        \langle g,e_{2m-1}\rangle_{L^2(0,1)}
        =\int_0^1 g(x)e_{2m-1}(x)\,dx
        =\int_0^1 g(1-x)e_{2m-1}(1-x)\,dx
        =-\int_0^1 g(x)e_{2m-1}(x)\,dx.
    \]
    Hence \(\langle g,e_{2m-1}\rangle_{L^2(0,1)}=0\). This completes the proof.
\end{proof}

\subsubsection{The construction}\label{sec:conlem:construct}
\begin{figure}[t]
    \centering
    \includegraphics[width=\linewidth]{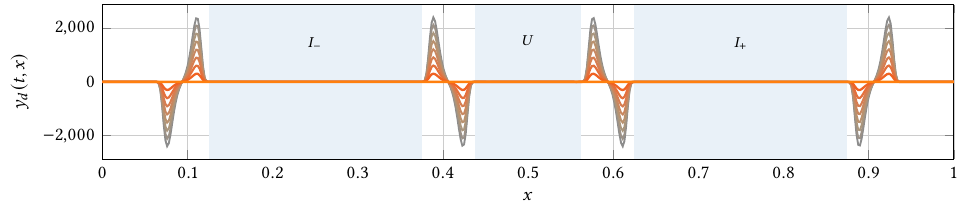}
    \caption{Numerical illustration of the construction in \cref{sec:conlem:construct}: Time snapshots of $y_d(t,x)$ from early (dark gray) to late (orange).}\label{fig:conlem:yd}
\end{figure}
Let $g\in H^2(0,1)\cap H_0^1(0,1)$ be a function satisfying the properties of \cref{conlem}. Set $\mu>0$ and
\[
    \Omega=(0,1),\qquad T=1,\qquad \lambda=1,\qquad y_d(t)=(1-t)g''-g,
\]
see \cref{fig:conlem:yd}.
Consider the point $(q,u)\equiv(0,0)$ in $L^2(0,T)\times L^2(\Omega)$. It is straightforward to verify that $(q,u)\in\mathcal C$,  and that the associated state-adjoint pair satisfies
\[
    y(t,x) = 0\quad\text{and}\quad \varphi(t,x) = (1-t) g(x)\quad \forall (t,x)\in[0,T]\times\Omega.
\]
In particular, we see that
\[
    \{\varphi(0)=-\mu\}=\Bigl[\frac18,\frac38\Bigr]
    \quad\text{and}\quad
    \{\varphi(0)=\mu\}=\Bigl[\frac58,\frac78\Bigr].
\]
Let $(h^*,v^*) \in L^2(0,T) \times L^2(\Omega)$ be given by $h^* \equiv 0$ and $v^* \equiv 1$. We will prove below that $(h^*,v^*) \in N_{\mathcal C}(q,u)$, that \cref{obsass} fails, and that there do not exist $\pi,\rho \in L^2(\Omega)$ satisfying the representation in \cref{Boucon}. The rest of the proof is divided into four steps.

\paragraph{Step 1 (the kernel of the observability operator).} We prove that
\[
    \operatorname{ker}\mathcal O
    =
    \Bigl\{\pi\in L^2(0,1): \pi(1-x)=\pi(x)\ \text{for a.e. }x\in(0,1)\Bigr\}.
\]
From this it follows that \cref{obsass} fails. Let $\pi\in L^2(0,1)$ be arbitrary, and consider its expansion in the Dirichlet eigenbasis
\[
    \pi=\sum_{n=1}^\infty a_n e_n
    \quad\text{in }L^2(0,1),\quad\text{where} \,\, e_n(x):=\sqrt2\sin(n\pi x)\quad \forall n\in\mathbb N.
\]
The solution of the heat equation with initial datum $\pi$ is given by
\[
    \xi(t)=\sum_{n=1}^\infty a_n e^{-n^2\pi^2 t}e_n
    \quad\text{in }L^2(0,1).
\]
Since $\varphi(t,x)=(1-t)g(x)$, we obtain
\[
    \mathcal O(\pi)(t)
    =
    (1-t)\int_0^1 g(x)\,\xi(t,x)\,dx
    =
    (1-t)\sum_{n=1}^\infty a_n \langle g,e_n\rangle_{L^2(0,1)}\,e^{-n^2\pi^2 t}.
\]
Since $g$ is midpoint-odd and $e_{2k-1}$ is midpoint-even, the product $g\,e_{2k-1}$ is
midpoint-odd, and therefore $\langle g,e_{2k-1}\rangle_{L^2(0,1)}=0$ for all $k\in\mathbb N$. Thus we may write
\begin{align}\label{postak}
    \mathcal O(\pi)(t)
    =
    (1-t)\sum_{k=1}^\infty a_{2k}\langle g,e_{2k}\rangle_{L^2(0,1)}\,e^{-4k^2\pi^2 t}.
\end{align}
We first prove that every midpoint-even function belongs to $\ker\mathcal O$.
Assume that $\pi(1-x)=\pi(x)$ for a.e. $x\in(0,1)$. Then $a_{2k}=0$ for all $k\in\mathbb N$. Hence the right-hand side of \eqref{postak} is
identically zero, and so $\mathcal O(\pi)=0$.
Therefore
\[
    \Bigl\{\pi\in L^2(0,1): \pi(1-x)=\pi(x)\ \text{a.e.}\Bigr\}
    \subseteq \operatorname{ker}\mathcal O.
\]
Suppose now that  $\pi\in\ker\mathcal O$. Then the
right-hand side of \eqref{postak} vanishes, i.e.,
\[
    \sum_{k=1}^\infty a_{2k}\langle g,e_{2k}\rangle_{L^2(0,1)}\,e^{-4k^2\pi^2 t}=0
    \qquad \forall t\in(0,1).
\]
By \cref{conlem}, $\langle g,e_{2k}\rangle_{L^2(0,1)}\neq 0$ for all $k\in\mathbb N$, and hence $a_{2k}=0$ for all $k\in\mathbb N$. Therefore the expansion of $\pi$ contains only the odd Dirichlet modes $e_{2k-1}$, which are midpoint-even, hence $\pi$ is itself midpoint-even.
This proves that
\[
    \operatorname{ker}\mathcal O\subseteq\Bigl\{\pi\in L^2(0,1): \pi(1-x)=\pi(x)\ \text{for a.e. }x\in(0,1)\Bigr\}.
\]
Therefore we conclude the desired result. In particular, we see that $\ker\mathcal O\neq\{0\}$; for example, $e_1$ belongs to
$\ker\mathcal O$. Thus $\mathcal O: L^2(\Omega)\to L^2(0,T)$ is not injective.

\paragraph{Step 2 (structure of the tangent cone).} Let $(h,v)\in T_{\mathcal C}(q,u)$ be given. We prove that $v$ is midpoint-odd, i.e., $v(x)=-v(1-x)$ for a.e. $x\in\Omega$.

Let $z,\psi\in W(0,T)$ be the solutions of \eqref{Btclin}. By the representation of the tangent cone,
\[
    v=
    \begin{cases}
        -\min\{\psi(0),0\} & \text{on }\left[\frac18,\frac38\right],\\
        -\max\{\psi(0),0\} & \text{on }\left[\frac58,\frac78\right],\\
        0 & \text{on } [0,1]\setminus\left(\left[\frac18,\frac38\right]\cup\left[\frac58,\frac78\right]\right).
    \end{cases}
\]
Observe that $z(t)=e^{t\Delta}v$ and
\[
    \psi(0)
    =
    \int_0^1 e^{t\Delta}\bigl(z(t)-h(t)\varphi(t)\bigr)\,dt
    =
    \int_0^1 e^{2t\Delta}v\,dt
    -
    \int_0^1 h(t)(1-t)e^{t\Delta}g\,dt.
\]
Define the linear operators $A:L^2(0,1)\to L^2(0,1)$ and
$K:L^2(0,1)\to L^2(0,1)$ by
\[
    Av:=\int_0^1 e^{2t\Delta}v\,dt,
    \qquad
    Kh:=\int_0^1 h(t)(1-t)e^{t\Delta}g\,dt.
\]
Then $\psi(0)=Av-Kh$. Next define $M:L^2(0,1)\to L^2(0,1)$ by
\[
    (Mw)(x)=
    \begin{cases}
        -\min\{w(x),0\} & \text{on } \left[\frac18,\frac38\right],\\
        -\max\{w(x),0\} & \text{on } \left[\frac58,\frac78\right],\\
        0 & \text{on } [0,1]\setminus\left(\left[\frac18,\frac38\right]\cup\left[\frac58,\frac78\right]\right).
    \end{cases}
\]
Then $v=M(Av-Kh)$. Thus $v$ is a fixed point of the map
\[
    F_h:L^2(0,1)\to L^2(0,1),\qquad
    F_h(w):=M(Aw-Kh).
\]
We claim that $F_h$ is a strict contraction. Indeed, $M$ is $1$-Lipschitz,
because it is induced pointwise by scalar maps with Lipschitz constant $1$.
Moreover,
\[
    Ae_n=\alpha_n e_n\quad\text{and}
    \quad
    \alpha_n=\int_0^1 e^{-2n^2\pi^2 t}\,dt
    =\frac{1-e^{-2n^2\pi^2}}{2n^2\pi^2}\quad\forall n\in\mathbb N.
\]
Hence $\|A\|=\sup_{n\in\mathbb N}\alpha_n=(2\pi^2)^{-1}(1-e^{-2\pi^2})<1$. Therefore, for all $w_1,w_2\in L^2(0,1)$,
\[
    \|F_h(w_1)-F_h(w_2)\|_{L^2(0,1)}
    \le
    \|A(w_1-w_2)\|_{L^2(0,1)}
    \le
    \|A\|\,\|w_1-w_2\|_{L^2(0,1)},
\]
and so $F_h$ is a strict contraction. In particular, $F_h$ has a unique fixed
point.

Consider the symmetry operator $S:L^2(0,1)\to L^2(0,1)$ given by $(Sw)(x):=-w(1-x)$.
The condition $Sw=w$ is equivalent to saying that $w$ is midpoint-odd.
Let us now verify that $S$ leaves the fixed-point equation invariant. First, the
Dirichlet heat semigroup commutes with reflection, hence $A$ commutes with $S$, i.e., $S(Aw)=A(Sw)$ for all $w\in L^2(0,1)$.
Second, since $g$ is midpoint-odd and $e^{t\Delta}$ preserves this
symmetry, we have $S(Kh)=Kh$ for all $h\in L^2(0,1)$. Third, one checks directly that $S(Mw)=M(Sw)$ for all $w\in L^2(0,1)$.  Since $v$ is a fixed point of $F_h$, we obtain that $v=M(Av-Kh)$. Hence
\[
    Sv
    =
    S\bigl(M(Av-Kh)\bigr)
    =
    M\bigl(S(Av-Kh)\bigr)
    =
    M\bigl(A(Sv)-Kh\bigr)
    =
    F_h(Sv).
\]
Thus $Sv$ is also a fixed point of $F_h$. By uniqueness of the fixed point, we conclude $Sv=v$. Therefore $v$ is midpoint-odd, which proves the claim.

\paragraph{Step 3 (the point $(h^*,v^*)$ belongs to ${N}_{\mathcal C}(q,u)$).}  By
the previous step, we know that $v$ is midpoint-odd.
Therefore
\[
    \int_0^1 v(x)\,dx
    =
    \int_0^1 v(1-x)\,dx
    =
    -\int_0^1 v(x)\,dx,
\]
and hence $\int_0^1 v(x)\,dx=0$. Since $h^*\equiv0$ and $v^*\equiv 1$, it follows that
\[
    \int_0^1 h^*(t)\,h(t)\,dt+\int_0^1 v^*(x)\,v(x)\,dx
    =
    0+\int_0^1 v(x)\,dx
    =
    0\quad\forall (h,v)\in T_{\mathcal C}(q,u).
\]
Whence the claim follows.

\paragraph{Step 4 (failure of the representation).}
Assume, for contradiction, that there exist
$\pi,\rho\in L^2(0,1)$ and $\xi,\eta\in W(0,1)$ satisfying the representation in \cref{Boucon}. In particular,
\[
    v^*=\rho-\pi-\eta(0)\quad\text{in $L^2(\Omega)$}\quad\text{and}\quad
    h^*(t)=\int_0^1 \varphi(t,x)\,\xi(t,x)\,dx
    \quad\text{for a.e. $t\in(0,T)$}.
\]
Since $h^*\equiv0$, the representation formula yields
\[
    \int_0^1 \varphi(t,x)\,\xi(t,x)\,dx=0
    \quad\text{for a.e. }t\in(0,T).
\]
Combining this with the fact that $y\equiv 0$, we obtain that $\mathcal O(\pi)=0$, and hence, by Step 1, $\pi$ is midpoint-even. On the other hand, by the sign conditions,
\[
    \pi\ge 0 \quad\text{a.e. on }\Bigl[\frac18,\frac38\Bigr],
    \qquad
    \pi\le 0 \quad\text{a.e. on }\Bigl[\frac58,\frac78\Bigr],
    \qquad
    \pi=0 \quad\text{a.e. on }(0,1)\setminus\Bigl(\Bigl[\frac18,\frac38\Bigr]\cup\Bigl[\frac58,\frac78\Bigr]\Bigr).
\]
Since $1-x\in\bigl[\frac58,\frac78\bigr]\) for a.e. \(x\in\bigl[\frac18,\frac38\bigr]\),
\[
    \pi(x)\ge 0
    \qquad\text{and}\qquad
    \pi(x)=\pi(1-x)\le 0\quad\text{for a.e. \(x\in\bigl[\frac18,\frac38\bigr]\).}
\]
Hence $\pi=0$ a.e. on $\Bigl[\frac18,\frac38\Bigr]$. By symmetry, also $\pi=0$ almost everywhere on $\Bigl[\frac58,\frac78\Bigr]$. We conclude that $\pi=0$ a.e. in $\Omega$. It follows that $\xi\equiv0$ and $\eta\equiv0$. Consequently, $v^*=\rho-\pi-\eta(0)=\rho$. Since $v^*\equiv 1$, we obtain $\rho=1$ a.e. in $\Omega$. This contradicts the sign condition $\rho\le0$ on $\Bigl[\frac18,\frac38\Bigr]$. The contradiction
shows that no such quadruple $(\pi,\rho,\xi,\eta)$ exists.

Therefore $(h^*,v^*)\in N_{\mathcal C}(q,u)$ but it does not admit the
representation asserted in \cref{Boucon}. This completes the proof.

\subsection{Proof of Proposition \ref{Boucon2}}
We argue in the spirit of the proof of \cref{Boucon}, but use the
Lebesgue nullity of the contact set in place of the observability condition. Under
this assumption the pointwise tangent relation induced by the soft-thresholding
map becomes linear almost everywhere. This allows us to convert the polar
inequality directly into the stated adjoint representation, together with the
corresponding simplified pointwise conditions.

\subsubsection{A technical lemma}
We first isolate the pointwise consequence of the zero-measure contact set that
will be used to pass from the polar inequality to the multiplier
representation.

\begin{lemma}\label{L1-threshold}
    Let $(q,u)\in\mathcal C$ with associated state--adjoint pair
    $(y,\varphi)\in W(0,T)\times W(0,T)$. For any $v^*\in L^2(\Omega)$ there exist
    $\xi,\eta\in W(0,T)$ and $\pi,\rho\in L^2(\Omega)$ such that
    \[
        \left\{
            \begin{aligned}
                \partial_t\xi-\Delta\xi+q(t)\,\xi &= 0
                && \text{in }Q,\\
                \xi &= 0
                && \text{on }\Sigma,\\
                \xi(0) &= \pi
                && \text{in }\Omega,
            \end{aligned}
        \right.
        \qquad
        \left\{
            \begin{aligned}
                -\partial_t\eta-\Delta\eta+q(t)\,\eta &= \xi
                && \text{in }Q,\\
                \eta &= 0
                && \text{on }\Sigma,\\
                \eta(T) &= 0
                && \text{in }\Omega,
            \end{aligned}
        \right.
    \]
    and $\rho=v^*+\lambda\,\pi+\eta(0)$,  where
    $\pi$ and $\rho$ satisfy
    \[
        \pi=0 \ \text{a.e.\ on } \{|\varphi(0)|<\mu\}\quad \text{and}\quad
        \rho=0 \ \text{a.e.\ on } \{|\varphi(0)|>\mu\}.
    \]
\end{lemma}
\begin{proof}
    Define the linear operator $A:L^2(\Omega)\to L^2(\Omega)$ by $A(r):=\eta_r(0)$, where $(\xi_r,\eta_r)\in W(0,T)\times W(0,T)$ is the unique solution of
    \[
        \left\{
            \begin{aligned}
                \partial_t\xi_r-\Delta\xi_r+q(t)\,\xi_r &= 0
                && \text{in }Q,\\
                \xi_r &= 0
                && \text{on }\Sigma,\\
                \xi_r(0) &= r
                && \text{in }\Omega,
            \end{aligned}
        \right.
        \qquad
        \left\{
            \begin{aligned}
                -\partial_t\eta_r-\Delta\eta_r+q(t)\,\eta_r &= \xi_r
                && \text{in }Q,\\
                \eta_r &= 0
                && \text{on }\Sigma,\\
                \eta_r(T) &= 0
                && \text{in }\Omega.
            \end{aligned}
        \right.
    \]
    Integration by parts yields $\langle Ar,r\rangle_{L^2(\Omega)}=\int_0^T \|\xi_r(t)\|_{L^2(\Omega)}^2\,dt\ge 0$ for all $r\in L^2(\Omega)$. Define
    \[
        V:=\Bigl\{\,r\in L^2(\Omega):\ r=0 \text{ a.e.\ on }\{|\varphi(0)|<\mu\}\Bigr\}.
    \]
    Notice that $V$ is a closed linear subspace of $L^2(\Omega)$. Consider the bilinear form $a: V\times V\to\mathbb R$ given by
    \[
        a(r,s):=\lambda\,\langle r,s\rangle_{L^2(\Omega)}+\langle Ar,s\rangle_{L^2(\Omega)}.
    \]
    It is bounded, and coercive since
    \[
        a(r,r)=\lambda\|r\|_{L^2(\Omega)}^2+\langle Ar,r\rangle_{L^2(\Omega)}
        \ge \lambda\|r\|_{L^2(\Omega)}^2
        \quad \forall r\in V.
    \]
    Define $\ell\in (L^2(\Omega))^*\subseteq V^*$ by $\ell(s):=-\langle v^*,s\rangle_{L^2(\Omega)}$. By the Lax--Milgram theorem, there exists a unique $r^*\in V$ such that $a(r^*,s) = \ell(s) $ for all $s\in V$,
    i.e.,
    \begin{align}\label{viqyps}
        \langle \lambda r^*+Ar^*+v^*, s\rangle_{L^2(\Omega)}=0 \quad \forall s\in V.
    \end{align}
    Set $(\xi,\eta):=(\xi_{r^*},\eta_{r^*})$ and define
    \[
        \pi:=\xi(0)=r^*\quad\text{and}\quad \rho:=v^*+\eta(0)+\lambda\,\xi(0)=v^*+Ar^*+\lambda r^*.
    \]
    Then $\pi\in V$, and hence $\pi=0 \ \text{a.e.\ on } \{|\varphi(0)|<\mu\}$. Moreover, from \eqref{viqyps}, one can conclude that $\rho\equiv 0$ a.e. on $\{|\varphi(0)|>\mu\}$. Indeed, fix $\zeta\in L^2(\Omega)$ with $\zeta=0$ a.e.\ on $\Omega\setminus\{|\varphi(0)|>\mu\}$.
    For sufficiently small $\varepsilon>0$ we have $\pi\pm \varepsilon \zeta\in V$
    (since $V$ imposes no sign restriction on $\{|\varphi(0)|>\mu\}$). Testing with
    $s=\pi\pm \varepsilon \zeta$ yields $\langle \rho,\zeta\rangle_{L^2(\Omega)}=0$ for all such $\zeta$,
    hence $\rho=0$ a.e.\ on $\{|\varphi(0)|>\mu\}$.
\end{proof}

\subsubsection{Polarity and the normal cone}
Fix $(q,u)\in\mathcal C$ with associated state--adjoint pair
$(y,\varphi)\in W(0,T)\times W(0,T)$.

\paragraph{Step 1 (a mixed identity).}
Let $\xi,\eta\in W(0,T)$ and $\pi\in L^2(\Omega)$ satisfy
\begin{align}\label{twopde}
    \left\{
        \begin{aligned}
            \partial_t\xi-\Delta\xi+q(t)\,\xi &= 0 && \text{in }Q,\\
            \xi &= 0 && \text{on }\Sigma,\\
            \xi(0) &= \pi && \text{in }\Omega,
        \end{aligned}
    \right.
    \qquad
    \left\{
        \begin{aligned}
            -\partial_t\eta-\Delta\eta+q(t)\,\eta &= \xi && \text{in }Q,\\
            \eta &= 0 && \text{on }\Sigma,\\
            \eta(T) &= 0 && \text{in }\Omega.
        \end{aligned}
    \right.
\end{align}
For $(h,v)\in T_{\mathcal C}(q,u)$, let $(z,\psi)$ satisfy \eqref{Btclin}, and define
for a.e.\ $t\in(0,T)$
\[
    H(t):=\int_\Omega\bigl(y(t)\,\eta(t)+\varphi(t)\,\xi(t)\bigr)\,dx .
\]
Testing the $z$--equation with $\eta$ and the $\eta$--equation with $z$,
subtracting and using \eqref{eq:ibp-W} with $\eta(T)=0$ and $z(0)=v$, yields
\[
    \int_0^T h(t)\int_\Omega y(t)\,\eta(t)\,dx\,dt
    =
    \int_\Omega v\,\eta(0)\,dx-\int_0^T\!\!\int_\Omega z\,\xi\,dx\,dt.
\]
Similarly, testing the $\psi$--equation with $\xi$ and the $\xi$--equation with $\psi$,
subtracting and using \eqref{eq:ibp-W} with $\psi(T)=0$ and $\xi(0)=\pi$, gives
\[
    \int_0^T h(t)\int_\Omega \varphi(t)\,\xi(t)\,dx\,dt
    =
    \int_0^T\!\!\int_\Omega z\,\xi\,dx\,dt-\int_\Omega \psi(0)\,\pi\,dx.
\]
Adding the last two identities yields
\begin{equation}\label{eq:mixed-id-Boucon-proof}
    \int_0^T H(t)\,h(t)\,dt
    =
    \int_\Omega v\,\eta(0)\,dx-\int_\Omega \psi(0)\,\pi\,dx.
\end{equation}

\paragraph{Step 2 (the first inclusion).}
Assume $(h^*,v^*)\in N_{\mathcal C}(q,u)$. Applying \cref{L1-threshold},
we obtain the existence of $\xi,\eta\in W(0,T)$ and $\pi,\rho\in L^2(\Omega)$ such
that \eqref{twopde} holds,
\begin{equation}\label{eq:vstar-trace-Boucon-proof}
    v^*=\rho-\lambda\,\pi-\eta(0),
\end{equation}
and $\rho$ and $\pi$ satisfy
\[
    \rho=0 \ \text{a.e.\ on } \{|\varphi(0)|>\mu\}\qquad\text{and}\qquad \pi=0 \ \text{a.e.\ on } \{|\varphi(0)|<\mu\}.
\]
Fix $(h,v)\in T_{\mathcal C}(q,u)$ and let $(z,\psi)$ be its associated pair.
By \eqref{eq:mixed-id-Boucon-proof} we have
\[
    \int_0^T H(t)\,h(t)\,dt=\int_\Omega v\,\eta(0)\,dx-\int_\Omega \psi(0)\,\pi\,dx.
\]
Hence, using \eqref{eq:vstar-trace-Boucon-proof},
\begin{align*}
    \int_0^T (h^*(t)-H(t))\,h(t)\,dt
    &=\int_0^T h^*(t)\,h(t)\,dt-\int_\Omega v\,\eta(0)\,dx+\int_\Omega \psi(0)\,\pi\,dx\\
    &=\Bigl(\int_0^T h^*(t)\,h(t)\,dt+\int_\Omega v^*\,v\,dx\Bigr)
    -\int_\Omega \rho\,v\,dx+\int_\Omega (\psi(0)+\lambda v)\,\pi\,dx.
\end{align*}
Since $(h^*,v^*)\in N_{\mathcal C}(q,u)$, the bracketed term is $\le 0$.
Moreover, the trace relation in \eqref{Btclin} yields
\[
    v=0 \ \text{on } \{|\varphi(0)|<\mu\}\qquad\text{and}\qquad \psi(0)+\lambda v=0 \ \text{on } \{|\varphi(0)|>\mu\}.
\]
Hence, using that $\{|\varphi(0)|=\mu\}$ has null measure, we conclude that
\[-
    \int_\Omega \rho\,v\,dx+\int_\Omega (\psi(0)+\lambda v)\,\pi\,dx = 0.
\]
Therefore,
\begin{equation}\label{eq:hstar-K-ineq-Boucon-proof}
    \int_0^T (h^*(t)-H(t))\,h(t)\,dt\le 0 \quad \forall (h,v)\in  T_{\mathcal C}(q,u).
\end{equation}
By \cref{lem:Theta-surj}, for every $h\in L^2(0,T)$ there exists $v\in L^2(\Omega)$
such that $(h,v)\in T_{\mathcal C}(q,u)$. By \eqref{eq:hstar-K-ineq-Boucon-proof},
\[
    \int_0^T (h^*(t)-H(t))\,h(t)\,dt\le 0 \quad \forall h\in L^2(0,T).
\]
Applying the same inequality to $-h$ gives the reverse inequality, thus
$\int_0^T (h^*(t)-H(t))\,h(t)\,dt=0$ for all $h$, and therefore $h^*=H$ in $L^2(0,T)$, i.e.
\[
    h^*(t)=\int_\Omega \bigl(y(t)\,\eta(t)+\varphi(t)\,\xi(t)\bigr)\,dx
    \quad\text{for a.e.\ }t\in(0,T).
\]
Together with \eqref{eq:vstar-trace-Boucon-proof} this yields the desired representation.

\paragraph{Step 3 (the second inclusion).}
Assume there exist $\xi,\eta\in W(0,T)$ and $\pi,\rho\in L^2(\Omega)$ satisfying
\[
    \left\{
        \begin{aligned}
            \partial_t\xi-\Delta\xi+q(t)\,\xi &= 0 && \text{in }Q,\\
            \xi &= 0 && \text{on }\Sigma,\\
            \xi(0) &= \pi && \text{in }\Omega,
        \end{aligned}
    \right.
    \qquad
    \left\{
        \begin{aligned}
            -\partial_t\eta-\Delta\eta+q(t)\,\eta &= \xi && \text{in }Q,\\
            \eta &= 0 && \text{on }\Sigma,\\
            \eta(T) &= 0 && \text{in }\Omega,
        \end{aligned}
    \right.
\]
such that
\[
    v^*=\rho-\lambda\,\pi-\eta(0)
    \quad\text{and}\quad
    h^*(t)=\int_\Omega \bigl(y(t,x)\,\eta(t,x)+\varphi(t,x)\,\xi(t,x)\bigr)\,dx,
\]
where
\[
    \pi=0 \ \text{a.e.\ on }\{|\varphi(0)|<\mu\}\qquad\text{and}\qquad
    \rho=0 \ \text{a.e.\ on }\{|\varphi(0)|>\mu\}.
\]
and the contact set
has Lebesgue measure zero. We show that $(h^*,v^*)\in N_{\mathcal C}(q,u)$.

Let $(h,v)\in T_{\mathcal C}(q,u)$ be arbitrary and let $(z,\psi)\in W(0,T)\times W(0,T)$
be associated with $(h,v)$ through \eqref{Btclin}.
Using $v^*=\rho-\lambda\pi-\eta(0)$ and the mixed identity from Step 1,  we obtain
\begin{equation}\label{eq:step3-key}
    \int_0^T h^*(t)\,h(t)\,dt+\int_\Omega v^*\,v\,dx
    =\int_\Omega\Bigl[\rho\,v-(\psi(0)+\lambda v)\,\pi\Bigr]\,dx.
\end{equation}
We now analyze the integrand pointwise using \eqref{Btclin}.
On $\{|\varphi(0)|<\mu\}$ we have $v=0$ and $\pi=0$, hence the integrand vanishes.
On $\{|\varphi(0)|>\mu\}$ we have $\rho=0$ and $v=-\psi(0)/\lambda$, hence $\psi(0)+\lambda v=0$ and again
the integrand vanishes. On the contact set $\{|\varphi(0)|=\mu\}$ no pointwise cancellation is needed;
since it has null measure and the integrand belongs to $L^1(\Omega)$ (because
$\rho,v,\psi(0),\pi\in L^2(\Omega)$ and $\Omega$ is bounded), we have
\[
    \int_{\{|\varphi(0)|=\mu\}}\Bigl[\rho\,v-(\psi(0)+\lambda v)\,\pi\Bigr]\,dx=0.
\]
Consequently the right-hand side of \eqref{eq:step3-key} equals $0$, and thus
\[
    \int_0^T h^*(t)\,h(t)\,dt+\int_\Omega v^*\,v\,dx=0
    \quad \forall (h,v)\in T_{\mathcal C}(q,u).
\]
Therefore $(h^*,v^*)\in N_{\mathcal C}(q,u)$, which proves the second inclusion.

\subsection{Proof of Corollary \ref{prop:ClarkeNormal}}
Under the assumption $\{|\varphi(0)|=\mu\}$ has null measure, \cref{Ctc} yields
$\hat T_{\mathcal C}(q,u)=T_{\mathcal C}(q,u)$. Therefore $ \hat N_{\mathcal C}(q,u)=N_{\mathcal C}(q,u)$, which proves the claim.

\subsection{Proof of Theorem \ref{thmindass}}
The proof is based on two separate constructions. In the first one, we construct an element of $\mathcal C$ for which \cref{obsass} holds, while the corresponding contact set has positive Lebesgue measure. In the second one, we construct an element of $\mathcal C$ for which the contact set has zero Lebesgue measure, but the operator in \cref{obsass} has a nontrivial kernel.

\subsubsection{A technical lemma}
We will need a preliminary construction to prove that  \cref{obsass} does not, in general, imply nullity of the contact set.

\begin{lemma}\label{lem:yate}
    Let $\mu>0$. Consider the sequence $(e_k)_{k\in\mathbb N}$ of functions $e_k:(0,1)\to\mathbb R$ defined by $e_k(x):=\sqrt{2}\sin(k\pi x)$.
    There exist $\varphi_0\in H^2(0,1)\cap H_0^1(0,1)$ such that
    \[
        0\le \varphi_0(x)\le \mu \quad\forall x\in(0,1),\qquad
        \varphi_0(x)=\mu \quad\forall x\in(\tfrac15,\tfrac25),\qquad \int_0^1 \varphi_0(x)e_k(x)\,dx\neq 0\quad \forall k\in\mathbb N.
    \]
\end{lemma}
\begin{proof}
    Set $I:=(\frac15,\frac25)$ and $J:=(\frac35,\frac45)$. Since the intervals are compactly
    contained in $(0,1)$ and have disjoint closures, we may choose
    $\psi\in C^\infty([0,1])$ such that
    \[
        0\le \psi\le \mu \quad\text{on }[0,1],\qquad
        \psi=\mu \quad\text{on }I,
    \]
    \[
        \psi=\frac{\mu}{2}
        \quad\text{on an open neighbourhood of }\overline J,
        \qquad
        \psi(0)=\psi(1)=0 .
    \]
    It is not hard to see that  $\psi\in H^2(0,1)\cap H_0^1(0,1)$. Set
    \[
        V:=
        \Bigl\{
            \theta\in H^2(0,1)\cap H_0^1(0,1):
            \operatorname{supp}\theta\subseteq\overline J
        \Bigr\}.
    \]
    Note that $V$ is a closed subspace of $H^2(0,1)\cap H_0^1(0,1)$.  For each $k\in\mathbb N$, define $L_k:V\to \mathbb R$ by
    \[
        L_k(\theta):=\int_0^1 \theta(x)e_k(x)\,dx.
    \]
    It is not hard to see that each  functional $L_k$ is linear, continuous and not
    identically zero. Consider the sequence $(a_k)_{k\in\mathbb N}$ given by $a_k:=\int_0^1 \psi(x)e_k(x)\,dx$. For each $k\in\mathbb N$, define
    \[
        A_k:=\{\theta\in V:\ a_k+L_k(\theta)=0\}.
    \]
    Since $A_k$ is a proper
    closed affine hyperplane in $V$, it therefore has empty interior. By Baire's
    theorem, $V\setminus\bigcup_{k=1}^\infty A_k$ is dense in $V$. Let $C>0$ be such that $\|\theta\|_{L^\infty(0,1)}
    \le C\|\theta\|_{H^2(0,1)}$ for all $\theta\in H^2(0,1)$. Due to density,  we can find $\theta\in V\setminus\bigcup_{k=1}^\infty A_k$ satisfying $\|\theta\|_{H^2(0,1)}<\frac{\mu}{4C}$. Then  $\|\theta\|_{L^\infty(0,1)}<\mu/4$. Define $\varphi_0:=\psi+\theta$. Since $\operatorname{supp}\theta\subseteq\overline J$ and
    $\overline I\cap\overline J=\varnothing$, we have
    \[
        \varphi_0=\psi=\mu
        \quad\text{on }I.
    \]
    Moreover, on $\operatorname{supp}\theta$ we have $\psi=\mu/2$, and therefore
    \[
        \frac{\mu}{4}<\varphi_0<\frac{3\mu}{4}
        \quad\text{on }\operatorname{supp}\theta .
    \]
    Outside $\operatorname{supp}\theta$, one has $\varphi_0=\psi$. Since
    $0\le\psi\le\mu$, this gives $0\le \varphi_0\le\mu$ on $(0,1)$. Also, $\varphi_0\in H^2(0,1)\cap H_0^1(0,1)$. Finally, for every $k\in\mathbb N$,
    \[
        \int_0^1 \varphi_0(x)e_k(x)\,dx
        =
        \int_0^1 \psi(x)e_k(x)\,dx
        +
        \int_0^1 \theta(x)e_k(x)\,dx
        =
        a_k+L_k(\theta).
    \]
    Since $\theta\notin A_k$ for every $k\in\mathbb N$, the last quantity is
    nonzero for every $k\in\mathbb N$. This proves the lemma.
\end{proof}

\subsubsection{Assumption \ref{obsass} does not imply nullity of the contact set}
Set \(\Omega=(0,1)\) and $T=1$. Consider the sequence $(e_k)_{k\in\mathbb N}$ of functions $e_k:(0,1)\to\mathbb R$ defined by $e_k(x):=\sqrt{2}\sin(k\pi x)$. This is the eigenbasis of the Dirichlet Laplacian on $\Omega$, hence  \((e_k)_{k\in\mathbb N}\) is an orthonormal basis of \(L^2(0,1)\) and
\[
    -e_k''=k^2\pi^2 e_k,\qquad e_k(0)=e_k(1)=0.
\]
Let \(\varphi_0\in H^2(0,1)\cap H_0^1(0,1)\) be given by \cref{lem:yate}. In
particular,
\[
    0\le \varphi_0\le \mu \ \text{on }[0,1],\quad
    \varphi_0=\mu \ \text{on }\Bigl(\tfrac15,\tfrac25\Bigr),\quad
    c_k:=\int_0^1 \varphi_0(x)e_k(x)\,dx\neq 0\ \forall k\in\mathbb{N}.
\]
We set $ q\equiv0$, $u\equiv0$, $y\equiv0$, and define $\varphi\in W(0,T)$ by $\varphi(t,x):=\Bigl(1-\frac{t}{T}\Bigr)\varphi_0(x)$. Define $y_d\in L^2(Q)$ by
\[
    y_d(t,x):=\partial_t\varphi(t,x)+\partial_{xx}\varphi(t,x)
    =
    -\frac1T\varphi_0(x)+\Bigl(1-\frac{t}{T}\Bigr)\varphi_0''(x).
\]
It is not hard to check that $(q,u)\in L^2(0,T)\times L^2(\Omega)$ belongs to $\mathcal C$. We now show that \cref{obsass} holds. Let
\(\pi\in L^2(0,1)\) be arbitrary satisfying $\mathcal O(\pi)=0$. Writing  $\pi_k:=\int_0^1 \pi(x)e_k(x)\,dx$ for $k\in\mathbb N$,  we can write the solution of the forward linearized equation as
\[
    \xi(t)=\sum_{k=1}^\infty \pi_k e^{-k^2\pi^2 t}e_k
    \quad\text{for a.e $t\in(0,T)$}.
\]
Because \(y=0\), we have that
\[
    0=\mathcal O(\pi)(t)
    =
    \int_0^1 \varphi(t,x)\xi(t,x)\,dx
    =
    \Bigl(1-\frac{t}{T}\Bigr)
    \sum_{k=1}^\infty c_k\pi_k e^{-k^2\pi^2 t}
\]
for a.e. \(t\in(0,T)\). Since
\(1-t/T\neq0\) for \(t\in(0,T)\), it follows that
\[
    \sum_{k=1}^\infty c_k\pi_k e^{-k^2\pi^2 t}=0
    \quad\text{for a.e. }t\in(0,T).
\]
It follows that \(c_k\pi_k=0\) for every \(k\in\mathbb N\); whence we conclude that $\pi=0$. Thus \(\mathcal O\) is injective, and \cref{obsass} holds. At the same
time, the contact set has positive Lebesgue measure. This proves that
\cref{obsass} does not imply that the contact set has Lebesgue measure zero.

\subsubsection{The nullity of the contact set does not imply Assumption \ref{obsass}.}
Set \(\Omega=(0,1)\) and $T=1$. Consider again the sequence $(e_k)_{k\in\mathbb N}$ of functions $e_k:(0,1)\to\mathbb R$ defined by $e_k(x):=\sqrt{2}\sin(k\pi x)$. Choose \(c\neq0\) such that
\[
    |c|\,\|e_1\|_{L^\infty(0,1)}<\mu .
\]
We set $ q\equiv0$, $u\equiv0$, $y\equiv0$, and define $\varphi\in W(0,T)$ by $\varphi(t,x):=c(1-t)e^{-\pi^2t}e_1(x)$. Define also $y_d\in L^2(Q)$ by $y_d:=\partial_t\varphi+\partial_{xx}\varphi$. It is not hard to check that $(q,u)\in L^2(0,T)\times L^2(\Omega)$ belongs to $\mathcal C$ and that $\{|\varphi(0)|=\mu\}=\emptyset$. We now show that \cref{obsass} fails. The solution of the forward linearized
equation with initial condition $e_2$ is given by
\[
    \xi(t)=\sum_{k\ge1}\langle e_2,e_k\rangle_{L^2(0,1)} e^{-k^2\pi^2t}e_k = e^{-4\pi^2t} e_2.
\]
Since \(y=0\), we have
\[
    \mathcal O(\pi)(t)
    =
    \int_0^1 \varphi(t,x)\xi(t,x)\,dx
    =
    c(1-t)e^{-\pi^2t}\int_0^1 e_1(x)\xi(t,x)\,dx
    =  c(1-t)e^{-\pi^2t}e^{-4\pi^2t}\langle e_1,e_2 \rangle_{L^2(0,1)} = 0.
\]
Thus every \(e_2\) belongs to \(\ker\mathcal O\), therefore the operator \(\mathcal O\) is not injective.
Since the contact set has zero Lebesgue measure (it is empty), while
\cref{obsass} fails, this completes the proof.

\subsection{Proof of Theorems \ref{thm:nec-F} and \ref{thm:nec-FN}}
Consider  the function $\Psi:L^2(0,T)\times L^2(\Omega)\to\mathbb R$ given by $\Psi(q,u):=\frac12\|u-u_d\|_{L^2(\Omega)}^2+\frac{\alpha}{2}\|q\|_{L^2(0,T)}^2$. Then $\Psi$ is $C^1$ on $L^2(0,T)\times L^2(\Omega)$ and $\nabla \Psi(q,u)=(\alpha q,\ u-u_d)$. Now, let $\bar q\in L^2(0,T)$ be a local minimizer of the bilevel problem \eqref{cost}--\eqref{Bip}, and set
$\bar u:=u_{\bar q}$. Since \(\mathcal C=\operatorname{gph}(q\mapsto u_q)\), the local minimality of
\(\bar q\) for the bilevel problem implies that \((\bar q,\bar u)\) is a local
minimizer of $\Psi$ over $\mathcal C$. The first order necessary condition then reads
\[
    \big\langle(\alpha \bar q,\ \bar u-u_d), (h,v) \big\rangle_{L^2(0,T)\times L^2(\Omega)}
    =
    \langle \alpha \bar q,h\rangle_{L^2(0,T)}
    +\langle \bar u-u_d,v\rangle_{L^2(\Omega)}
    \ge0
    \quad \forall(h,v)\in T_{\mathcal C}(\bar q,\bar u).
\]
By definition of the normal cone, this can be written as $(-\alpha \bar q,\ u_d-\bar u) \in N_{\mathcal C}(\bar q,\bar u)$. The result follows then from \cref{Boucon}.

The proof of Theorem \ref{thm:nec-FN} follows by the same argument, replacing
\cref{Boucon} with \cref{Boucon2}.

\section{Numerical example}\label{sec:numerics}

In this section, we illustrate the derived optimality conditions as well as the applicability of the bilevel approach to experiment design for inverse problems with a numerical example, where we compute stationary controls using a projected gradient method.

\subsection{Projected reduced gradient descent with backtracking}\label{sec:gradient}

The algorithmic approach considered in this work is directly motivated by the optimality systems in \cref{thm:nec-F,thm:nec-FN}: given a candidate control $q$, one evaluates the associated state, adjoint, and linearized variables, and then updates
$q$ using the gradient relation in \eqref{opt:upper}. \Cref{alg:pgd} describes this process.

\begin{algorithm}[ht]
    \caption{Projected reduced gradient method with Armijo line search for problem \eqref{Bip}.}\label{alg:pgd}
    \label{alg:bilevel:pgd}
    \begin{algorithmic}[1]
        \Require Upper level objective  functional $\mathcal F:L^2(0,T)\to \mathbb R$, feasible set $Q \subseteq L^2(0,T)$, initial iterate $q^0 \in q$, Armijo parameters $c,\beta \in (0,1)$, initial trial step size $\tau_0>0$
        \For{$k\in\N$}
        \State Solve the lower-level system \eqref{opt:lower} at $q^k$ and obtain $(y^k,\varphi^k,u^k)$
        \State Define the active sets
        \(
        A_{\le}^k \defeq \{\abs{\varphi^k(0)} \le \mu\}
        \)
        and
        \(
        A_>^k \defeq \{\abs{\varphi^k(0)} > \mu\}
        \)
        \State Solve the upper-level system \eqref{opt:upper} for $(\xi^k,\eta^k)$ with $\lambda \pi^k = u^k - u_d - \eta^k(0)$ on $A_>^k$ and $\pi^k = 0$ on $A_\le^k$.
        \State Compute the gradient
        \[
            g^k(t) \defeq \alpha q^k(t)
            + \int_\Omega \bigl(y^k\eta^k + \varphi^k\xi^k\bigr)(t,x)\,dx
        \]
        \State Choose $\tau_k = \tau_0\beta^{m_k}$ with the smallest $m_k\in\N$ such that
        \[
            \mathcal F\bigl(\proj_Q(q^k-\tau_k g^k)\bigr)
            \le
            \mathcal F(q^k) - c\tau_k \norm{g^k}_{L^2(0,T)}^2.
        \]
        \State Set
        \[
            q^{k+1} \defeq \proj_Q(q^k-\tau_k g^k)
        \]
        \EndFor
    \end{algorithmic}
\end{algorithm}

Although the problem is nonsmooth due to the presence of the proximal mapping and the associated complementarity conditions, the structure of the optimality system indicates that, locally, the problem behaves like a smooth one whenever the active and inactive remain unchanged and the contact sets is empty. In such regime, the standard gradient-based methods are therefore expected to yield descent and convergence behavior comparable to the smooth setting.

\subsection{Setting}

We test the proposed framework in a setting where the goal is to learn the control $q$ that improves the reconstruction of parameter $u$ across multiple instances of the lower level problem. The control is treated as a model parameter and calibrated from a collection of datasets. This setup allows us to assess how effectively the bilevel formulation learns a control that consistently enhances reconstruction quality.

We consider $N$ datasets $\{u_d^k\}_{k=1}^N$, each associated with noisy observations of the corresponding state. For each $k$, the lower-level problem reads
\begin{equation}\label{cost:mod}
    \min_{u^k \in L^2(\Omega)}
    \left\{
        \frac12 \int_{T_0}^{T}\int_{\Omega} |y^k - e^k|^{2}\,dx\,dt
        +\frac{\lambda}{2}\int_{\Omega}|u^k|^{2}\,dx
        +\mu\int_{\Omega}|u^k|\,dx
    \right\},
\end{equation}
subject to
\begin{equation}\label{system:mod}
    \left\{
        \begin{aligned}
            \partial_t y - \Delta y + q(t)\,y &= 0 && \text{in } Q,\\
            y &= 0 && \text{on } \Sigma,\\
            y(0) &= u^k - u_d^k && \text{in } \Omega.
        \end{aligned}
    \right.
\end{equation}
Since $y(0) = u^k - u_d^k$, due to the linearity of the state equation the corresponding state can be written as $y^k = y_u^k - y_{u_d}^k$ where $y_u^k$ and $y_{u_d}^k$ denote the solutions of \eqref{system:mod} with initial data $u_k$ and $u^k_d$, respectively. Consequently, the data misfit term can be interpreted as $y^k - e^k= y_u^k - ( y_{u_d}^k +  e^k)$, so that the lower-level problem seeks a regularized reconstruction $u^k$ whose forward response matches noisy observations generated under the model parameter $q$.

The upper-level problem determines a control $q$ that performs well across all datasets:
\begin{equation}
    \mathcal F(q)
    \defeq
    \sum_{k=1}^N
    \frac{1}{2c_k}
    \int_\Omega |u_q^k(x) - u_d^k(x)|^2 \, dx
    + \frac{\alpha}{2} \int_0^T |q(t)|^2 \, dt
    + \iota_{[q_m, q_M]}(q),
\end{equation}
where $c_k \defeq \abs{\Omega}^{-1} \int_\Omega |u_d^k(x)|^2 \, dx$, so that each term corresponds to an $\abs{\Omega}^{-1}$ weighted relative reconstruction error and the additional control constraints. The final term is a pointwise almost everywhere constraint on the control that serves to exclude singular ``unphysical'' controls in this specific example.

In the experiments, we take $T_0 = 2 \cdot 10^{-2}$, $T = 10^{-1}$, $\Omega = [0,1]^2$, $\lambda = 10^{-5}$, $\mu = \tfrac{1}{4}\cdot 10^{-4}$, $N = 8$, $\alpha = 10^{-7}$, and $[q_m,q_M] = [-350,350]$.
We use $8$ datasets for training and reserve $4$ additional datasets for validation (see the first columns of \cref{fig:training-1-4,fig:training-5-8,fig:valid-1-4}).

The optimality systems are discretized using a Galerkin finite element method in spatial domain. The lower-level problem is fully mass-lumped to simplify computations, while the upper-level problem retains the mass matrix. Time discretization is performed using forward differences, and $q$ is represented as a piecewise constant function.

Let $J\in\mathbb N$ denote the number of time steps and $J_0\in\{0,\dots,J-1\}$ the index corresponding to $T_0$. We define
$
\tau \defeq \frac{T}{J},
$
$
t_j \defeq j\tau,
$
$
j=0,\dots,J,
$
so that
$
T_0 = J_0\tau,
$
$
T = J\tau.
$
The observation interval \([T_0,T]\) is thus represented by the index range \(j=J_0,\dots,J-1\) and the control \(q\) is constant on each interval \([t_j,t_{j+1})\).

Writing $R$ for the stiffness matrix, $M$ for the mass matrix, $W$ for the diagonal mass-lumped matrix, and defining
$
A_j
\defeq
M + \tau \bigl(R + q_j M \bigr),
$
the discrete lower level optimality system reads
\begin{align*}
    &
    \left\{\begin{aligned}
            y^k_0
            &
            = u^k- u^k_d,
            \\
            A_jy^k_{j+1}
            &
            = M y^k_j,
            \quad
            j=0,\ldots,J-1,
    \end{aligned}\right.
    \\
    &
    \left\{\begin{aligned}
            \varphi^k_{J}
            &
            = 0,
            \\
            A_j^\top \varphi^k_j
            &
            = M \varphi^k_{j+1}
            +
            \tau W\bigl(y^k_{j+1} - e^k_{j+1}\bigr),
            \quad j=J_0,\ldots,J-1,\\
            A_j^\top \varphi^k_j
            &
            = M \varphi^k_{j+1},
            \quad j=0,\ldots,J_0-1,
    \end{aligned}\right.
\end{align*}
for $u^k$ satisfying $0 \in \varphi^k_0 + \lambda u^k + \mu \partial \abs{u^k}$.
The discrete upper level system reads
\begin{align*}
    &
    \left\{\begin{aligned}
            \xi^k_0
            &
            = \pi^k,\\
            A_j\xi^k_{j+1}
            &
            = M\xi^k_j,
            \quad j=0,\ldots,J-1,
    \end{aligned}\right.
    \\
    &
    \left\{\begin{aligned}
            \eta^k_J &= 0,
            \\
            A_j^\top \eta^k_j
            &
            =
            M\eta^k_{j+1} + \tau\,M\xi^k_{j+1},
            \quad j=0,\ldots,J-1,
    \end{aligned}\right.
\end{align*}
where $\pi^k$ satisfying $\lambda \pi^k + \eta^k_0 = u^k-u^k_d$ pointwise at all nodes $i$ with $\abs{[u^k]_i} > \epsilon$ for some $\epsilon > 0$,
and $\pi^k = 0$ for all remaining nodes. Here, the condition for the active set is relaxed to account for numerical inaccuracies; for $\epsilon = 0$, the condition $\abs{[u^k]_i} > \epsilon$ is equivalent to $\abs{[\varphi^k_0]_i} > \mu$ at optimality.
The discretized gradient condition reads
\[
    g_j
    =
    \alpha\tau\,q_j
    +
    \tau
    \sum_{k=1}^N
    \bigl(
        (\eta^k_j)^\top M y^k_{j+1}
        +
        (\varphi^k_j)^\top M\xi^k_{j+1}
    \bigr) \in - N_{[q_m, q_M]}(q_j),
\]
and $q_j \in [q_m, q_M]$, for $j=0,\ldots,J-1$. Here $N_{[q_m, q_M]}(q_j)$ is the normal cone to the interval $[q_m, q_M]$ at $q_j$.

We employ a mesh with $2500$ nodes and $50$ time steps. The noise $e^k$ is sampled independently at each time step from a Gaussian distribution with zero mean and with standard deviation $0.01$.

The lower-level problem is solved using Fast Iterative Soft-Thresholding (FISTA) \cite{BeckTeboulle} with step size $0.98/L_y$, where $L_y$ is an estimate of the Lipschitz constant of lower-level data term obtained via power iteration. The FISTA acceleration is reset every $\sqrt{L_y/\mu}$ iterations. The stopping criterion is $
\norm{u^k - \prox_{(\mu/\lambda)\abs{\cdot}}\left(-\frac1\lambda \varphi^k_0\right)}_2 \le 10^{-7}$, with a maximum of $25{,}000$ iterations.

We use tolerance $\epsilon = 10^{-10}$ for the active set identification. The upper-level adjoint problem is solved using the conjugate gradient method to accuracy $10^{-8}$, with at most $25{,}000$ iterations. The outer optimization employs Armijo backtracking with parameters $c = 10^{-2}$, $\beta = 10^{-1}$, and $\tau_0 = 25/h$, where $h$ is the scale of the discretization.

All experiments were implemented in Python using \texttt{NumPy} and \texttt{SciPy}, with meshes generated via \texttt{FEniCSx}. The source code is available on Zenodo \cite{dominguez_corella_2026_20364570}.

\subsection{Results}

\newcommand{\ResultPath}{results}
\newcommand{\CachePath}{results/figures}
\newcommand{\NumTargets}{12}

\newcommand{\GlobalPlotWidth}{0.82\textwidth}

\newcommand{\ReconImageWidth}{0.30\textwidth}
\newcommand{\ReconColorbarWidth}{0.045\textwidth}

\newcommand{\TargetDir}[1]{\ResultPath/target_i#1}

\newcommand{\IncludeGlobalQ}{%
    \includegraphics[width=\GlobalPlotWidth]{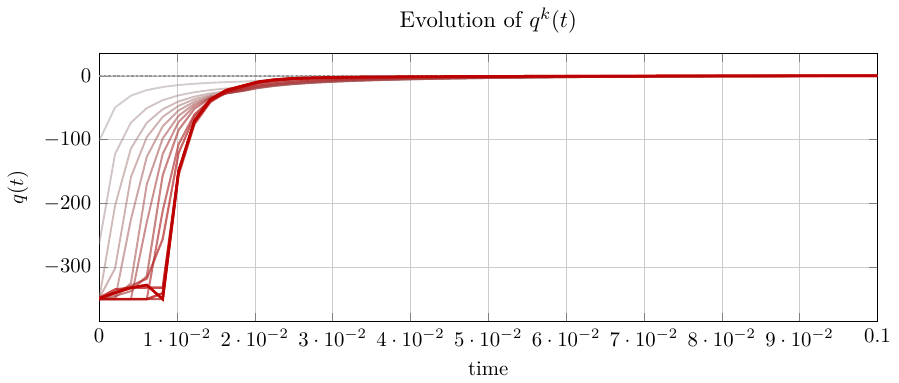}%
}

\newcommand{\IncludeGlobalObjective}{%
    \includegraphics[width=\GlobalPlotWidth]{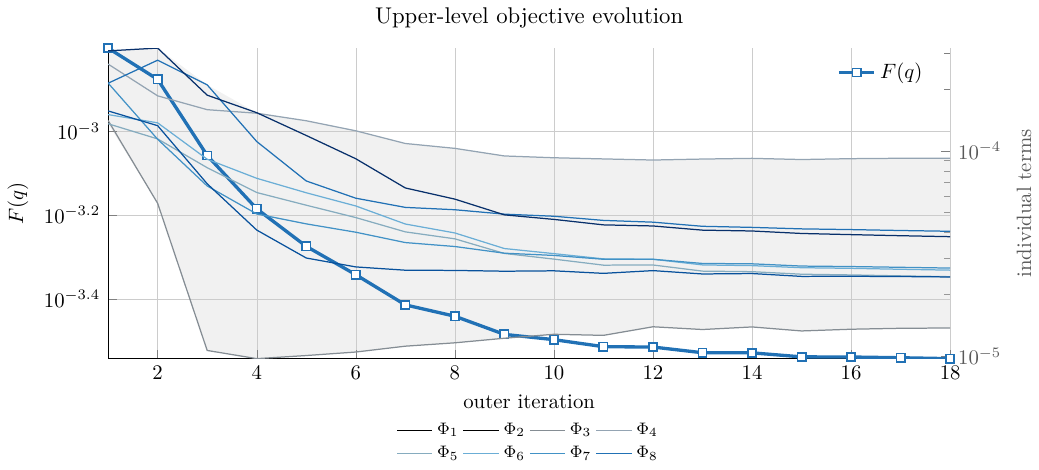}%
}

\pgfplotstableread[col sep=space]{\detokenize{results/train_relative_errors.txt}}\TrainTable
\pgfplotstableread[col sep=space]{\detokenize{results/validation_relative_errors.txt}}\ValTable

\newcommand{\TrainStr}[2]{%
    \pgfplotstablegetelem{#1}{#2}\of{\TrainTable}%
    \pgfplotsretval%
}
\newcommand{\ValStr}[2]{%
    \pgfplotstablegetelem{#1}{#2}\of{\ValTable}%
    \pgfplotsretval%
}
\newcommand{\TrainNum}[2]{%
    \pgfplotstablegetelem{#1}{#2}\of{\TrainTable}%
    \pgfmathprintnumber[fixed,precision=2,fixed zerofill]{\pgfplotsretval}%
}
\newcommand{\ValNum}[2]{%
    \pgfplotstablegetelem{#1}{#2}\of{\ValTable}%
    \pgfmathprintnumber[fixed,precision=2,fixed zerofill]{\pgfplotsretval}%
}

\newcommand{\ErrorTable}[1]{%
    \pgfplotstabletypeset[
    col sep=space,
    header=false,
    display columns/0/.style={
        string type,
        column name={Dataset},
        column type={r},
    },
    display columns/1/.style={
        column name={Initial q},
        fixed,
        precision=2,
        column type={r},
    },
    display columns/2/.style={
        column name={Final q},
        fixed,
        precision=2,
        column type={r},
    },
    display columns/3/.style={
        column name={Relative},
        fixed,
        precision=2,
        column type={r},
    },
    every head row/.style={
        before row=\toprule,
        after row=\midrule,
    },
    every last row/.style={
        after row=\bottomrule,
    }
    ]{#1}%
}

\Cref{fig:obj-q-sidebyside} shows the evolution of the objective value and the recovered control. The method reaches approximate stationarity after $18$ outer iterations. At the last iterate, we obtained $\norm{g}_{L^2(0,T)} = 2.318 \cdot 10^{-7}$.

\Cref{fig:training-1-4,fig:training-5-8,fig:valid-1-4} collect the reconstruction results. Each row shows the true state, the reconstruction obtained with the initial control $q_{q=0} \equiv 0$, and the reconstruction obtained with the optimized control $q_{\text{opt}}$. Visually, the optimized control improves the reconstruction quality in both the training dataset (\cref{fig:training-1-4,fig:training-5-8}) and in the validation dataset (\cref{fig:valid-1-4}). \Cref{tab:rel-all-low} reports the relative errors for both datasets, confirming the improvement we observed visually.

\newcommand{\TikzCachePath}{results/figures}
\newcommand{\NumTargetsTrain}{8}
\newcommand{\NumTargetsVal}{4}

\newcommand{\FieldPdfWidth}{0.27\textwidth}
\newcommand{\ColorbarPdfWidth}{0.075\textwidth}
\newcommand{\GlobalPdfWidth}{0.88\textwidth}
\newcommand{\TrainTargetDir}[1]{\ResultPath/train_compare/target_i#1}
\newcommand{\ValTargetDir}[1]{\ResultPath/validation/target_i#1}

\newcommand{\IncludeCachedPdf}[2][]{%
    \IfFileExists{\TikzCachePath/#2.pdf}{%
        \includegraphics[#1]{\TikzCachePath/#2.pdf}%
    }{%
        \fbox{\parbox[c][3cm][c]{0.8\linewidth}{\centering Missing file:\\ \texttt{\TikzCachePath/#2.pdf}}}%
    }%
}

\begin{figure}[t]
    \centering
    \begin{minipage}[t]{0.46\textwidth}
        \centering
        \IncludeCachedPdf[width=\textwidth]{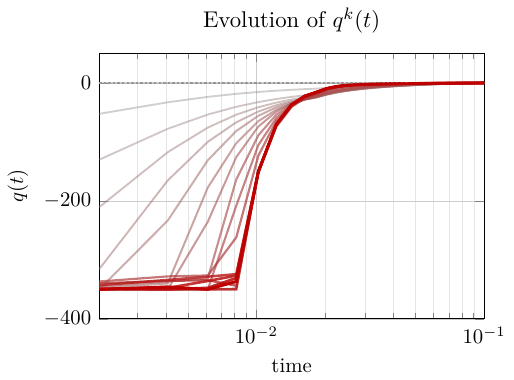}
    \end{minipage}
    \hfill
    \begin{minipage}[t]{0.534\textwidth}
        \centering
        \IncludeCachedPdf[width=\textwidth]{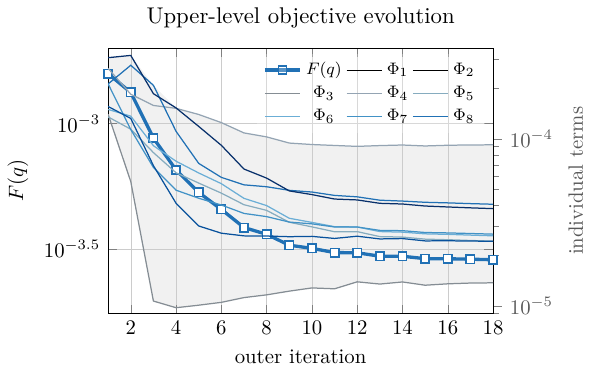}
    \end{minipage}
    \caption{Left: evolution of the upper-level iterates $q^k$. The color transitions from gray at the initial iterate to red at the final iterate. Right: convergence history of the full upper-level objective and of the individual per-dataset contributions. The shaded band indicates the minimum and maximum component values at each outer iteration.}
    \label{fig:obj-q-sidebyside}
\end{figure}

\begin{table}[b]
    \centering
    \caption{Baseline vs. learned relative reconstruction error for training and validation datasets.}
    \label{tab:rel-all-low}
    \begin{tabular}{cccccccc}
        \toprule
        \multicolumn{4}{c}{Training} & \multicolumn{4}{c}{Validation} \\
        \cmidrule(lr){1-4}\cmidrule(lr){5-8}
        Dataset & Initial $q$ & Final $q$ & Initial $q$/Final $q$ & Dataset & Initial $q$ & Final $q$ & Initial $q$/Final $q$ \\
        \midrule
        \liningnums{1} & \TrainNum{0}{1} & \TrainNum{0}{2} & \TrainNum{0}{3}
        & \liningnums{1} & \ValNum{0}{1} & \ValNum{0}{2} & \ValNum{0}{3} \\

        \liningnums{2} & \TrainNum{1}{1} & \TrainNum{1}{2} & \TrainNum{1}{3}
        & \liningnums{2} & \ValNum{1}{1} & \ValNum{1}{2} & \ValNum{1}{3} \\

        \liningnums{3} & \TrainNum{2}{1} & \TrainNum{2}{2} & \TrainNum{2}{3}
        & \liningnums{3} & \ValNum{2}{1} & \ValNum{2}{2} & \ValNum{2}{3} \\

        \liningnums{4} & \TrainNum{3}{1} & \TrainNum{3}{2} & \TrainNum{3}{3}
        & \liningnums{4} & \ValNum{3}{1} & \ValNum{3}{2} & \ValNum{3}{3} \\

        \liningnums{5} & \TrainNum{4}{1} & \TrainNum{4}{2} & \TrainNum{4}{3}
        &   &               &               &               \\

        \liningnums{6} & \TrainNum{5}{1} & \TrainNum{5}{2} & \TrainNum{5}{3}
        &   &               &               &               \\

        \liningnums{7} & \TrainNum{6}{1} & \TrainNum{6}{2} & \TrainNum{6}{3}
        &   &               &               &               \\

        \liningnums{8} & \TrainNum{7}{1} & \TrainNum{7}{2} & \TrainNum{7}{3}
        &   &               &               &               \\
        \bottomrule
    \end{tabular}
\end{table}

\begin{figure}[htp]
    \foreach \ii in {0,...,3} {

        \ifnum\ii<10
            \edef\TargetIndex{0\ii}
        \else
            \edef\TargetIndex{\ii}
        \fi

        \centering
        \makebox[\textwidth][c]{%
            \IncludeCachedPdf[width=\FieldPdfWidth]{train-target-\TargetIndex-utrue}%
            \hspace{0.01\textwidth}%
            \IncludeCachedPdf[width=\FieldPdfWidth]{train-target-\TargetIndex-uinit}%
            \hspace{0.01\textwidth}%
            \IncludeCachedPdf[width=\FieldPdfWidth]{train-target-\TargetIndex-ufinal}%
            \hspace{0.005\textwidth}%
            \IncludeCachedPdf[width=\ColorbarPdfWidth]{train-target-\TargetIndex-colorbar}%
        }
    }
    \caption{Training dataset from set 1 to 4: ground truth, baseline reconstruction, and reconstruction with optimized control.}\label{fig:training-1-4}
\end{figure}

\begin{figure}[htp]
    \foreach \ii in {4,...,\numexpr\NumTargetsTrain-1\relax} {

        \ifnum\ii<10
            \edef\TargetIndex{0\ii}
        \else
            \edef\TargetIndex{\ii}
        \fi

        \centering
        \makebox[\textwidth][c]{%
            \IncludeCachedPdf[width=\FieldPdfWidth]{train-target-\TargetIndex-utrue}%
            \hspace{0.01\textwidth}%
            \IncludeCachedPdf[width=\FieldPdfWidth]{train-target-\TargetIndex-uinit}%
            \hspace{0.01\textwidth}%
            \IncludeCachedPdf[width=\FieldPdfWidth]{train-target-\TargetIndex-ufinal}%
            \hspace{0.005\textwidth}%
            \IncludeCachedPdf[width=\ColorbarPdfWidth]{train-target-\TargetIndex-colorbar}%
        }
    }
    \caption{Training dataset from set 5 to 8: ground truth, baseline reconstruction, and reconstruction with optimized control.}\label{fig:training-5-8}
\end{figure}

\begin{figure}[htp]
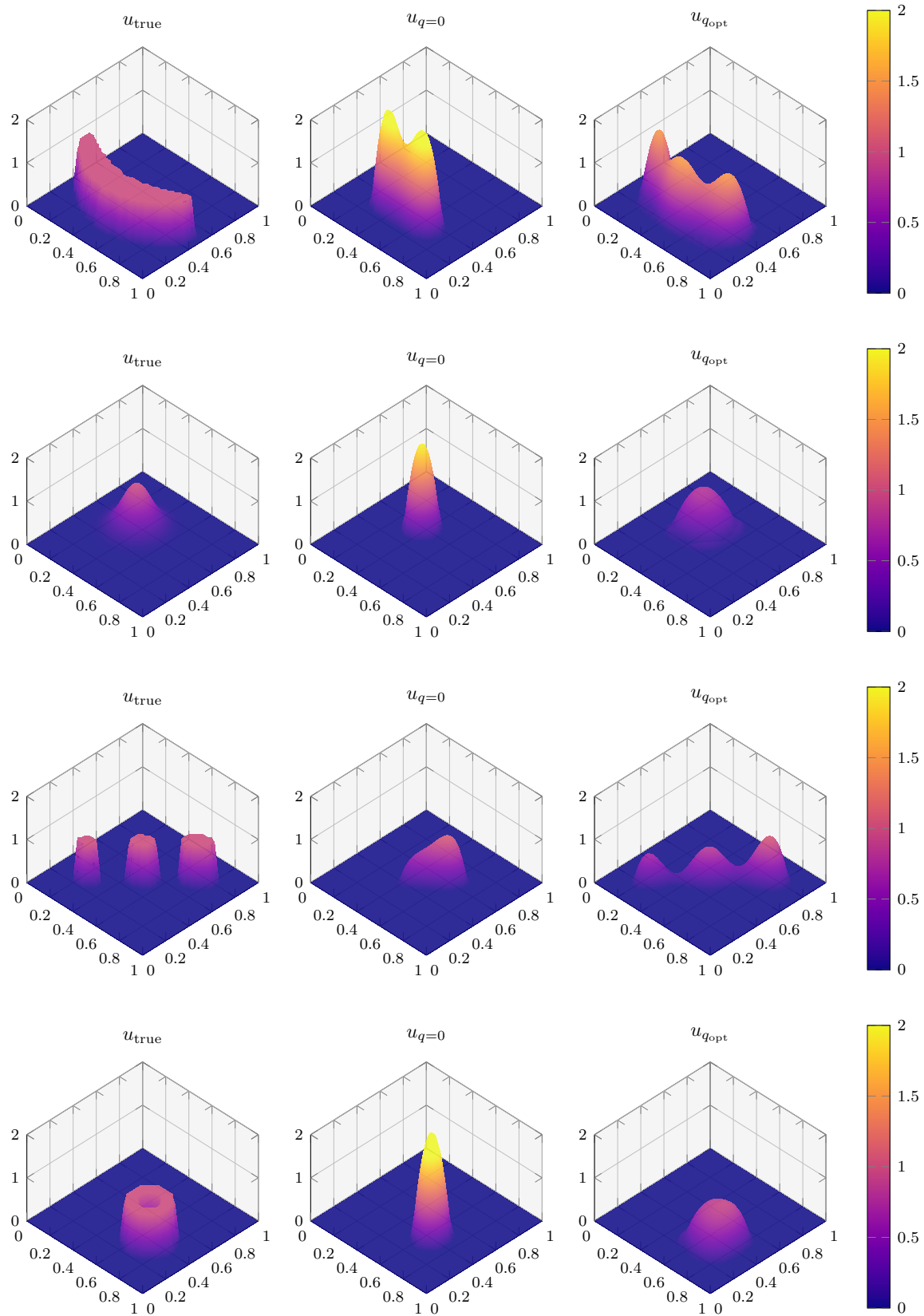

    \foreach \ii in {0,...,\numexpr\NumTargetsVal-1\relax} {

        \ifnum\ii<10
            \edef\TargetIndex{0\ii}
        \else
            \edef\TargetIndex{\ii}
        \fi

        \centering
        \makebox[\textwidth][c]{%
            \IncludeCachedPdf[width=\FieldPdfWidth]{val-target-\TargetIndex-utrue}%
            \hspace{0.01\textwidth}%
            \IncludeCachedPdf[width=\FieldPdfWidth]{val-target-\TargetIndex-uinit}%
            \hspace{0.01\textwidth}%
            \IncludeCachedPdf[width=\FieldPdfWidth]{val-target-\TargetIndex-ufinal}%
            \hspace{0.005\textwidth}%
            \IncludeCachedPdf[width=\ColorbarPdfWidth]{val-target-\TargetIndex-colorbar}%
        }
    }
    \caption{Validation datasets from 1 to 4: ground truth, baseline reconstruction, and reconstruction with optimized control.}\label{fig:valid-1-4}
\end{figure}

\section{Conclusion}

We have formulated optimal experiment design for sparse initial-value recovery in a bilinear parabolic partial differential equation as a nonsmooth bilevel optimal control problem. By lifting the lower-level problem to an upper-level constraint set, we have derived necessary optimality conditions in qualified form under a novel constraint qualification that amounts to an observability condition for a coupled adjoint state system. Crucially, this condition neither requires nor implies that the nonsmoothness in the lower-level problem is only active on a set of Lebesgue measure zero. A numerical example using a reduced gradient method derived from this approach illustrates the benefits of the approach for a simple model problem.

This work can be extended in a number of directions. First, it is relatively straightforward to apply our approach to other state equations including the Bloch--Torrey equations, which would have applications to optimal pulse sequence design in MRI \cite{Aigner:2015,Aigner:2018}. Of interest would also be to consider other (space-dependent) parameters in the state equation (right-hand sides, boundary conditions, coefficients), which however require significant adaptations to the observability condition. While our results generalize in a direct way to other nonsmooth lower-level regularization terms that act pointwise (e.g., to promote discrete-valued reconstructions \cite{Clason:2021}), treating the popular total variation regularization \cite{RudinOsherFatemi1992} is significantly more challenging since the soft-thresholding relation that enters the multiplier system has to be replaced by a relation involving (distributional) differential operators, and some of the multipliers will only be proper measures.
Finally, follow-up work will treat in more detail the numerical solution of the bilevel optimal control problem and study the solvability of the optimality conditions via semismooth Newton methods \cite{Ulbrich:2011}.

\section*{AI Declaration}

ChatGPT 5.2--5.4 was used for editing, checking for grammar and typos, and formatting of the text.
It was also used for debugging and improving the plotting and output of the code, which was further cleaned up with ChatGPT 5.5 for upload to Zenodo. The codes have been audited manually.

\bibliography{citations}

\end{document}